\documentclass[a4paper]{article}
\usepackage[utf8]{inputenc}
\usepackage{amsmath,amsthm}
\usepackage{amsfonts,amssymb}
\usepackage[normalem]{ulem}
\usepackage{totcount}
\usepackage{graphicx}
\usepackage[all]{xy}

\usepackage[colorlinks=true,linktocpage=true,linkcolor=black,citecolor=black]{hyperref}
\usepackage[capitalise]{cleveref} 

\usepackage{enumitem}
\setenumerate[1]{label=(\arabic*)} 
\setlist[itemize]{noitemsep}

\usepackage{footmisc} 

\usepackage{mathtools}

\usepackage{tikz-cd} 
\usetikzlibrary{calc}
\tikzset{curve/.style={settings={#1},to path={(\tikztostart)
   .. controls ($(\tikztostart)!\pv{pos}!(\tikztotarget)!\pv{height}!270:(\tikztotarget)$)
   and ($(\tikztostart)!1-\pv{pos}!(\tikztotarget)!\pv{height}!270:(\tikztotarget)$)
   .. (\tikztotarget)\tikztonodes}},
   settings/.code={\tikzset{quiver/.cd,#1}
       \def\pv##1{\pgfkeysvalueof{/tikz/quiver/##1}}},
   quiver/.cd,pos/.initial=0.35,height/.initial=0}
\usepackage{tikz}
\usetikzlibrary{calc}
\usetikzlibrary{nfold}

 \DeclareMathSymbol{:}{\mathpunct}{operators}{"3A}

 \newcommand{\uvar}{\mathord{\relbar}}

  \newcommand{\cto}{\rightarrowtail}
  \tikzset{cto/.style={>->}}
  \newcommand{\fto}{\twoheadrightarrow}
  \tikzset{fto/.style={->>}}

\newcommand{\Ib}{\mathbb{I}}

\newcommand{\Db}{\mathbb{D}}

\newcommand{\Gb}{\mathbb{G}}

\newcommand{\Lb}{\mathbb{L}}

\newcommand{\Nb}{\mathbb{N}}

\newcommand{\Ocal}{\mathcal{O}}

\newcommand{\Dcal}{\mathcal{D}}

\newcommand{\Wcal}{\mathcal{W}}

\newcommand{\Ccal}{\mathcal{C}}
\newcommand{\Vcal}{\mathcal{V}}

\newcommand{\Ncal}{\mathcal{N}}

\swapnumbers

\theoremstyle{plain}
\newtheorem{theorem}{Theorem}[section]
\newtheorem{lemma}[theorem]{Lemma}
\newtheorem{prop}[theorem]{Proposition}
\newtheorem*{prop*}{Proposition}
\newtheorem{cor}[theorem]{Corollary}
\newtheorem{conj}[theorem]{Conjecture}

\theoremstyle{definition}
\newtheorem{definition}[theorem]{Definition}

\newtheorem{remark}[theorem]{Remark}
\newtheorem{example}[theorem]{Example}

\newtheorem{construction}[theorem]{Construction}
\newtheorem{notation}[theorem]{Notation}

\newcommand{\change}[1]{{\color{black}#1}}

\newtotcounter{countcorrection}
\newtotcounter{countcorrectiondone}
\newcommand{\leftdivision}[2]{%
 #1_{#2}^{\begin{tikzpicture}[ampersand replacement=\&][baseline= (a).base]
\node[scale=.25] (a) at (0,0){%
\begin{tikzcd}[column sep=0.08in]
    {\bullet} \& {\bullet} \& {\bullet}
    \arrow[shorten <= -.2em, shorten >= -.2em, ultra thick, curve={height=-12pt},  no head, from=1-2, to=1-3]
    \arrow[shorten <= -.2em, shorten >= -.2em, ultra thick, no head, from=1-1, to=1-2]
    \arrow[shorten <= -.2em, shorten >= -.2em, ultra thick, curve={height=12pt}, no head, from=1-2, to=1-3]
\end{tikzcd}
};
\end{tikzpicture}}}

\newcommand{\rightdivision}[2]{%
#1_{#2}^{\begin{tikzpicture}[ampersand replacement=\&][baseline= (a).base]
\node[scale=.25] (a) at (0,0){%
\begin{tikzcd}[column sep=0.08in]
    \bullet \& \bullet \& \bullet
    \arrow[shorten <= -.2em, shorten >= -.2em,  no head,  ultra thick, from=1-2, to=1-3]
    \arrow[no head, shorten <= -.2em, shorten >= -.2em, ultra thick, curve={height=12pt}, from=1-1, to=1-2]
    \arrow[no head, shorten <= -.2em, shorten >= -.2em, ultra thick, curve={height=-12pt}, from=1-1, to=1-2]
\end{tikzcd}
};
\end{tikzpicture}}}

\newcommand{\cylinder}[2]{%
 #1_{#2}^{\begin{tikzpicture}[ampersand replacement=\&][baseline= (a).base]
\node[scale=.25] (a) at (0,0){%
\begin{tikzcd}[column sep=0.02in , row sep=0.02in]
    \bullet \& \bullet \\
    \bullet \& \bullet
    \arrow[shorten <= -.2em, shorten >= -.2em, ultra thick, no head, from=1-1, to=1-2]
    \arrow[shorten <= -.2em, shorten >= -.2em, ultra thick, no head, from=2-2, to=1-2]
    \arrow[shorten <= -.2em, shorten >= -.2em, ultra thick, no head, from=2-1, to=2-2]
\end{tikzcd}
};
\end{tikzpicture}}}

\DeclareMathOperator*{\Uni}{\text{Uni}}
\DeclareMathOperator*{\colim}{\text{Colim}}
\DeclareMathOperator*{\plim}{\text{pLim}}
\DeclareMathOperator*{\lax}{\text{pLimLax}}

\makeatletter
\newcommand*\fsize{\dimexpr\f@size pt\relax}
\makeatother

\newlength{\csize} 

\newcommand{\ptens}{\mathbin{\text{\setlength{\csize}{0.1 \fsize}\hspace{0.55\csize}\tikz[anchor = base , baseline = -1.1 \csize]{
     \node[circle,draw,line width = 0.25 \csize,text height = 4.7 \csize, inner sep=0pt,scale = 0.58 ] (c) at (0,0){$\mathrlap{\hspace{1.1\csize}\sim}{\phantom{\to}}$};
   }\hspace{0.55\csize}}}}

\newcommand{\hatptens}{\mathbin{\hat \ptens}}

\newcommand{\ltens}{\mathbin{\text{\setlength{\csize}{0.1 \fsize}\hspace{0.55\csize}\tikz[anchor = base , baseline = -1.1 \csize]{
     \node[circle,draw,line width = 0.25 \csize,text height = 4.7 \csize, inner sep=0pt,scale = 0.58 ] (c) at (0,0){$\to$};
   }\hspace{0.55\csize}}}}

\newcommand{\hatltens}{\mathbin{\hat \ltens}}
\newcommand{\hatotimes}{\mathbin{\hat \otimes}}

\newcommand{\icat}{\infty\text{-}\mathbf{Cat}}

\newcommand{\micat}[1][m]{\infty\text{-}\mathbf{Cat}^{+ #1}}
\newcommand{\sset}{\textbf{sSet}}
\newcommand{\strat}{\textbf{Strat}}

\newcommand{\mstrat}[1][m]{\textbf{Strat}^{+ #1}}

\newcommand{\ind}{\text{Ind}}
\newcommand{\satind}{\text{Sat-Ind}}
\newcommand{\coind}{\text{Coind}}
\newcommand{\verity}{\text{V}}
\newcommand{\can}{\text{Can}}

\DeclareMathOperator*{\Colim}{Colim}
\DeclareMathOperator*{\Lim}{Lim}

\DeclareMathOperator{\Hom}{Hom}

\begin{document}

\pagestyle{plain}
\title{An inductive model structure for strict $\infty$-categories}
\date{}

\author{Simon Henry\footnote{Simon Henry has received research support from Natural Sciences and Engineering Research Council of Canada, RGPIN-2020-06779.} and Félix Loubaton\footnote{ Félix Loubaton has received support from  the Agence Nationale de la Recherche program 3ia Côte d’Azur ANR-19- P3IA-0002, and the European Research Council Horizons 2020 grant 670624}}

\maketitle

\begin{abstract}
We construct a left semi-model category of ``marked strict $\infty$-categories'' for which the fibrant objects are those whose marked arrows satisfy natural closure properties and are \change{invertible up to higher marked arrows}. The canonical model structure on strict $\infty$-categories can be recovered as a left Bousfield localization of this model structure. We show that an appropriate extension of the Street nerve to the marked setting produces a Quillen adjunction between our model category and the Verity model structure for complicial sets, generalizing previous results by the second named author. Finally, we use this model structure to study, in the setting of strict $\infty$-categories, the idea that, \change{because they are two different ``truncation functors'' taking an $(\infty,n)$ to an $(\infty,n-1)$-category, there are two non-equivalent definitions for the $(\infty,1)$-category of $(\infty,\infty)$-categories as a limit of the $(\infty,1)$-categories of $(\infty,n)$-categories. We show that in fact there seem to be at least three non-equivalent ways of constructing an $(\infty,1)$-category of $(\infty,\infty)$-categories.}
\end{abstract}




\tableofcontents

\section{Introduction}
In the present paper, we introduce (in \cref{subsec:marked_cat}) a category $\micat$ of ``$m$-marked (strict) $\infty$-categories'' for $m \in \Nb \cup \{\infty\}$. Marked $\infty$-categories are $\infty$-categories with the additional data of a collection of arrows that are meant to be invertible. This is similar to relative categories or stratified simplicial sets. $m$-marked means that all arrows of dimension $>m$ are marked, and the marked arrows are required to be closed under composition, and all identity arrows are marked.

This category $\micat$ is equipped with two monoidal closed structures denoted $\ltens$ and $\ptens$ that are both the Gray-Crans tensor product on the underlying strict $\infty$-categories but act differently on markings. These two monoidal structures are meant to respectively be models for the ``lax-Gray tensor product'' and the ``pseudo-Gray tensor product''.

Our main result is the construction of various left semi-model\footnote{See \cref{sec:semi-model-categories} for a quick review of the theory of left semi-model structures.} structures on $\micat$, that are in the same spirit as the canonical (or ``folk'') model structure $\icat_\can$ on strict $\infty$-categories from \cite{lafont2010folk}, the main one being the saturated inductive model $\micat_\satind$ which is meant to model the homotopy theory of strict $\infty$-categories and serves as toy models for the homotopy theory of weak $(\infty,m)$-categories and $(\infty,\infty)$-categories.

The motivations for this work come from two different places that we will now explain before presenting in more detail the content of this work:

\subsection{The Street Nerve as a Right Quillen Functor}

Complicial sets are a model for weak $(\infty,n)$-categories introduced by Verity in \cite{verity2008weak}. Concretely, a complicial set is a ``stratified simplicial set'', which means that it is a simplicial set where some arrows are marked as being ``thin'', which moreover satisfies some filling conditions that refine those for Kan complexes and quasicategories. One essentially recovers Kan complexes when $n=0$ and quasicategories when $n=1$. We denote by $\mstrat$ the category of $m$-stratified simplicial sets, i.e., stratified simplicial sets where all simplices of dimension $>m$ are thin. It is equipped with a model structure $\mstrat_\verity$, which we refer to as the Verity model structure, whose fibrant objects are the complicial sets. More precisely, we will use the ``saturated'' version of this model structure constructed in \cite{riehl2018complicial}, which we review in \cref{subsec:complicial}.

In \cite{loubaton2021conditions}, the second named author has shown that the Street nerve of a strict $\infty$-category can be made into a complicial set by defining the ``thin'' simplices as those whose top-dimensional arrows are ``coinductively'' invertible, i.e., admit inverses up to arrows of dimension $(n+1)$ that are themselves invertible up to arrows of dimension $(n+2)$, and so on up to infinity.

From there, it is natural to ask whether this stratified version of the Street nerve also preserves fibrations, and hence is a morphism of categories of fibrant objects (and this will be shown in the present paper as \cref{prop:Marked_nerve_pres_fib}).

In fact, more generally, one could ask if it is possible to make this version of the Street nerve into a right Quillen functor (for the Verity model structure on complicial sets from \cite{verity2008weak}). This is not directly possible simply because this stratified Street nerve is not a right adjoint functor. The solution to this problem is to work with markings on both sides: The usual Street nerve from strict $\infty$-categories to simplicial sets is a right adjoint functor, and one can extend it to a right adjoint functor from marked $\infty$-categories to ``marked'' simplicial sets (or rather \emph{stratified} simplicial sets to follow the terminology of \cite{verity2008weak}). In \cref{subsec:complicial}, we show that this functor is indeed a right Quillen functor from the Verity model structure on complicial sets to the saturated inductive semi-model structure on marked $\infty$-categories.

This right Quillen functor from marked $\infty$-categories to stratified simplicial sets is meant to be a model for the forgetful functor from strict $\infty$-categories to weak $(\infty,\infty)$-categories. In particular, the corresponding left Quillen functor from stratified simplicial sets to marked $\infty$-categories is a model for the more mysterious ``strictification functor'', sending weak $(\infty,\infty)$-categories to strict $\infty$-categories.

At the level of $\infty$-groupoids, this strictification functor corresponds essentially to (non-abelian) homology, through the equivalence between strict $\infty$-groupoids and crossed chain complexes (\cite{brown1981equivalence}) which is well-known to be a conservative functor by Whitehead's theorem for homology. The first named author has conjectured \cite{313748} that more generally this strictification functor should be conservative on weak $(\infty,m)$-categories for all $m$. This allows us to state a concrete version of this conjecture here:

\begin{conj}
The left Quillen functor $|\_| : \mstrat_V \to \micat_\satind$ from \cref{subsec:complicial} reflects weak equivalences between cofibrant objects.
\end{conj}

\subsection{The Two (?) Notions of $(\infty,\infty)$-Categories}
\label{intro1.2}

C. Schommer-Pries and C. Rezk have independently argued (\cite{134099}) that there should be more than one notion of weak $(\infty,\infty)$-categories. More precisely, they both arrive at the conclusion that even if one accepts (which seems to be a clear consensus nowadays) that there is only one notion of weak $(\infty,n)$-categories for finite $n$, there are at least two different ways to build a notion of $(\infty,\infty)$-categories out of it.

Before we go into further details, we should say that the following discussion is mostly informal and speculative, and most of it has not been formalized in any models—in fact, one motivation for the present paper is to formalize some of it in the context of strict $\infty$-categories.

First, let us go over the argument put forward by Rezk and Schommer-Pries, or at least how we understand it.

Assuming we agree on what the $(\infty,1)$-category of $(\infty,n)$-categories is for each $n$, the forgetful (or inclusion) functor from $(\infty,n)$-categories to $(\infty,n+1)$-categories is supposed to have both a left adjoint $\pi_n$, which freely adds inverses to all $(n+1)$-arrows, and a right adjoint $\tau_n$ which removes all non-invertible $(n+1)$-arrows. This allows us to produce two different towers \change{of $(\infty,1)$-categories}:

\[ (\infty,0)\text{-}\mathbf{Cat} \overset{\pi_0}{\leftarrow} (\infty,1)\text{-}\mathbf{Cat} \overset{\pi_1}{\leftarrow} (\infty,2)\text{-}\mathbf{Cat} \overset{\pi_2}{\leftarrow} \dots \overset{\pi_{n-1}}{\leftarrow} (\infty,n)\text{-}\mathbf{Cat} \overset{\pi_n}{\leftarrow} \dots \]
\[ (\infty,0)\text{-}\mathbf{Cat} \overset{\tau_0}{\leftarrow} (\infty,1)\text{-}\mathbf{Cat} \overset{\tau_1}{\leftarrow} (\infty,2)\text{-}\mathbf{Cat} \overset{\tau_2}{\leftarrow} \dots \overset{\tau_{n-1}}{\leftarrow} (\infty,n)\text{-}\mathbf{Cat} \overset{\tau_n}{\leftarrow} \dots \]
and one can take the projective limit of either of these two towers to give a definition of what an $(\infty,\infty)$-category is.

\change{
The difference between these two definitions can be seen in the notion of invertibility. }

\change{A $k$-arrow of an $(\infty,n)$-category is "invertible" if it is invertible up to $(k+1)$-arrows, which are themselves invertible up to $(k+2)$-arrows, and so on up to arrows of dimension larger than $n+1$, which are all assumed to be invertible. Therefore, proving that an arrow $c$ is invertible amounts to producing a tower $T_c$ of inverses, witnesses of invertibility, inverses of these witnesses, and so forth, up to arrows of dimension $n+1$. But when $n$ goes to $\infty$, there might be more than one way to define what it means for a cell to be invertible.}

\change{ We say that an arrow $c$ is \textit{coinductively invertible} when there is such a tower $T_c$ of inverses, witnesses of invertibility, inverses of these witnesses, and so forth, but that never ends. This is, for example, how invertibility is defined in the context of strict $\infty$-categories in \cite{lafont2010folk}, or how it is used in \cite{cheng2007omega}. 

Suppose first that we take the limit of the $\tau$-tower as our definition of an $\infty$-category, that is, an $(\infty,\infty)$-category $X$ corresponds to a collection of $(\infty,n)$-categories $X_n$ such that $X_n \simeq \tau_n X_{n+1}$. Note that, in particular, $X_n$ and $X_{n+1}$ have the same $k$-arrow for $k \leqslant n$, so we can talk about the set (or space) of $k$-arrows of $X$ for any $k$: it just means the set of $k$-arrows of $X_n$ for any $n \geqslant k$. In particular, there is another notion of "invertibility" that is present by definition: an $n$-arrow is said to be invertible if it belongs to $X_{n-1}$. In this case, we say that the arrow is "inductively invertible." Note that an arrow that has an inverse up to inductively invertible arrows is itself inductively invertible, but given a coinductively invertible arrow $c$, if none of the arrows of the tower $T_c$ are inductively invertible, the arrow $c$ might not be inductively invertible.

If one takes the definition of the $\pi$-tower as our definition of an $(\infty,\infty)$-category, however, it is quite different: an $(\infty,\infty)$-category in this sense corresponds to a collection of $(\infty,n)$-categories $X_n$ such that $X_n \simeq \pi_n X_{n+1}$. In this definition, if an arrow $c$ is coinductively invertible in the previous sense, with a tower of inverses and witnesses $T_c$, then for each integer $n$, $\pi_n(T_c)$ is a tower of invertibility for the arrow $\pi_n(c)$, which is therefore invertible in $X_n$ for all $n$. Hence, the arrow should be considered invertible in $X$.}

To show more precisely that the two limits should really be different, one can for example consider the $(\infty,\infty)$-category of cobordisms \change{(see for example \cite{baez1995higher})}. In the limit of the $\tau$-tower, one can define it by taking $X_n$ to be the $(\infty,n)$-categories of cobordisms—which do satisfy $\tau_n(X_{n+1}) \simeq X_n$ and hence form a well-defined element of the limit of that $\tau$-tower. By definition, this $(\infty,\infty)$-category $X$ is such that $\tau_m(X) = X_m$, so informally, one recovers the $(\infty,m)$-category of cobordisms by dropping the non-invertible arrows of dimension $>m$.

But this $(\infty,\infty)$-category of cobordisms also has the property that every arrow in every dimension has a dual, because in the $(\infty,m)$-category of cobordisms, every arrow of dimension strictly inferior to $m-1$ has a dual. There is a result by E. Cheng (see \cite{cheng2007omega}) that says that (in at least one model) if every arrow has a dual, then every arrow is coinductively invertible. In particular, if one tries to construct a $\pi_m(X)$, i.e., forces all the arrows of $X$ of dimension $>m$ to be invertible, then we will actually make all the arrows of $X$ of all dimensions invertible, so that $\pi_m(X)$ is in fact an $\infty$-groupoid and does not depend on $m$. It follows that any attempt at constructing something like the $(\infty,\infty)$-category of cobordisms, with properties similar to those we can construct in the limit of the $\tau$-tower, will, in the limit of the $\pi$-tower, result in a constant family of $\infty$-groupoids which will not remember the subtle structure of the $(\infty,m)$-category of cobordisms for finite $m$.

\bigskip

Using the \change{saturated inductive left semi-model} structures on marked strict $\infty$-categories we construct in this paper, we will make the construction of the $\pi$-tower and of the $\tau$-tower formal in the context of strict $\infty$-categories. See the next subsection for a more detailed account. This is of course only meant to be a toy model for the case of weak $(\infty,\infty)$-categories, but it is already interesting, and it will show that the picture above, while correct, needs to be refined a little.

\change{
First, we will show in \cref{subsec:truncation} that the saturated inductive left semi-model structure $\micat[\infty]$ corresponds to the (putative) homotopy limit of the saturated inductive left semi-model structure on $\micat$ for $m < \infty$ using the $\tau_n$ functor as transition functors.

Here there is a small gap we should disclaim: The notion of homotopy limit of a tower of model structures from which we have taken inspiration was introduced in \cite{bergner2012homotopy}. However, they only developed the theory of such limits for Quillen model categories and not left semi-model categories, and we will apply their construction to our left semi-model categories directly.

In order for our argument to be complete despite this, we will prove that the construction from \cite{bergner2012homotopy} does yield a left semi-model category, but we will not reprove that it corresponds to a homotopy limit as in \cite[Theorem 5.1]{bergner2012homotopy}. For this reason, we will call this construction the \textit{putative limit} of the $\tau$-tower.
However, it should be noted that in practice, the argument of \cite{bergner2012homotopy} seems to carry over to our setting with almost no changes, so this gap is not really a concern.
}

\bigskip

In \cref{subsec:Folk}, we will show that the canonical model structure is equivalent to a left Bousfield localization $\micat[\infty]_\coind$ of $\micat[\infty]_\satind$ which corresponds to turning all coinductively invertible arrows into equivalences.

However, we will also show in \cref{sec:lim_of_pi_tower} that the canonical model structure is not equivalent to the limit of the $\pi$-tower. More precisely, the natural functor from the canonical model structure to this limit is not an equivalence. It is unclear if the limit of the tower of $\pi_n$ corresponds to a further localization of our model structure, or if it is something entirely different. Nevertheless, we find that the argument we will give in \cref{sec:lim_of_pi_tower} to distinguish between the canonical model structure and the limit of the $\pi$-tower shows that this limit exhibits behaviors that are not really expected from a notion of $(\infty,\infty)$-categories, or at least are not typical of any known model of $\infty$-categories.

\bigskip

Returning to the world of weak $(\infty,\infty)$-categories, this suggests that the two most interesting notions of weak $(\infty,\infty)$-categories should be the limit of the $\tau_n$ tower, which corresponds to an ``inductive'' notion of equivalences, and its localization that turns the coinductive equivalences into equivalences\footnote{Provided that we can define the notion of coinductively invertible arrow in a ``model independent'' way, which is not investigated in this article.}. However, this localization should be different from the limit of the $\pi_n$-tower, which might not be an interesting notion of $(\infty,\infty)$-categories. What we mean here is that we are not aware of any attempt to give a concrete definition of $(\infty,\infty)$-categories that seems to produce something that could be equivalent to this limit.

\subsection{Overview of the Paper}
Finally, we give a short presentation of the contents of the paper, the various model structures, and Quillen functors we will construct. We will assume some familiarity with the theory of left semi-model categories— the necessary material is recalled in \cref{sec:semi-model-categories}.

In \cref{subsec:recall_infty_cat}, we briefly recall the basics of the theory of strict $\infty$-categories, mostly in order to fix our notations. The category $\micat$ of $m$-marked $\infty$-categories is introduced in \cref{subsec:marked_cat}, and in \cref{subsec:Marked_Gray} we introduce the two monoidal structures $\ptens$ and $\ltens$ on $\micat$, which both correspond to the Gray-Crans tensor products at the level of the underlying strict $\infty$-categories but behave differently on the markings. $\ltens$ is meant to correspond to the Lax Gray-Crans tensor product, while $\ptens$ corresponds to the pseudo Gray-Crans tensor product.

Next, in \cref{subsec:Semi_model}, exploiting these monoidal structures, we set up the first left semi-model structure on $\micat$, which we call the \emph{inductive} model structure, whose properties are summarized in:

\begin{theorem} \change{For any $m \in \Nb \cup \{\infty\}$}, there is a combinatorial left semi-model structure on the category $\micat$ of $m$-marked $\infty$-categories, called the \emph{inductive} or \emph{unsaturated inductive} model structure and denoted $\micat_\ind$, such that:
  
 \begin{itemize}
 \item This model structure is monoidal for both tensor products $\ptens$ and $\ltens$ (from \cref{subsec:Marked_Gray}).
 \item The cofibrations are the maps that are cofibrations of the canonical model structure between the underlying $\infty$-categories. (\cref{prop:cofibration of the marked model structure})
 \item The fibrant objects are the marked $\infty$-categories in which all marked arrows admit marked inverses \change{up to higher marked arrows}, and in which if there is a marked arrow $a \to b$, then $a$ is marked if and only if $b$ is marked.
 \item Fibrations between fibrant objects are the ``isofibrations'' (as defined in \cref{subsec:isofibration}).
 \item Weak equivalences between fibrant objects are ``equivalences of marked $\infty$-categories'' (as defined in \cref{subsec:equivalence}).
 \end{itemize}
\end{theorem}

The existence of this model structure is established in \cref{subsec:Semi_model}, but some of its properties, in particular, the characterization of fibrant objects and fibrations between fibrant objects, will only be established in \cref{sec:eq_saturation}.
  
This model structure is intended as a model for ``strict $(\infty,m)$-categories'', i.e., strict $\infty$-categories whose arrows of dimension strictly superior to $m$ are invertible in a weak sense. In this interpretation, marked arrows should correspond to weakly invertible arrows. \change{When $m = \infty$, it still retains an "inductive" notion of invertibility, like what is expected of the limit of the $\tau$-tower as mentioned in \cref{intro1.2}.}

 However, it is not quite the case yet due to a small defect: Given $X$ a fibrant object, there might be arrows in $X$ that are invertible up to higher marked arrows without being marked themselves. Hence, the fibrant objects are carrying an additional piece of data compared to what $(\infty,n)$-categories should be: some of their invertible arrows are marked and others are not.

To solve this problem, in \cref{subsec:saturation} we consider a left Bousfield localization $\micat_\satind$, called the \emph{saturated inductive model structure}, in which the fibrant objects are the marked $\infty$-categories in which an arrow is marked if and only if it is invertible up to higher-dimensional marked arrows. These are really our intended model for strict $(\infty,n)$-categories. So we have a first (identity) left Quillen functor:
\[ \micat_\ind \to \micat_\satind \]

We consider the saturated inductive model structure $\micat_\satind$ to be the most interesting one, as it actually models strict $(\infty,m)$-categories. The only reason we use $\micat_\ind$ is because it is the one that naturally arises from our construction in \cref{subsec:Semi_model}. It is not completely clear to us what $\micat_\ind$ actually models at a homotopy theoretic level.

In \cref{subsec:truncation}, we study how these model structures relate when $m$ varies. We show that for $m < p \leqslant \infty$, the obvious inclusion functor $\iota_p: \micat \subset \micat[p]$ has both a left adjoint $\pi_m$ and a right adjoint $\tau_m: \micat[p] \to \micat$. We show that these form two Quillen adjunctions $(\pi_m \dashv \iota_p)$ and $(\iota_p \dashv \tau_m)$ between the saturated inductive model structures.

We also investigate how the saturated inductive model structure $\micat[\infty]_\satind$ can be understood as a certain limit of the tower of right Quillen functors
\[ \micat[0]_\satind \overset{\tau_0}{\leftarrow} \micat[1]_\satind \overset{\tau_1}{\leftarrow} \micat[2]_\satind \overset{\tau_2}{\leftarrow} \dots \overset{\tau_{n-1}}{\leftarrow} \micat[n]_\satind \overset{\tau_n}{\leftarrow} \dots  \]
as explained previously.

Next, in \cref{subsec:Folk}, in the case where $m = +\infty$, we can take a further left Bousfield localization, which we study in \cref{subsec:Folk}, called the coinductive model structure, denoted $\micat[\infty]_\coind$, whose fibrant objects are marked $\infty$-categories where the marked arrows are exactly the ``coinductively invertible arrows'' (see \cref{def:coinductively_invertible}):

\[ \micat[\infty]_\ind \to \micat[\infty]_\satind \to \micat[\infty]_\coind \]

Of course, we can also try to define $\micat_\coind$ for finite $m$, but this is the same as $\micat_\satind$.

This second localization $\micat[\infty]_\coind$ is in fact equivalent to the canonical model structure on $\infty$-categories $\icat_\can$ from \cite{lafont2010folk}, in a fairly strong sense: the functor
\[
  \begin{array}{ccc}
    \icat_\can & \to & \micat[\infty]_\coind \\
                C     & \mapsto & C^\flat 
  \end{array}
\]
where $C^\flat$ is the minimal marking (i.e., only the identity arrows are marked) defined in \cref{ex:sharp_and_flat}, is a left Quillen equivalence. Its right adjoint (the forgetful functor $\micat[\infty] \to \icat$) induces an equivalence of categories between the categories of fibrant objects that preserves and detects weak equivalences and fibrations (between fibrant objects). Thus, their categories of fibrant objects are literally the same, with the same fibrations and weak equivalences.

Finally, in \cref{subsec:complicial}, we study a marked version of the Street nerve. The usual Street Nerve is the right adjoint functor $N_\Ocal: \icat \to \sset$, defined using Street's Orientals $\Ocal: \Delta \to \icat$, where $N_\Ocal(X)_n = \Hom(\Ocal [n], X)$. We extend it to a Nerve/realization Quillen adjunction:
\[ |\_| : \mstrat_\verity \leftrightarrows \micat_\satind : N \]
where $\mstrat_\verity$ is the category of $m$-marked simplicial sets equipped with the (saturated) Verity model structure from \cite{verity2008weak} and \cite{riehl2018complicial}, which we review in \cref{subsec:complicial}. As explained above, this generalizes the results of the second named author from \cite{loubaton2021conditions}.

\section{$\infty$-Categories and Marked $\infty$-Categories}

\subsection{$\infty$-Categories}
\label{subsec:recall_infty_cat}

A globular set is a presheaf on the globular category $\Gb$:

\[
\begin{tikzcd}
\Db_0 \ar[r,bend left = 30,"i^+_0"] \ar[r,bend right = 30,"i^-_0"below]  &\Db_1 \ar[r,bend left = 30,"i^+_1"] \ar[r,bend right = 30,"i^-_1"below] & \Db_2 \ar[r,bend left = 30,"i^+_2"] \ar[r,bend right = 30,"i^-_2"below] & \Db_3 \ar[r,bend left = 30,"i^+_3"] \ar[r,bend right = 30,"i^-_3"below] &  \Db_4 \dots
\end{tikzcd}
\]
with the relations $i_n^{+} i_{n-1}^\epsilon = i_n^{-} i_{n-1}^\epsilon $ for any $n>0$ and $\epsilon \in \{+,-\}$. 
For any $n>k$ and $\epsilon \in \{+,-\}$, we also denote by $i^{\epsilon}_k$ the composite $\Db_{k} \xrightarrow{i^{\epsilon}_k} \Db_{k+1}\xrightarrow{f} \Db_n$ where $f$ is any map. These and the identity arrows are the only maps in the category $\Gb$.

\begin{notation}
If $X$ is a globular set, one denotes by $X_n$ the set $X(\Db_n)$. The map $X_n \to X_k$ induced by $i^\epsilon_k : \Db_k \to \Db_n$ is denoted by $\pi^\epsilon_k$.
\end{notation}

\begin{definition}
Let $X$ be a globular set and $n$ a positive integer. A \textit{$n$-arrow of $X$} is an element of $X_n$.

A \textit{arrow of $X$} is an element of $\coprod_{k \geq 0} X_k$. If $a$ is an arrow of $X$, its \textit{dimension} is the integer $n$ such that $a$ belongs to $X_n$.

If $a$ is an $n$-arrow of $X$ and $k$ an integer strictly less than $n$, the \textit{$k$-source of $a$} is the $k$-arrow $\pi^-_k(a)$ and the \textit{$k$-target of $a$} is the $k$-arrow $\pi^+_k(a)$.
\end{definition}

\begin{definition}
An \emph{$\infty$-category} is a globular set $X$ together with operations of \emph{compositions}
\[ X_n \times_{X_k} X_n \to X_n  \quad (0 \leq k < n) \]
which associates to two $n$-arrows $(x,y)$ verifying $\pi_k^+(x) = \pi_k^-(y)$, one $n$-arrow $x \#_k y$,
as well as \emph{identities}
\[ X_n \to X_{n+1} \]
associating to an $n$-arrow $x$, an $(n+1)$-arrow $\Ib_x$, and satisfying the following axioms:

\begin{enumerate}
\item $\forall x \in X_n, \pi^\epsilon_n(\Ib_x) = x$.

\item $\pi^-_k (x \#_k y) = \pi_k^-(x)$ and $\pi^+_k(x \#_k y) = \pi_k^+(y)$ whenever the composition is defined and $k \leqslant n$.

\item $\pi^\epsilon_k (x \#_k y) = \pi_k^{\epsilon}(x) \#_k \pi^\epsilon_k(y)$ whenever the composition is defined and $k > n$.

\item $x \#_k \Ib_{\pi^+_k x} = x$ and $\Ib_{\pi^-_k x} \#_k x = x$.

\item $(x \#_k y) \#_k z = x \#_k (y \#_k z)$ as soon as one of these is defined.

\item If $k < n$

\[ (x \#_k y) \#_k (z \#_k w) = (x \#_k z) \#_k (y \#_k w) \]
when the left-hand side is defined.
\end{enumerate}

A morphism of $\infty$-categories is a map of globular sets commuting with both operations. The category of $\infty$-categories is denoted $\icat$.
\end{definition}

\begin{definition}
An $(n+1)$-arrow $c$ in an $\infty$-category is said to be \emph{trivial}, or an \emph{identity arrow}, if there exists an $n$-arrow $d$ such that $c=\Ib_d$.
\end{definition}

\begin{example}
By abuse of notation, we also denote $\Db_n$ as the $\infty$-category that admits for any $k < n$ only two non-trivial $k$-arrows, denoted $e_k^-$ and $e_k^+$, and a single non-trivial $n$-arrow, denoted $e_n$, satisfying:
\[
\begin{array}{rcl}
\pi_l^{-}(e_k^\epsilon)= e_l^{-}&\pi_l^{+}(e_k^\epsilon)= e_l^{+}& \mbox{ for $l\leq k<n$}\\
\pi_l^{-}(e_n)= e_l^{-}&\pi_l^{+}(e_n)= e_l^{+}& \mbox{ for $l\leq n$}\\
\end{array}
\]

The $\infty$-category $\partial\Db_n$ is obtained from $\Db_n$ by removing the $n$-arrow $e_n$. We thus have a morphism
\[i_n: \partial\Db_n \to \Db_n.\]
Note that $\partial \Db_0 = \emptyset$.
\end{example}

\begin{definition}
\label{defi:suspension}
If $X$ is an $\infty$-category, we define the globular set $\Sigma X$, called the \textit{suspension of $X$}, by the formula 
\[
(\Sigma X)_{0} = \{a, b\}, \quad (\Sigma X)_{n+1} := X_n \cup \{\Ib^n a, \Ib^n b\},
\]
where $\Ib^n_a$ (resp. $\Ib^n_b$) is the $n$-times iterated identity of $a$ (resp. of $b$). Moreover, $\Sigma X$ inherits from $X$ a structure of an $\infty$-category.

Eventually, for an integer $n$, we define the $\infty$-category $\Sigma^n X$, called the \textit{$n$-suspension of $X$}, as the $n$-times iterated suspension of $X$.
\end{definition}

Next, we define the notion of polygraphs, first introduced under the name ``computads'' by R.~Street in \cite{street1976limits} for $2$-categories, with the general notion being hinted at in \cite{street1987algebra}. As far as we know, the first formal introduction of polygraphs in the literature is in \cite{power1991n} and independently in \cite{burroni1993higher}, where the name ``polygraphs'' was introduced. Here we will exploit that the category of polygraphs identifies with a (non-full) subcategory of $\icat$ to give a shorter definition. We refer to the references above for a more complete introduction.

\begin{definition} $ $
   \begin{itemize}
   \item We say that an $\infty$-category $X$ is a polygraph if it can be constructed from the empty $\infty$-category by freely adding arrows with specified source and target. That is, $X$ can be obtained as a transfinite composition $\emptyset = X_0 \to X_1 \to \dots \to X_i \to \colim X_i = X$, where for each $i$, the map $X_i \to X_{i+1}$ is a pushout of $\coprod_S \partial \Db_n \to \coprod_S \Db_{n+1}$.
   \item An arrow of a polygraph is said to be a \emph{generator} if it is one of the arrows that has been freely added at some stage.
   \item A morphism of $\infty$-categories between two polygraphs is said to be a \emph{morphism of polygraphs} or a \emph{polygraphic} morphism if it sends each generator to a generator.
   \item An \emph{$n$-polygraph} is a polygraph whose generators are all of dimension less than or equal to $n$.
   \end{itemize}
\end{definition}

\begin{remark}
Generators of a polygraph can be shown to be exactly the arrows that cannot be written as a composite in a non-trivial way\footnote{The trivial ones being decompositions involving units, such as the decompositions $ u = u \#_i \Ib^k_{\pi_i^+ u} = \Ib^k_{\pi_i^- u} \#_i u $.}, see 16.6.1 and 16.6.2 in \cite{ara2023polygraphs}.

So, the notion of generator does not depend on the choice of the presentation of $X$, and any isomorphism between polygraphs is automatically polygraphic, see 16.6.3 in \cite{ara2023polygraphs}.
\end{remark}

\begin{example}
The only $n$-polygraph for $n < 0$ is the empty $\infty$-category. The category of $0$-polygraphs is equivalent to the category of sets and corresponds to discrete $\infty$-categories. The category of $1$-polygraphs (and polygraphic morphisms between them) is equivalent to the category of directed graphs, and they correspond to categories that are free on a graph.
\end{example}

We will sometimes distinguish between a polygraph seen as an object of the category of polygraphs and polygraphic morphisms, and the corresponding $\infty$-category, which we call the free $\infty$-category on the polygraph.

\begin{remark}
Each arrow in a polygraph can be written as an iterated composite of the generators (not necessarily in a unique way). For an $n$-arrow $f$, the set of generators of dimension $n$ that appear in such an expression, and even the number of times they appear, is the same for all such expressions, see section 4.3 of \cite{metayer2008cofibrant}. We will say that an $n$-generator \emph{appears} in an $n$-arrow if it appears in any such expression.
\end{remark}

\begin{construction} \label{cstr:Tensor_product}
The category $\icat$ admits a closed monoidal structure, called the Gray tensor product or Crans-Gray tensor product, which we denote as
\[
 \begin{array}{ccc}
   \icat \times \icat & \to & \icat \\
   X, Y & \mapsto & X \otimes Y
 \end{array}
\]
Its explicit construction is very involved, and we will assume the reader is already familiar with it. It was first introduced by S.~Crans in his Ph.D. thesis \cite{crans1995combinatorial}. We refer to \cite{al2002multiple} for an introduction to this tensor product close to its original definition, and to \cite{steiner2004omega} for a more modern account. The proof of the existence of this monoidal structure in \cite{steiner2004omega} contains some gaps that have been fixed in Appendix A of \cite{ara2016join}.

It is easy to see from either of these definitions that $\Db_n \otimes \Db_m$ has a unique non-trivial arrow of dimension $n+m$. If $f$ and $g$ are respectively an $n$-arrow of $X$ and an $m$-arrow of $Y$, which correspond to morphisms $f: \Db_n \to X$ and $g: \Db_m \to Y$, we denote by $f \otimes g$ the $(m+n)$-arrow of $X \otimes Y$ obtained as the image of this non-trivial $(n+m)$-arrow by the functor $f \otimes g : \Db_n \otimes \Db_m \to X \otimes Y$.
\end{construction}

\begin{example}
\label{example:of_gray_cylinder}
The following description of $\Db_1 \otimes \Db_n$ comes from Appendix B.1 of \cite{ara2016join} (see Proposition B.1.4): As a polygraph, the generating arrows of $\Db_1 \otimes \Db_n$ are:
 \[ a^-_0 \otimes e^\epsilon_k \qquad a^+_0 \otimes e^\epsilon_k \qquad a \otimes e^\epsilon_k \]
where the arrows of $\Db_1$ have been denoted ``$a$'' instead of ``$e$'' to distinguish them, and $\epsilon$ is either $+$ or $-$, $k \leqslant n$, and $e^+_n = e^-_n$. Their source and target are given as follows:
 \[ \pi^-(a^-_0 \otimes e^\epsilon_k ) = a^-_0 \otimes e^-_{k-1} \qquad \pi^+(a^-_0 \otimes e^\epsilon_k ) = a^-_0 \otimes e^+_{k-1} \]
 \[ \pi^-(a^+_0 \otimes e^\epsilon_k ) = a^+_0 \otimes e^-_{k-1} \qquad \pi^+(a^+_0 \otimes e^\epsilon_k ) = a^+_0 \otimes e^+_{k-1} \]
   \[\pi^-(a \otimes e^\epsilon_k) = (a^-_0 \otimes e^\epsilon_k) \#_0 (a \otimes e^+_0) \#_1 \dots \#_{k-1} (a \otimes e^+_{k-1})\]
   \[\pi^+(a \otimes e^\epsilon_k) = (a \otimes e^-_{k-1}) \#_{k-1} \dots \#_1 (a \otimes e^-_0) \#_0 (a^+_0 \otimes e^\epsilon_k) \]
   We did not put parentheses in the expressions above to keep them shorter; the default convention is to perform the composition $\#_i$ in order of increasing values of $i$. The last two equations are given by Proposition B.1.4 of \cite{ara2016join}, though note that this reference uses a different convention than ours regarding the composition order.
\end{example}

We recall from Theorem 1.35 of \cite{hadzihasanovic2017algebra} or from \cite{ara2020folk}:
\begin{prop}
\label{prop:generator_of_gray_tensor_of_polygraph_unmarked_case}
 If $X$ and $Y$ are polygraphs, then $X \otimes Y$ is also a polygraph. The $n$-generators of $X \otimes Y$ are the arrows of the form $x \otimes y$, where $x$ and $y$ are respectively a \change{$(n-k)$-generator of $X$ and a $k$-generator of $Y$, with $k \leq n$.} 
\end{prop}

\change{
\begin{lemma}
\label{lemma:generator_of_gray_tensor_unmarked_case}
Let $X$ and $Y$ be $\infty$-categories. The $\infty$-category $X \otimes Y$ is generated by composition of $n$-arrows of the shape $x \otimes y$, where $x$ and $y$ are respectively an $(n-k)$-arrow of $X$ and a $k$-arrow of $Y$, with $k \leq n$.
\end{lemma}
\begin{proof}
Remark first that given a diagram $F: I \to \icat$ such that, for any $i \in I$, $F(i)$ is generated by composition from a set $M_i$, then the $\infty$-category $\colim_{i \in I} F_i$ is generated by composition from the set $\cup_{i \in I} f_i(M_i)$ where $f_i$ is the canonical map $f(i):F(i)\to\colim_{i \in I} F_i$.

Secondly, Theorem 1.12 of \cite{berger2002cellular} states that a certain subcategory of $\icat$, denoted by $\Theta$ and whose objects are polygraphs, is dense. As the Gray tensor product preserves colimits, the previous remark implies that we can reduce to the case where $X$ and $Y$ are elements of $\Theta$, and so in particular polygraphs. The result then follows from \cref{prop:generator_of_gray_tensor_of_polygraph_unmarked_case}.
\end{proof} }
Finally, we recall from \cite{lafont2010folk} that $\icat$ carries a model structure, called the canonical model structure, in which every object is fibrant and where the generating cofibrations are the maps $\partial \Db_n \to \Db_n$. Its weak equivalences are a natural class of equivalence of $\infty$-categories that generalizes the equivalences of ordinary categories. It was shown in \cite{metayer2008cofibrant} that the cofibrant objects are exactly the polygraphs, and it also follows from this that the cofibrations between cofibrant objects are the polygraphic inclusions. It was shown in \cite{ara2020folk} that this model structure is a monoidal model structure for the Gray tensor product.

\subsection{Marked $\infty$-Categories}
\label{subsec:marked_cat}

For the rest of the article, we fix an $m \in \mathbb{N} \cup \{\infty\}$.

\begin{definition}
An \emph{$m$-marked $\infty$-category} is an $\infty$-category $X$, together with a set $M \subset \coprod_{k > 0} X(k)$ of arrows of positive dimension called \emph{marked} arrows such that:
\begin{itemize}
\item All identity arrows $\Ib_x$ are marked.
\item All arrows of dimension strictly greater than $m$ are marked.
\item If $x$ and $y$ are marked $n$-arrows and $x \#_k y$ is defined, then $x \#_k y$ is marked.
\end{itemize}

A morphism of $m$-marked $\infty$-categories is a morphism between the underlying $\infty$-categories that sends marked arrows to marked arrows. The category of $m$-marked $\infty$-categories is denoted $\micat$.
\end{definition}
Note that if $m = \infty$, then the second condition of the definition simply disappears; this is the main case we are interested in.

\begin{example}\label{ex:sharp_and_flat}
If $X$ is an $\infty$-category, we denote by $X^\#$ the $m$-marked $\infty$-category $(X, X_{>0})$ where all arrows of positive dimension are marked. We denote by $X^\flat$ the $m$-marked $\infty$-category where only identity arrows and $k$-arrows for $k > m$ are marked.
\end{example}

\change{
\begin{notation}
To simplify notation and when there is no confusion, the marked $\infty$-category $X^{\flat}$ will simply be denoted as $X$.
\end{notation}
}

\begin{construction} \label{cstr:saturation}
If $X$ is an $\infty$-category and $M \subset \coprod_{k > 0} X_k$ is a set of arrows of $X$, we denote by $\overline{M}$ the smallest set of arrows such that $M \subset \overline{M}$ and $(X, \overline{M})$ is an $m$-marked $\infty$-category. That is, $\overline{M}$ is the union of the set of arrows of dimension strictly greater than $m$ and the set of all $n$-arrows that can be written as iterated composites of $n$-arrows in $M$ and arrows of the form $\Ib_x$ for $x$ an $(n-1)$-arrow. For example, $X^\flat = (X, \overline{\emptyset})$.
\end{construction}

\begin{construction} \label{cstr:marked_colimits}
The category of $m$-marked $\infty$-categories has all colimits, and they are easily described in terms of colimits of $\infty$-categories and of \cref{cstr:saturation}: if $(X_i, M_i)_{i \in I}$ is a diagram of $m$-marked $\infty$-categories indexed by a category $I$, then:
\[ \colim_{i \in I} (X_i, M_i) = \left(\colim_{i \in I} X_i, \overline{\cup_i f_i(M_i)} \right) \]
where $f_i$ denotes the canonical map $f_i: X_i \to \colim_{i \in I} X_i$ and $f_i(M_i)$ is simply the set of arrows of the form $f_i(x)$ for $x \in M_i$.

This is easily shown by checking that the right-hand side has the universal property of the colimit.
\end{construction}

\change{
\begin{remark}
\label{rem:micat localement presentable}
Theorem 1.12 of \cite{berger2002cellular} identifies a small full subcategory of $\icat$, denoted $\Theta$, which is dense. We denote by $\Theta^{+m}$ the full subcategory of $\micat$ whose objects are of the form $(C, M)$ with $C$ in $\Theta$. From the description of colimits of $m$-marked $\infty$-categories given in \cref{cstr:marked_colimits}, it follows that $\Theta^{+m}$ is dense in $\micat$.  Moreover, as objects of $\Theta^{+m}$ have a finite number of non trivial cells, they are $\omega$-small. It follows that $\micat$ is locally finitely presentable.
\end{remark}}
%
%
%
%

\subsection{Tensor Product of $m$-Marked $\infty$-Categories}
\label{subsec:Marked_Gray}

In this section, we construct two monoidal closed structures on the category of $m$-marked $\infty$-categories, respectively called the \emph{pseudo-Gray} tensor product $\ptens$ and the \emph{lax-Gray} tensor product $\ltens$. Both are obtained by putting different markings on the Gray tensor product from \cref{cstr:Tensor_product}. For example, the lax-Gray tensor product $\Db_1 \ltens \Db_1$ is $C_1^\flat$, where $C_1$ is the polygraph

\[ C_1 = \left( \begin{tikzcd}
    \bullet & \bullet \\
    \bullet & \bullet
    \arrow[from=1-1, to=1-2]
    \arrow[from=1-1, to=2-1]
    \arrow[from=2-1, to=2-2]
    \arrow[from=1-2, to=2-2]
    \arrow[shorten <=4pt, shorten >=4pt, Rightarrow, from=1-2, to=2-1]
\end{tikzcd} \right) \]
while $\Db_1 \ptens \Db_1$ is the  $m$-marked polygraph $(C_1, \overline{D})$, where $D$ only contains the unique $2$-dimensional generator of $C_1$. So, unless $m = 0$ or $m = 1$, the two tensor products are distinct. At the derived or homotopy-theoretic level, the pseudo-Gray tensor product should correspond to the Cartesian product.

The formal definition is as follows:

\begin{construction}\label{cstr:ltens_ptens}
Given two $m$-marked $\infty$-categories $(X, M)$ and $(Y, N)$, we define two sets of arrows in $X \otimes Y$:

\begin{itemize}
\item $M \ltens N$ is the set of arrows of the form $x \otimes y \in X \otimes Y$ where either $x \in M$ or $y \in N$.

\item $M \ptens N$ contains all arrows in $M \ltens N$ together with all arrows of the form $x \otimes y$ with $x$ and $y$ both of dimension strictly greater than $0$.
\end{itemize}

Note that $M \ltens N$ and $M \ptens N$ are not markings on $X \otimes Y$: they are not stable under composition. So we define:
\[ (X, M) \ltens (Y, N) = (X \otimes Y, \overline{M \ltens N}) \]
\[ (X, M) \ptens (Y, N) = (X \otimes Y, \overline{M \ptens N}) \]
\end{construction}
We will show in \cref{lem:acyclic_cof_left_pushout_product} that both make the category of $m$-marked $\infty$-categories into a monoidal closed category.

In order to show this, it is convenient to introduce the following notations:

\begin{notation}\label{notation:set_of_arrows}
For $A$ and $B$ subsets of arrows in $\infty$-categories, we denote by $A \otimes B$ the set of arrows of the form $a \otimes b \in X \otimes Y$ for $a \in A$ and $b \in B$. For an $\infty$-category $X$, we denote by $X_{\geqslant 0}$ the set of all arrows of $X$ and by $X_{>0}$ the set of all arrows of dimension strictly greater than $0$. We can hence, for $(X, M)$ and $(Y, N)$ two $m$-marked $\infty$-categories, rewrite the definitions above as:

\[\begin{array}{rcl}
  M \ltens N &=& \left( M \otimes Y_{\geqslant 0} \right) \cup \left( X_{\geqslant 0} \otimes N \right) \\
  M \ptens N &=& \left( M \ltens N \right) \cup \left( X_{>0} \otimes Y_{>0} \right) \\
  &=& \left( M \otimes Y_{\geqslant 0} \right) \cup \left( X_{\geqslant 0} \otimes N \right) \cup \left( X_{>0} \otimes Y_{>0} \right)
\end{array}
\]
\end{notation}

By definition of the Gray tensor product, we have the following result:

\begin{lemma}
\label{lem:generators_of_gray_tensor}
Let $X$ and $Y$ be two $\infty$-categories. Then:

\[ \overline{X_{\geqslant 0} \otimes Y_{\geqslant 0}} = \left( X \otimes Y \right)_{\geqslant 0} \]
\[ \overline{X_{>0} \otimes Y_{\geqslant 0} \cup X_{\geqslant 0} \otimes Y_{>0}} = \left( X \otimes Y \right)_{>0}. \]
\end{lemma}

\change{
\begin{proof}
The first equality corresponds to the fact that $X \otimes Y$ is generated under composition by arrows of the form $x \otimes y$, as proven in \cref{lemma:generator_of_gray_tensor_unmarked_case}. The second equality corresponds to the fact that arrows of dimension strictly greater than $0$ in $X \otimes Y$ are generated under composition by arrows of the form $x \otimes y$ where either $x$ or $y$ has dimension strictly greater than $0$, which directly follows from the previous claim, and from the fact that $x \otimes y$ is of dimension strictly greater than $0$ if at least one of $x$ or $y$ is.
\end{proof}
}

\begin{lemma}
\label{lem:compatibility_of_coproduct_and_saturation}
Let $X$ be an $\infty$-category and $M, N$ be two subsets of arrows of $X$. Then:

\[ \overline{M \cup N} = \overline{\overline{M} \cup N} = \overline{M \cup \overline{N}} = \overline{\overline{M} \cup \overline{N}} \]
\end{lemma}

\begin{proof}
This is straightforward.
\end{proof}

\begin{lemma}\label{lem:compatibility_of_gray_tensor_and_saturation}
Let $X$ and $Y$ be two $\infty$-categories and $M \subset X_{\geqslant 0}$ and $N \subset Y_{\geqslant 0}$. Then:

\[ \overline{M \otimes N} = \overline{\overline{M} \otimes N} = \overline{M \otimes \overline{N}} = \overline{\overline{M} \otimes \overline{N}} \]
\end{lemma}

\begin{proof}
We will only show the equality $\overline{M \otimes N} = \overline{\overline{M} \otimes N}$. The equality $\overline{M \otimes N} = \overline{M \otimes \overline{N}}$ can be proved in the same way, and the last equality follows immediately by applying the result to $M$ and $\overline{N}$. 

We will also only prove the results for $m = \infty$; the case of a general $m$ follows immediately as it marks all arrows of dimension strictly greater than $m$ on each side of these equalities.

The evident inclusion $M \subset \overline{M}$ implies $\overline{M \otimes N} \subset \overline{\overline{M} \otimes N}$, so it is enough to show that $\overline{M} \otimes N \subset \overline{M \otimes N}$.

Let $K$ be the set of arrows $k$ in $X$ such that $k \otimes n \in \overline{M \otimes N}$ for all $n \in N$. We need to show that $K$ is closed under identity and composition to finish the proof. 

If $k = \Ib_x$, then $k \otimes n = \Ib_{x \otimes n} \in \overline{M \otimes N}$. Let now $k, k' \in K$ of dimension $n$ such that $k \#_i k'$ is defined. They are encoded by a map $\Db_n \coprod_{\Db_i} \Db_n \to X$, and let $y \in N$ be an arrow of dimension $m$ in $Y$, encoded by a map $\Db_m \to Y$.

Together these induce a map $e:\left( \Db_n \coprod_{\Db_i} \Db_n \right) \otimes \Db_m \to X \otimes Y$. $\left( \Db_n \coprod_{\Db_i} \Db_n \right) \otimes \Db_m$ is a polygraph of dimension $m+n$ with only two generating arrows of maximal dimensions that are sent to $k \otimes y$ and $k' \otimes y$, which are by hypothesis in $\overline{M\otimes N}$.

Now the arrow corresponding to $(k \#_i k') \otimes y$ in $\left( \Db_n \coprod_{\Db_i} \Db_n \right) \otimes \Db_m$ is in $\overline{M\otimes N}$ as all the top-dimensional generators that appear in it are in $\overline{M\otimes N}$. We have proved that $k \#_i k' \otimes y \in \overline{M \otimes N}$ for all $y \in N$, hence $k\#_i k' \in K$ and this concludes the proof.
\end{proof}

\begin{lemma}
\label{lem:compatibility_of_lax_and_pseudo_with_saturation}
Let $X,Y$ be two $\infty$-categories, $M \subset X_{\geqslant 0}$ and $N\subset Y_{\geqslant 0}$. Then we have
\[ \begin{array}{rcl}
\overline{M\ltens N} & = & \overline{\overline{M} \ltens \overline{N}}\\
\overline{M\ptens N} & = & \overline{\overline{M} \ptens \overline{N}}.
\end{array} \]
\end{lemma}
\begin{proof}
Given the formula for $M \ltens N$ and $M \ptens N$ from \cref{notation:set_of_arrows}, this is a direct consequence of \cref{lem:compatibility_of_coproduct_and_saturation} and \cref{lem:compatibility_of_gray_tensor_and_saturation}.
\end{proof}

\begin{lemma}
\label{lem:associativity_of_lax_and_pseudo_product}
Let $X,Y,Z$ be three $\infty$-categories, $M\subset X_{>0}$, $N\subset Y_{>0}$ and $P\subset Z_{>0}$. Then we have
\[ \begin{array}{rcl}
\overline{(M\ltens N)\ltens P}  &= & \overline{M\ltens (N\ltens P)}\\
\overline{(M\ptens N)\ptens P}  &= & \overline{M\ptens (N\ptens P)}
\end{array} \]
\end{lemma}

\begin{proof}
We begin with the first equality. Let
\[ E := \left( M \otimes Y_{\geqslant0 }\otimes Z_{\geqslant0 } \right) \cup \left( X_{\geqslant 0} \otimes N \otimes  Z_{\geqslant0 } \right) \cup \left( X_{\geqslant0 } \otimes Y_{\geqslant0 }  \otimes  P \right). \]
The lemmas \ref{lem:generators_of_gray_tensor}, \ref{lem:compatibility_of_coproduct_and_saturation}, and \ref{lem:compatibility_of_gray_tensor_and_saturation} imply the following equalities:

\[\begin{array}{rcl}
\overline{E} & = & \overline{ M \otimes \overline{ Y_{\geqslant0 }\otimes Z_{\geqslant0 }} \cup  X_{\geqslant 0} \otimes \left( N \otimes  Z_{\geqslant0 }  \cup  Y_{\geqslant0 }  \otimes  P \right) } \\
& = & \overline{ M \otimes ( Y \otimes Z)_{\geqslant 0} \cup X_{\geqslant 0} \otimes (N \ltens P)} \\
& = & \overline{ M \ltens (N \ltens P) }  
\end{array}
\]
A very similar computation also shows that $\overline{E}  = \overline{(M\ltens N)\ltens P}$, which concludes the proof of the first equality.

For the second equality, we define
\[ F := \left( X_{\geqslant 0} \otimes Y_{>0} \otimes Z_{>0} \right) \cup \left( X_{> 0} \otimes Y_{\geqslant 0} \otimes Z_{>0} \right) \cup \left( X_{> 0} \otimes Y_{>0} \otimes Z_{\geqslant 0} \right) \]
The second equality of \cref{lem:generators_of_gray_tensor} implies that:

\[ \overline{F} = \overline{  X_{ \geqslant 0} \otimes Y_{>0} \otimes Z_{>0}  \cup  X_{> 0} \otimes (Y\otimes Z)_{>0}} \]
and then that
\[\begin{array}{rcl}
\overline{E \cup F} &=& \overline{M \otimes (Y\otimes Z)_{\geqslant 0} \cup X_{\geqslant 0} \otimes  (N\ptens P) \cup X_{>0} \otimes (Y\otimes Z)_{>0}}\\
&=&  \overline{M\ptens (N\ptens P)}
\end{array}
\]
and here again, a similar computation shows $\overline{E\cup F}  = \overline{(M\ptens N)\ptens P}$, which concludes the proof.
\end{proof}
\begin{lemma}
\label{lem:unit_axiom_of_gray_tensor}
Let $X$ be an $\infty$-category, $M\subset X_{>0}$. Then the empty set, considered as a subset of the $\infty$-category $\Db_0$, satisfies (up to the identifications $\Db_0 \otimes X \simeq X \otimes \Db_0 \simeq X$):
\[
\begin{array}{rcl}
\emptyset \ltens M = M\ltens \emptyset  = M\\
\overline{\emptyset \ptens M} = \overline{M\ptens \emptyset} = \overline{M}
\end{array} \]
\end{lemma}
\begin{proof}
The first equality is a straightforward application of the definition of $\ltens$. For the second case, we also use the fact that all arrows of $(\Db_0)_{>0} \otimes X_{>0}$ are identities and so all belong to $\overline{M}$.
\end{proof}

\begin{prop}
Both the lax-Gray tensor product $\ltens$ and the pseudo-Gray tensor product $\ptens$, as defined above, are monoidal structures on the category of $m$-marked $\infty$-categories. In both cases, the forgetful functor to $\infty$-categories is monoidal, and their unit is $\Db_0^\flat = \Db_0^\#$. 
\end{prop}

\begin{proof} 
Note that $\Db_0^\flat = \Db_0^\# = (\Db_0,\overline{\emptyset})$ as all arrows of $\Db_0$ of dimension strictly superior to $0$ are identities.

The proposition states that the structural maps (associativity and unit isomorphisms) of the Gray tensor product of $\infty$-categories preserve the marking we specified on the tensor product.

For the unit, let $(X,M)$ be an $m$-marked $\infty$-category. The Lemmas \ref{lem:compatibility_of_gray_tensor_and_saturation} and \ref{lem:unit_axiom_of_gray_tensor} imply that
\[
\begin{array}{rclrrrr}
(X,M) \ltens (\Db_0,\overline{\emptyset}) &=& (X \otimes \Db_0 ,\overline{M\ltens \emptyset}) & =& (X,M)\\
(X,M) \ptens (\Db_0,\overline{\emptyset}) &=& (X \otimes \Db_0 ,\overline{M\ptens \emptyset}) &=& (X,M)
\end{array} \]
and
\[
\begin{array}{rclrrrr}
(\Db_0,\overline{\emptyset})\ltens (X,M)  &=& (\Db_0 \otimes X ,\overline{\emptyset \ltens M}) &=& (X,M)\\
(\Db_0,\overline{\emptyset})\ptens (X,M) &=& (\Db_0 \otimes X ,\overline{ \emptyset\ptens M}) &=& (X,M)
\end{array}\]
For the associativity isomorphism, let $(X,M)$, $(Y,N)$, and $(Z,P)$ be three marked $\infty$-categories.
\Cref{lem:compatibility_of_gray_tensor_and_saturation} implies that
\[\begin{array}{rcl}
\big((X,M) \ltens (Y,N)\big)\ltens (Z,P) &=& (X\otimes Y \otimes Z, \overline{(M\ltens N) \ltens P})\\
\big((X,M) \ptens (Y,N)\big)\ltens (Z,P) &=& (X\otimes Y \otimes Z, \overline{(M\ptens N) \ltens P})\\
\end{array}\]
and
\[\begin{array}{rcl}
(X,M) \ltens \big((Y,N)\ltens (Z,P)\big) &=& (X\otimes Y \otimes Z, \overline{M\ltens (N \ltens P)})\\
(X,M) \ptens \big((Y,N)\ltens (Z,P)\big) &=& (X\otimes Y \otimes Z, \overline{M\ptens (N \ltens P)}).\\
\end{array}\]

\Cref{lem:associativity_of_lax_and_pseudo_product} shows that these two markings on $X \otimes Y \otimes Z$, in the lax and pseudo cases, coincide.
\end{proof}

\begin{prop}
The pseudo and lax-Gray tensor products $\ltens$ and $\ptens$ preserve colimits in each variable.
\end{prop}

\begin{proof}
It follows from the fact that the Gray tensor product $\otimes$ preserves colimits in each variable, the description of colimits of $m$-marked $\infty$-categories given in \cref{cstr:marked_colimits}, and \cref{lem:compatibility_of_gray_tensor_and_saturation}.
\end{proof}

\change{
\begin{remark}
Remark \ref{rem:micat localement presentable} states that $\micat$ is locally presentable. Consequently, the preceding proposition implies that the functors $C\ltens\uvar$, $\uvar\ltens C$, $C\ptens\uvar$, and $\uvar\ptens C$ admit right adjoints. In particular, this immediately implies that both tensor products are closed monoidal structures.
\end{remark}}

\subsection{The Inductive Left Semi-Model Structure}
\label{subsec:Semi_model}

\change{
In this section, we will construct a left semi-model structure on the category $\micat$. The definitions and results on left semi-model structures that we will use here are recalled in \cref{sec:semi-model-categories}.}

\begin{definition}
\label{def:generating_cofibrations}
We define the set $I = I^{\partial} \cup I^{+m}$ to be our \textit{set of generating cofibrations} in $\micat$ where:
\[ I^{\partial} = \{ i_n : \partial \Db_n^\flat \to \Db_n^\flat \mid n \geqslant 0 \} \]
\[ I^{+m} = \{ \Db_n^\flat \to (\Db_n,\overline{\{e_n\}}) \mid n \geqslant 0 \} \]
An arrow in $\micat$ is said to be an \textit{acyclic fibration} if it has the right lifting property against all arrows in $I$. An arrow in $\micat$ is said to be a \textit{cofibration} if it has the left lifting property against all acyclic fibrations.
\end{definition}

\begin{remark}
It immediately follows from the small object argument that every morphism can be factored into a cofibration followed by an acyclic fibration, and that all cofibrations are retracts of transfinite compositions of pushouts of morphisms in $I$.
\end{remark}

\change{
\begin{prop}
\label{prop:cofibration of the marked model structure}
A morphism $(K,M) \to (L,N)$ is a cofibration in $\micat$ if and only if the induced functor $K \to L$ is a cofibration in the canonical model structure $\icat_\can$ recalled in \cref{theo:canonical_model_structure}.

In particular, the cofibrant objects of $\micat$ are exactly the $m$-marked $\infty$-categories whose underlying $\infty$-category is free on a polygraph, with any possible marking on them. 
\end{prop}

\begin{proof}
As recalled in \cref{theo:canonical_model_structure}, the set of generating cofibrations of the canonical model structure is given by $\{i_n:\partial \Db_n\to \Db_n \mid n \geqslant 0\}$. Note that the trivial marking functor $(\uvar)^\flat:\icat\to \micat$ and the forgetful functor $U:\micat\to \icat$ preserve colimits. We can directly deduce that both of these functors preserve cofibrations. 

In particular, a cofibration $(K,M) \to (L,N)$ induces a cofibration $K \to L$ in $\icat_{\can}$.

Conversely, suppose we are given a morphism $(K,M) \to (L,N)$ such that the induced morphism $K \to L$ is a cofibration in $\icat_{\can}$. We have a canonical square:
\[
\begin{tikzcd}[ampersand replacement=\&]
	{K^\flat} \& {(K,M)} \\
	{L^\flat} \& {(L,N)}
	\arrow[from=1-1, to=1-2]
	\arrow[from=1-1, to=2-1]
	\arrow[from=1-2, to=2-2]
	\arrow[from=2-1, to=2-2]
\end{tikzcd}
\]
where the left-hand vertical morphism is a cofibration. The canonical morphism $L^\flat\coprod_{K^\flat}(K,M)\to (L,N)$ is the identity on the underlying category and is thus an iterated pushout of morphisms in $I^{+m}$. In particular, it is a cofibration, and by stability under pushouts and compositions, so is $(K,M) \to (L,N)$.

Finally, the last claim follows from \cite[Theorem 7.4]{metayer2008cofibrant}, which asserts that cofibrant objects of $\icat_{\can}$ correspond to $\infty$-categories that are free on a polygraph.
\end{proof}
}

\begin{remark}
A morphism $\pi: X \to Y$ has the right lifting property against all morphisms in $I^\partial$ if its image by the forgetful functor to $\icat$ is an acyclic fibration; that is, if for every pair of parallel $n$-arrows $u, v$ in $X$, the map $\Hom_X(u,v) \to \Hom_Y(\pi(u),\pi(v))$ is surjective.

$\pi$ has the right lifting property against all morphisms in $I^{+m}$ if and only if for every arrow $f \in X$ such that $\pi(f)$ is marked in $Y$, $f$ is marked in $X$. An acyclic fibration is a map that has both these properties.
\end{remark}

The pushout-product, or corner-product (sometimes also called the Leibniz product) $f \hatltens g$ and $f \hatptens g$ is defined as usual: if $f :X \to Y$ and $g: A \to B$ are two morphisms in $\micat$, then $f \hatltens g$ is the canonical morphism:
\[ X \ltens B \coprod_{X \ltens A} Y \ltens A \to Y \ltens B \]
and $f \hatptens g$ is the canonical morphism
\[ X \ptens B \coprod_{X \ptens A} Y \ptens A \to Y \ptens B \]
We refer to the appendix of \cite{joyal2007quasi} for the general theory of pushout products and their formal properties.

\begin{prop}\label{prop:pushout_product_cofibrations} If $f$ and $g$ are two cofibrations in $\micat$, then $f \hatltens g$ and $f \hatptens g$ are both cofibrations. \end{prop}

\begin{proof}
 By the usual properties of the corner-product, it is enough to check this when $f$ and $g$ are generating cofibrations. If $f$ and $g$ are both in $I^\partial$, then $f \ltens g$ has no marked arrows in either its domain or codomain and coincides with the corner-product $f \hatotimes g$ in $\icat$, which is a cofibration by \cite[theorem 3.9]{ara2020folk}. $f \ptens g$ is the same except that some arrows are marked, but we can always add these markings by taking additional pushouts by morphisms in $I^{+m}$, so it is again a cofibration.

The forgetful functor $\micat \to \icat$ is monoidal for both tensor products and preserves colimits, so it preserves the corner-product. In particular, if either $f$ or $g$ is in $I^{+m}$, then it is sent to isomorphisms by this forgetful functor, and hence $f \hatltens g$ and $f \hatptens g$ induce isomorphisms between their underlying $\infty$-categories. Now, if $f:(X, N) \to (X, M)$ is a morphism in $\micat$ that induces an isomorphism on underlying $\infty$-categories, then it is a pushout of morphisms in $I^{+m}$: one simply needs to take such pushouts to make all arrows in $M$ marked.
\end{proof}

\begin{construction} \label{cstr:interval} We define $I := \Db_1^\sharp = (\Db_1,\{e_1\})$. It is the $\infty$-category with two objects, $e_0^-$ and $e_0^+$, and a marked arrow $e_1:e_0^- \to e_0^+$. We denote by $j_-$ and $j_+$ the two maps $\Db_0 \to I$ corresponding, respectively, to the two objects $e_0^-$ and $e_0^+$. This gives a diagram:
\[ \Db_0 \coprod \Db_0 \cto I \to \Db_0 \]
which will play the role of the interval object for our left semi-model structure on $\micat$.
\end{construction}

We will take as a set of ``generating anodyne cofibrations'' (also called a ``pseudo-generating set of acyclic cofibrations'') the set of maps of the form $j_+ \hatptens i$ where $i$ is a generating cofibration, more precisely:

\begin{definition} 

\label{def:generating_anodyne} $ $
\begin{itemize}
\item \change{We say that a morphism is a \textit{generating anodyne cofibration} if it is of the form $j_+ \hatptens i$ with $i$ a generating cofibration.}
\item We say that a morphism in $\micat$ is a \emph{naive fibration} if it has the right lifting property against all morphisms of the form $ j_+ \hatptens i$, where $j_+: \Db_0 \to I$ is as in \cref{cstr:interval}, and $i$ is one of the generating cofibrations as in \cref{def:generating_cofibrations}.

\item \change{We say that an $m$-marked $\infty$-category $C$ is \textit{fibrant} if the morphism $C \to 1$ is a naive fibration.}
\item We say that a morphism in $\micat$ is an \emph{anodyne cofibration} if it has the right lifting property against all naive fibrations.

\item We say that a cofibration in $\micat$ is \emph{acyclic} if it has the lifting property against all naive fibrations between fibrant objects.

\item We say that a map in $\micat$ is a \emph{fibration} if it has the right lifting property against all acyclic cofibrations.

\end{itemize}

\end{definition}

As before, it immediately follows from the small object argument that every morphism factors as an anodyne cofibration followed by a naive fibration, and all anodyne cofibrations are retracts of transfinite compositions of pushouts of the generating anodyne cofibrations.

\begin{remark}
It immediately follows from \cref{prop:pushout_product_cofibrations} that, as $j_+$ is a cofibration, all maps of the form $ j_+ \hatptens i$ are cofibrations. In particular, all acyclic fibrations are also naive fibrations and all anodyne cofibrations are cofibrations.
\end{remark}

\begin{prop}
Acyclic cofibrations and fibrations form a cofibrantly generated weak factorization system on $\micat$. A morphism with fibrant target is a fibration if and only if it is a naive fibration.
\end{prop}

\change{
\begin{proof}
This is a direct application of the results of Section 4 of \cite{henry2020combinatorial}. Starting from the premodel (see \cref{def:premodel}) structure on $\micat$ whose weak factorization systems are (cofibrations, acyclic fibrations) and (anodyne cofibrations, naive fibrations), we obtain the one with (cofibrations, acyclic fibrations) and (acyclic cofibrations, fibrations) as its ``left saturation'' $\Lb (\micat)$ in the sense of Theorem 4.1 of \cite{henry2020combinatorial}. All the claims in the proposition follow from this Theorem 4.1.
\end{proof}
}

\begin{remark}\label{rk:choice_def_anodyne}
Note that replacing $\hatptens$ by $\hatltens$ in \cref{def:generating_anodyne} would not change the definition. Indeed, if $X = Y^\sharp$ is an $m$-marked $\infty$-category whose arrows of dimension strictly greater than $0$ are all marked, then for any $m$-marked $\infty$-category $Z$ one has $X \ptens Z = X \ltens Z$. As this applies to both the domain and the co-domain of $j_+$, it follows that $j_+ \hatptens i = j_+ \hatltens i$.

Also, the reader should not be worried about the use of $j_+$ in \cref{def:generating_anodyne} rather than $j_-$ or both $j_-$ and $j_+$. While using $j_-$ or both $j_-$ and $j_+$ instead of $j_+$ would change the definition of naive fibrations and anodyne cofibrations, this does not affect the definition of (naive) fibrations between fibrant objects; hence, the acyclic cofibrations and fibrations would not be changed. Indeed, once the existence of a (monoidal) model structure is established, it follows that $j_-$ is acyclic by $2$-out-of-$3$, and hence all the maps $j_- \hatptens i = j_- \hatltens i$ are also acyclic cofibrations.
\end{remark}

\begin{lemma}\label{lem:acyclic_cof_left_pushout_product}
If $f$ is an anodyne (resp. acyclic) cofibration and $g$ is a cofibration, then $f \hatptens g$ and $f \hatltens g$ are anodyne (resp. acyclic).
\end{lemma}

\begin{proof}
To get the result for ``anodyne cofibrations,'' it is enough to prove it for the generating anodyne cofibrations. Let $i$ be one of the generating cofibrations and $f = j_+ \hatptens i'$ be one of the generating anodyne cofibrations. We have $f \hatptens i = j_+ \hatptens (i \hatptens i')$. As $i' \hatptens i$ is a pushout of generating cofibrations $i_1, \dots, i_k$ by \cref{prop:pushout_product_cofibrations}, it follows that $j_+ \hatptens (i \hatptens i')$ is a pushout of the $j_+ \hatptens i_k$ and hence is an anodyne cofibration.

The result for acyclic cofibrations follows from the formal properties of the pushout product: it follows that if $i$ is a cofibration and $p$ is a naive fibration, then the (right) pullback exponential $\langle p / i \rangle$ is a naive fibration. If $p$ is a (naive) fibration between fibrant objects, then $\langle p / i \rangle$ is a naive fibration between fibrant objects, hence a fibration. It follows that if $i$ is an acyclic cofibration and $j$ is a cofibration, then $i \hatptens j$ is an acyclic cofibration as it is a cofibration by \cref{def:generating_cofibrations}, and if $p$ is a fibration between fibrant objects, then $i \hatptens j$ has the right lifting property against $p$ because $j$ has the left lifting property against $\langle p / i \rangle$.

The case of $\ltens$ works exactly the same considering the first half of \cref{rk:choice_def_anodyne}.
\end{proof}

\begin{theorem}\label{th:canonical_left_semi_model}
The category $\micat$ of $m$-marked $\infty$-categories admits a $\omega$-combinatorial left semi-model structure (\cref{defi:basic def on left semi-model structure}), called the \emph{inductive model structure} and denoted by $\micat_\ind$, in which the cofibrations and acyclic fibrations are as in \cref{def:generating_cofibrations} and the fibrations are as in \cref{def:generating_anodyne}. Moreover, this left semi-model structure is monoidal (\cref{defi:basic def on left semi-model structure}) for both tensor products $\ptens$ and $\ltens$ (from \cref{subsec:Marked_Gray}).
\end{theorem}

\begin{proof}
The existence of the left semi-model structure immediately follows from Theorem 6.12 of \cite{henry2020combinatorial}. Because of \cref{prop:pushout_product_cofibrations} and \cref{lem:acyclic_cof_left_pushout_product}, tensoring by the interval object $I$ of \cref{cstr:interval} is a ``strong Quillen functor'' in the sense of Section 6 of \cite{henry2020combinatorial}. Note that to apply Theorem 6.12, one needs to observe that $\micat$, with the (cofibrations, acyclic fibrations) and (acyclic cofibrations, fibrations) weak factorization systems, is both ``right saturated'' and ``left saturated'', that is, that a fibration that has the right lifting property against all cofibrations between cofibrant objects is an acyclic fibration, and that a cofibration that has the left lifting property against all fibrations between fibrant objects is an acyclic cofibration. The first one holds because the generating cofibrations are cofibrations between cofibrant objects, and the second because that is how we defined acyclic fibrations.

As $\micat$ is finitely locally presentable, and as the codomains of the generating cofibrations and generating anodyne cofibrations are $\omega$-small, Theorem 4.1 of \cite{henry2020combinatorial} implies that $\micat_{\ind}$ is $\omega$-combinatorial.

The fact that this left semi-model structure is monoidal directly follows from \cref{prop:pushout_product_cofibrations} and \cref{lem:acyclic_cof_left_pushout_product}.
\end{proof}

\begin{remark}
The proof of \cref{th:canonical_left_semi_model} above also shows that $\micat$ also admits a right  semi-model category structure whose fibrations and acyclic cofibrations are the fibrations and acyclic cofibrations of \cref{def:generating_anodyne} and whose cofibrations are as in \cref{def:generating_cofibrations}.

This, however, does not clearly make $\micat$ into a Quillen model structure but rather into a ``two-sided model category'' as in Section 5 of \cite{henry2020combinatorial}. We refer to Section 5 of \cite{henry2020combinatorial} for what this means more precisely, but in short, the problem is that the left and right left semi-model categories have different classes of weak equivalences. The two classes of equivalence, however, coincide for morphisms that are between fibrant or cofibrant objects. Another way to talk about this difference is that the left and right left semi-model categories are Quillen equivalent and have the same homotopy category but define different functors $\micat \to \text{Ho}(\micat)$. The two functors agree on objects that are either fibrant or cofibrant but differ on general objects: one sends an object $X$ to its cofibrant replacement while the other sends it to a fibrant replacement, and we do not know if these are always homotopy equivalent when $X$ is neither fibrant nor cofibrant itself.
\end{remark}

\begin{remark}
We do not know if $\micat$ is actually a Quillen model category or not. In the unmarked case, this follows from the fact that all objects are fibrant. But that is no longer the case in this situation. In terms of the ``two-sided model structure'' mentioned in the previous remark, the question is whether $\micat$ satisfies one of the equivalent conditions of Proposition 5.3 of \cite{henry2020combinatorial}.
\end{remark}

We conclude this section with the following lemma that will be useful later:

\begin{lemma}\label{lem:hemisphere_are_anodyne}
The map
\[ i^{+}_n : \Db_n^\flat \to (\Db_{n+1}, \overline{\{e_{n+1}\}}) \]
where $e_{n+1}$ is the unique non-identity arrow of $\Db_{n+1}$, is an anodyne cofibration.
\end{lemma}

\begin{proof}
\change{
We will show it is a retract of the map $j_+ \hatptens i_n$ where $i_n$ is the map $\partial \Db_n \to \Db_n$. We then have to construct two morphisms $i$, $p$ fitting in a diagram of the form
\[\begin{tikzcd}[ampersand replacement=\&]
	{(\Db_{n+1},\overline{\{e_{n+1}\}})} \& {I \ptens \Db_n^\flat} \\
	\& {(\Db_{n+1},\overline{\{e_{n+1}\}})}
	\arrow["i", from=1-1, to=1-2]
	\arrow["p", from=1-2, to=2-2]
	\arrow[Rightarrow, no head, from=1-1, to=2-2]
\end{tikzcd}\]
and such that $p$ and $i$ send the domain of $i^+_n$ and of $j_+ \hatptens i_n$ to each other.

In order to achieve this, we will use the explicit description of $\Db_1 \otimes \Db_n$ given in \cref{example:of_gray_cylinder}. The object we are interested in is $I \ptens \Db_n^\flat$ which is the same polygraph endowed with the marking where all the arrows $a \otimes e^\epsilon_k$ are marked. }We call $i : (\Db_{n+1},\overline{\{e_{n+1}\}}) \to I \ptens \Db_n^\flat$ the unique morphism sending $e_{n+1}$ to $a \otimes e_n$. This is well defined because $a \otimes e_n$ is a marked arrow. Next, we define a map $p: I \ptens \Db_n^\flat \to (\Db_{n+1},\overline{\{e_{n+1}\}})$ by:
   \[ p(a^\epsilon_0 \otimes e^\mu_k) = e^\mu_k  \text{ if } k<n. \]
   \[ p(a^\epsilon_0 \otimes e_n) = e^\epsilon_n \]
   \[ p(a \otimes e^\epsilon_k) = \Ib_{e^\epsilon_k} \text{ if } k<n. \]
   \[ p(a \otimes e_n) = e_{n+1} \]
   In order to check that this is well defined, we first need to check that this definition is compatible with the source and target given above, which follows from an immediate calculation. Then we need to show that this is compatible with the marking, which is the case as both $\Ib_{e^\epsilon_k}$ and $e_{n+1}$ are marked.

   Finally, the composite $ p \circ i$ sends the arrow $e_{n+1}$ to $p(a \otimes e_n) = e_{n+1}$ and hence is the identity of $\Db_{n+1}$.

   To conclude the proof, we just have to observe that the maps $p$ and $i$ defined above send the domain of $i^+_n$ and of $j_+ \hatptens i_n$ to each other.

   The domain of $j_+ \hatptens i_n$ is the sub-polygraph of $I \ptens \Db_n^\flat$ which contains all the generators except  $a_0^- \otimes e_n$ and $a \otimes e_n$, while the domain of $i^+_n$ contains all generators of $\Db_{n+1}$ except $e_{n+1}$ and $e^-_n$.

   In order to check that the map $i$ is compatible with these sub-polygraphs, it is enough to check that $i(e^+_n)$ is in the domain of $j_+ \hatptens i_n$. To see this, we compute:
  \[ i(e^+_n) = \pi^+ i(e_{n+1}) = \pi^+(a \otimes e_n) = (a \otimes e^-_{n-1}) \#_{n-1} \dots \#_1 (a \otimes e^-_0) \#_0 (a^+_0 \otimes e_n) \]
   and we observe that this expression involves neither $a_0^- \otimes e_n$ nor $a \otimes e_n$, hence it does belong to the domain of $j_+ \hatptens i_n$.

   In order to check that the map $p$ is compatible with these sub-polygraphs, we need to check the image by $p$ of all the generators of $I \ptens \Db_n^\flat$ except  $a_0^- \otimes e_n$ and $a \otimes e_n$. These are given by the formulas $p(a^\epsilon_0 \otimes e^\mu_k) = e^\mu_k$ if $k<n$, $p(a^+_0 \otimes e_n) = e^+_n$ and $p(a \otimes e^\epsilon_k) = \Ib_{e^\epsilon_k}$, which all indeed belong to the image of $i_n^+$.

\end{proof}

\section{Equations and Saturations in an $m$-Marked $\infty$-Category.}
\label{sec:eq_saturation}The general goal of this section is to arrive at a better description of the fibrant objects and fibrations between fibrant objects of the model structure of \cref{th:canonical_left_semi_model}. This is achieved using the notion of \emph{equations} in an $\infty$-category introduced by the second named author in \cite{loubaton2021conditions}. We will recall the basic theory of equations, in a slightly different language, and introduce an analog of equations to deal with the markings, which we call \emph{saturations}.

\subsection{Definitions of Equations and Saturations}
\label{subsec:eq_saturation}

\begin{definition} 
\label{defi:equation_in_a_marked_category}
\change{
A morphism of $m$-marked polygraphs $\Lambda P \to P$ is a \emph{left equation} if there exists an integer $n$, and two generators $x, y$ of $P$ of dimension respectively $n$ and $n+1$, such that}
\begin{enumerate}

\item \change{$\Lambda P$ is the  $m$-marked sub-polygraph of $P$ that contains all generators except $x$ and $y$,}

\item $y$ is a marked arrow,

\item if $n \leq m$, $x$ is an unmarked arrow of $P$,

\item \label{defi:equation_in_a_marked_category:decomp} the source of $y$ admits a decomposition:
\[ \pi_n^- y = l_n \#_{n-1} (l_{n-1} \#_{n-2} \dots \#_{1} (l_1 \#_0 x \#_0 r_1) \#_{1} \dots \#_{n-2} r_{n-1}) \#_{n-1} r_n \]
where for each $i$, $l_i$ and $r_i$ are marked $i$-arrows in $P$, with $l_n$ and $r_n$ not containing $x$. In particular, $x$ appears only once in $\pi_n^- y$,

\item $x$ does not appear in the target of $y$.

\end{enumerate}

\emph{Right equations} are defined in the exact same way except the source and target of $y$ are exchanged in the last two conditions.

We say that $\Lambda P \to P$ is an \textit{equation} to mean that it is either a left or right equation.

\end{definition}

\begin{remark}
\change{Note that the integer $n$ and the arrows $x$ and $y$ in the previous definition are uniquely determined by the inclusion $\Lambda P \to P$.}
\end{remark}

\begin{remark}
The name ``equation'' comes from the idea that we are looking for an element $x$ such that a certain composite of $x$ with other arrows is isomorphic to another given arrow. From this point of view, a map $\Lambda P \to X$ corresponds to such an equation in $X$, and an extension $P \to X$ corresponds to a solution of the equation, or rather the image of $x$ is the solution and $y$ represents the isomorphism witnessing that $x$ is a solution.
\end{remark}

\change{
\begin{definition}\label{defi:saturation_in_a_marked_category}
A morphism of  $m$-marked polygraphs $\Omega P \to P$ is a \emph{left saturation} if it is an isomorphism on the underlying polygraphs, and if there exists an integer $n$, and two marked generators $x, y$ of $P$ of dimension respectively $n$ and $n+1$, such that 
\begin{enumerate}
\item any marked generator of $P$ that is different from $x$ is marked in $\Omega P$,
\item $x$ and $y$ are marked in $P$,
\item the target of $y$ is marked,
\item the arrows $x$ and $y$ satisfy the conditions ($5$) and ($6$) of \cref{defi:equation_in_a_marked_category}.
\end{enumerate}

\emph{Right saturations} are defined in the exact same way except the source and target of $y$ are exchanged in the last two conditions.

We say that $\Omega P \to P$ is an \textit{saturation} to mean that it is either a left or right saturation.
\end{definition}
}

\begin{construction}
\label{rem:generating_anodyne_cofiration_seen_as_an_equation}
Let $n$ be a non-negative integer. The morphism
\[ j_+ \hatptens i_n := I \ptens \partial \Db_n \coprod \{1\} \ptens \Db_n \to I \ptens \Db_n \]
is a left equation. Indeed, let $y$ be the top-dimensional generator of $I \ptens \Db_n$. If we denote by $x$ the top-dimensional arrow of $\{0\} \ptens \Db_n$, and for $0 < k \leq n$, by
$a_k$ the image of the top-dimensional $k$-generator of $I \ptens \Db_{k-1}$ by the morphism
\[ I \ptens \delta^-_{k-1} : I \ptens \Db_{k-1} \to I \ptens \Db_n, \]
we recall that we gave an explicit description of $\Db_1 \otimes \Db_n$ in \cref{example:of_gray_cylinder}. The object we are interested in is $I \ptens \Db_n^\flat$, which is the same polygraph endowed with the marking where all the arrows $a \otimes e^\epsilon_k$ are marked.
Using this description, we see that if we name $y = a \otimes e_n$ and $x = a_0^- \otimes e_n$ the two arrows of $I \ptens \Db_n$ that are not in the image of $j_+ \hatptens i_n$, then we have a decomposition of the source of $y$ of the form:
\[ (((x \#_0 a_0) \#_1 a_2) \dots) \#_{n-1} a_n \]
and all the $a_k$ are marked.
We denote it
\[ \cylinder{\textbf{eq}}{n} \colon \cylinder{\Lambda\textbf{Eq}}{n} \to \cylinder{\textbf{Eq}}{n}. \]
\end{construction}

\change{
\begin{example}
The underlying $\infty$-category of $\cylinder{\textbf{Eq}}{1}$ is 
\[\begin{tikzcd}[ampersand replacement=\&]
	\bullet \& \bullet \\
	\bullet \& \bullet
	\arrow[from=1-1, to=2-1]
	\arrow[from=2-1, to=2-2]
	\arrow["x", from=1-1, to=1-2]
	\arrow["{a_0}", from=1-2, to=2-2]
	\arrow["y"', shorten <=4pt, shorten >=4pt, Rightarrow, from=1-2, to=2-1]
\end{tikzcd}\]
and the underlying $\infty$-category of $\cylinder{\textbf{Eq}}{2}$ is 
\[\begin{tikzcd}[ampersand replacement=\&]
	\bullet \& \bullet \&\& \bullet \& \bullet \\
	\bullet \& \bullet \&\& \bullet \& \bullet
	\arrow[from=1-1, to=2-1]
	\arrow[from=2-1, to=2-2]
	\arrow[""{name=0, anchor=center, inner sep=0}, curve={height=-18pt}, from=1-1, to=1-2]
	\arrow[""{name=1, anchor=center, inner sep=0}, "{a_0}", from=1-2, to=2-2]
	\arrow["{a_1}"', shorten <=4pt, shorten >=4pt, Rightarrow, from=1-2, to=2-1]
	\arrow[""{name=2, anchor=center, inner sep=0}, from=1-1, to=1-2]
	\arrow[""{name=3, anchor=center, inner sep=0}, from=1-4, to=2-4]
	\arrow[""{name=4, anchor=center, inner sep=0}, from=2-4, to=2-5]
	\arrow[from=1-4, to=1-5]
	\arrow[from=1-5, to=2-5]
	\arrow[""{name=5, anchor=center, inner sep=0}, curve={height=18pt}, from=2-4, to=2-5]
	\arrow[shorten <=4pt, shorten >=4pt, Rightarrow, from=1-5, to=2-4]
	\arrow["x", shorten <=2pt, shorten >=2pt, Rightarrow, from=0, to=2]
	\arrow[shorten <=2pt, shorten >=2pt, Rightarrow, from=4, to=5]
	\arrow["y", shorten <=13pt, shorten >=13pt, Rightarrow, scaling nfold=3, from=1, to=3]
\end{tikzcd}\]
\end{example}
}
\begin{construction}
\label{rem:generating_anodyne_cofiration_seen_as_a_saturation}
Similarly, the morphism
\[ j_+ \hatptens s_n : I \ptens \Db_n \coprod \{1\} \ptens (\Db_{n}, \overline{\{e_n\}}) \to I \ptens (\Db_{n}, \overline{\{e_n\}}) \]
where $s_n$ is the ``identity'' map $\Db_n \to (\Db_n, \overline{\{e_n\}})$
is a left saturation, which we denote
\[\cylinder{\textbf{sat}}{n} \colon \cylinder{\Omega\textbf{Sat}}{n} \to \cylinder{\textbf{Sat}}{n}. \]
\end{construction}

\change{
\begin{prop}
\label{prop:generating_anodyne_cof_are_equations_and_saturations}
Generating anodyne cofibrations are either equations or saturations.
\end{prop}

\begin{proof}
By \cref{def:generating_anodyne}, the generating anodyne cofibrations are of the form $j_+ \hatptens i$ with $i$ being either $\partial \Db_n \to \Db_n$ or $\Db_n \to (\Db_{n}, \overline{\{e_n\}})$ for an integer $n$. By Constructions \ref{rem:generating_anodyne_cofiration_seen_as_an_equation} and \ref{rem:generating_anodyne_cofiration_seen_as_a_saturation}, these morphisms are either equations or saturations.
\end{proof}}

\change{
\begin{definition}
We define some left equations which play an important role. In each case, $k$ and $n$ are integers with $0 < k \leqslant n$.
\begin{itemize}
\item $\leftdivision{\textbf{eq}}{k,n} \colon \leftdivision{\Lambda\textbf{Eq}}{k,n} \to \leftdivision{\textbf{Eq}}{k,n}$, whose codomain is generated by
\begin{enumerate}
\item[$-$] a $n$-arrow $x$, a marked $k$-arrow $a$ such that $\pi_{k-1}^+(a) = \pi_{k-1}^-(x)$,
\item[$-$] a $n$-arrow $b$ of source $a \#_{k-1} \pi^{-}_{n-1}(y)$ (resp. $\pi^-_{n-1}(a)$) and of target $a \#_{k-1} \pi^{+}_{n-1}(y)$ (resp. $\pi^+_{n-1}(x)$) if $k < n$ (resp. if $k = n$),
\item[$-$] a marked $(n+1)$-arrow $y$ of source $a \#_{k-1} x$ and of target $b$,
\end{enumerate}
and whose domain is obtained by removing $x$ and $y$.

\item $\rightdivision{\textbf{eq}}{k,n} \colon \rightdivision{\Lambda\textbf{Eq}}{k,n} \to \rightdivision{\textbf{Eq}}{k,n}$, whose codomain is generated by
\begin{enumerate}
\item[$-$] a $n$-arrow $x$, a marked $k$-arrow $a$ such that $\pi_{k-1}^+(x) = \pi_{k-1}^-(a)$,
\item[$-$] a $n$-arrow $b$ of source $\pi^{-}_{n-1}(y) \#_{k-1} a$ (resp. $\pi^-_{n-1}(x)$) and of target $\pi^{+}_{n-1}(y) \#_{k-1} a$ (resp. $\pi^+_{n-1}(a)$) if $k < n$ (resp. if $k = n$),
\item[$-$] a marked $(n+1)$-arrow $y$ of source $y \#_{k-1} x$ and of target $b$.
\end{enumerate}
and whose domain is obtained by removing $x$ and $y$.
\end{itemize}
\end{definition}

\begin{example}
The underlying $\infty$-category of $\leftdivision{\textbf{Eq}}{1,1}$ is generated by the diagram 
\[\begin{tikzcd}[ampersand replacement=\&]
	\& \bullet \\
	\bullet \&\& \bullet
	\arrow["a", from=2-1, to=1-2]
	\arrow["x", from=1-2, to=2-3]
	\arrow[""{name=0, anchor=center, inner sep=0}, "b"', from=2-1, to=2-3]
	\arrow["y"', shorten <=3pt, shorten >=3pt, Rightarrow, from=1-2, to=0]
\end{tikzcd}\]
and the underlying $\infty$-category of $\leftdivision{\textbf{Eq}}{1,2}$ is generated by the diagram 
\[\begin{tikzcd}[ampersand replacement=\&]
	\bullet \& \bullet \& \bullet \&\& \bullet \&\& \bullet
	\arrow[""{name=0, anchor=center, inner sep=0}, curve={height=-18pt}, from=1-2, to=1-3]
	\arrow[""{name=1, anchor=center, inner sep=0}, curve={height=18pt}, from=1-2, to=1-3]
	\arrow[""{name=2, anchor=center, inner sep=0}, "{a\#_1\pi^-x}", curve={height=-18pt}, from=1-5, to=1-7]
	\arrow[""{name=3, anchor=center, inner sep=0}, "{a\#_1\pi^+x}"', curve={height=18pt}, from=1-5, to=1-7]
	\arrow["y", shorten <=14pt, shorten >=14pt, Rightarrow, scaling nfold=3, from=1-3, to=1-5]
	\arrow["a", from=1-1, to=1-2]
	\arrow["x", shorten <=5pt, shorten >=5pt, Rightarrow, from=0, to=1]
	\arrow["y", shorten <=5pt, shorten >=5pt, Rightarrow, from=2, to=3]
\end{tikzcd}\]
\end{example}

\begin{definition}
Let $\Lambda P \to P$ be a left equation, and $n$, $x$, $y$ the integer and the two generators of \cref{defi:equation_in_a_marked_category}. We denote $(x_0, y_0)$ and $(x_1, y_1)$ the images of the couple $(x, y) \in P$ by the two inclusions $P \to P \coprod_{\Lambda P} P$.
The  $m$-marked polygraph $\Uni_{\Lambda P}(P)$ is obtained from $P \coprod_{\Lambda P} P$ by adding an unmarked $(n+1)$-generator $z$ of $n$-source $x_0$ and $n$-target $x_1$. 

A map $f : P \coprod_{\Lambda P} P \to X$ corresponds to a map $\Lambda P \to X$, which is an equation in $X$, together with two solutions $P \to X$, given by pairs $(x_0, y_0)$ and $(x_1, y_1)$. The morphism $f$ lifts to $\Uni_{\Lambda P}(P)$ if there exists a marked arrow $z : x_0 \to x_1$. Formally, this expresses that the two solutions are equivalent.
\end{definition}

\begin{example}
The underlying $\infty$-category of $\Uni_{\leftdivision{\Lambda\textbf{Eq}}{1,1}}(\leftdivision{\textbf{Eq}}{1,1})$ is 
\[\begin{tikzcd}[ampersand replacement=\&]
	\& \bullet \&\&\&\& \bullet \\
	\bullet \&\& \bullet \&\& \bullet \&\& \bullet
	\arrow[""{name=0, anchor=center, inner sep=0}, "b"', from=2-5, to=2-7]
	\arrow["a", from=2-5, to=1-6]
	\arrow["{x_0}", from=1-6, to=2-7]
	\arrow["a", from=2-1, to=1-2]
	\arrow[""{name=1, anchor=center, inner sep=0}, "{x_1}"{description}, from=1-2, to=2-3]
	\arrow[""{name=2, anchor=center, inner sep=0}, "b"', from=2-1, to=2-3]
	\arrow[""{name=3, anchor=center, inner sep=0}, "{x_0}", curve={height=-24pt}, from=1-2, to=2-3]
	\arrow["{y_0}", shorten <=3pt, shorten >=3pt, Rightarrow, from=1-6, to=0]
	\arrow["z"', shorten <=4pt, shorten >=4pt, Rightarrow, from=3, to=1]
	\arrow["{y_1}", shorten <=3pt, shorten >=3pt, Rightarrow, from=1-2, to=2]
\end{tikzcd}\]
\end{example}
}

\begin{definition}
Let $C$ be an $m$-marked $\infty$-category and $\Lambda P \to P$ a left equation (resp. right equation).

\emph{The equation $\Lambda P \to P$ has solutions in $C$} if for all morphisms $\Lambda P \to C$, there exists a lifting $(x,y): P \to C$ such that $x$ is sent to a marked arrow whenever the target of $y$ is (resp. the source of $y$ is).

\emph{Solutions to an equation $\Lambda P \to P$ in $C$ are weakly unique} if $C$ has the right lifting property against $P \coprod_{\Lambda P} P \to \Uni_{\Lambda P}(P)$.

\emph{The equation $\Lambda P \to P$ has weakly unique solutions in $C$} if the equation $\Lambda P \to P$ has solutions in $C$ and they are weakly unique.

\end{definition}

It will be useful to have a ``coherent'' version of $\Uni_{\Lambda P}(P)$.
\change{
\begin{definition}
Let $\Lambda P \to P$ be a left equation, and $n$, $x$, $y$ the integer and the two generators of \cref{defi:equation_in_a_marked_category}. Suppose given a decomposition
\[ \pi_n^- y = l_n \#_{n-1} (l_{n-1} \#_{n-2} \ldots \#_{1} (l_1 \#_0 x \#_0 r_1) \#_{1} \ldots \#_{n-2} r_{n-1}) \#_{n-1} r_n \]
of the $n$-source of $y$. 
We denote $(x_0, y_0)$ and $(x_1, y_1)$ the images of the couple $(x, y) \in P$ by the two inclusions $P \to P \coprod_{\Lambda P} P$.
The  $m$-marked polygraph $\Uni^{coh}_{\Lambda P}(P)$ is obtained from $P \coprod_{\Lambda P} P$ by
\begin{enumerate}
\item adding an unmarked $(n+1)$-generator $z$ of $n$-source $x_0$ and $n$-target $x_1$,
\item adding a marked $(n+2)$-generator $w$ of $(n+1)$-source
$$ l_n \#_{n-1} (l_{n-1} \#_{n-2} \ldots \#_{1} (l_1 \#_0 z \#_0 r_1) \#_{1} \ldots \#_{n-2} r_{n-1}) \#_{n-1} r_n \#_{n} y_1 $$
and of $(n+1)$-target $y_0$.
\end{enumerate}
By construction, the morphism $P \coprod_{\Lambda P} P \to \Uni^{coh}_{\Lambda P}(P)$ is a left equation.

Let $\Lambda P \to P$ be a right equation, and $n$, $x$, $y$ the integer and the two generators of \cref{defi:equation_in_a_marked_category}. Suppose given a decomposition
\[ \pi_n^+ y = l_n \#_{n-1} (l_{n-1} \#_{n-2} \ldots \#_{1} (l_1 \#_0 x \#_0 r_1) \#_{1} \ldots \#_{n-2} r_{n-1}) \#_{n-1} r_n \]
of the $n$-target of $y$. 
We denote $(x_0, y_0)$ and $(x_1, y_1)$ the images of the couple $(x, y) \in P$ by the two inclusions $P \to P \coprod_{\Lambda P} P$.
The $m$-marked polygraph $\Uni_{\Lambda P}(P)$ is obtained from $P \coprod_{\Lambda P} P$ by
\begin{enumerate}
\item adding an unmarked $(n+1)$-generator $z$ of $n$-source $x_0$ and $n$-target $x_1$,
\item adding a marked $(n+2)$-generator $w$ of $(n+1)$-source
$$ y_0 \#_{n} l_n \#_{n-1} (l_{n-1} \#_{n-2} \ldots \#_{1} (l_1 \#_0 z \#_0 r_1) \#_{1} \ldots \#_{n-2} r_{n-1}) \#_{n-1} r_n \to y_1 $$
and of $(n+1)$-target $y_0$.
\end{enumerate}
By construction, the morphism $P \coprod_{\Lambda P} P \to \Uni^{coh}_{\Lambda P}(P)$ is a right equation.
\end{definition}

\begin{remark}
Let $\Lambda P\to P$  be an equation and $X$ a $m$-marked $\infty$-category.
A map $f: P \coprod_{\Lambda P} P \to X$ corresponds to a map $\Lambda P \to X$, together with two solutions $P \to X$ given by pairs $(x_0, y_0)$ and $(x_1, y_1)$. If the equation $P \coprod_{\Lambda P} P \to \Uni^{coh}_{\Lambda P}(P)$ has a solution in $C$, it implies that given any pair of solutions $(x_0, y_0)$ and $(x_1, y_1)$ of $\Lambda P \to P$, there exists a marked arrow $z: x_0 \to x_1$, which informally expresses that the two solutions are equivalent, together with marked arrows
$$ l_n \#_{n-1} (l_{n-1} \#_{n-2} \ldots \#_{1} (l_1 \#_0 z \#_0 r_1) \#_{1} \ldots \#_{n-2} r_{n-1}) \#_{n-1} r_n \#_{n} y_1 \to y_0 $$
(resp.
$$ y_0 \#_{n} l_n \#_{n-1} (l_{n-1} \#_{n-2} \ldots \#_{1} (l_1 \#_0 z \#_0 r_1) \#_{1} \ldots \#_{n-2} r_{n-1}) \#_{n-1} r_n \to y_1 $$
which express a compatibility between $z$, $y_0$, and $y_1$.

In particular, this implies that the equation $\Lambda P \to P$ has weakly unique solutions in $C$.
\end{remark}

}
\change{
\begin{example}
The underlying $\infty$-category of $\Uni^{coh}_{\leftdivision{\Lambda\textbf{Eq}}{1,1}}(\leftdivision{\textbf{Eq}}{1,1})$ is 
\[\begin{tikzcd}[ampersand replacement=\&]
	\& \bullet \&\&\&\& \bullet \\
	\bullet \&\& \bullet \&\& \bullet \&\& \bullet
	\arrow[""{name=0, anchor=center, inner sep=0}, "b"', from=2-5, to=2-7]
	\arrow[""{name=1, anchor=center, inner sep=0}, "a", from=2-5, to=1-6]
	\arrow["{x_0}", from=1-6, to=2-7]
	\arrow["a", from=2-1, to=1-2]
	\arrow[""{name=2, anchor=center, inner sep=0}, "{x_1}"{description}, from=1-2, to=2-3]
	\arrow[""{name=3, anchor=center, inner sep=0}, "b"', from=2-1, to=2-3]
	\arrow[""{name=4, anchor=center, inner sep=0}, "{x_0}", curve={height=-24pt}, from=1-2, to=2-3]
	\arrow["{y_0}", shorten <=3pt, shorten >=3pt, Rightarrow, from=1-6, to=0]
	\arrow["z"', shorten <=4pt, shorten >=4pt, Rightarrow, from=4, to=2]
	\arrow["{y_1}", shorten <=3pt, shorten >=3pt, Rightarrow, from=1-2, to=3]
	\arrow["w", shorten <=29pt, shorten >=29pt, Rightarrow, scaling nfold=3, from=2, to=1]
\end{tikzcd}\]
\end{example}}

\subsection{Characterization of Fibrant Objects of The Inductive Left Semi-Model Structure}
\label{sec_fibrantobjects}
In this section, we will give a simple characterization of the fibrant objects of the left semi-model structure introduced in \cref{th:canonical_left_semi_model}. We will temporarily call the objects satisfying this characterization ``prefibrant'' (\cref{def:prefibrant}) and then show in \cref{char_fibrant_obj} that these are exactly the fibrant objects.

\begin{definition}\label{def:inverse}
Let $a$ be an $(n+1)$-arrow in a $m$-marked $\infty$-category $C$. An \emph{inverse} for $a$ is an arrow $a^{-1}$ such that there exist two marked arrows:
\[\epsilon:  a\#_{n} a^{-1}\to \Ib~~~~~~ \nu:a^{-1}\#_n a \to \Ib.\]
An arrow is \emph{invertible} if it has an inverse.
\end{definition}

\begin{definition}\label{def:prefibrant}
An $m$-marked $\infty$-category $C$ is \emph{prefibrant} if
\begin{enumerate}
\item marked arrows of $C$ are invertible and their inverses are marked,
\item whenever $a$ and  $c:a\to b$ are marked in $C$, so is $b$.
\end{enumerate}
This directly implies that if $b$ and $c:a\to b$ are marked, so is $b$.
\end{definition}

This notion is purely temporary: we will show in \cref{char_fibrant_obj} that an object is fibrant for the left semi-model structure of \cref{th:canonical_left_semi_model} if and only if it is prefibrant.

\begin{prop}
\label{prop:C_prefibrant_implies_solution_to_left_division/saturation_and_right_division/saturation}
{Let $0<k\leq n$ be two integers.}
If $C$ is prefibrant, then equations $\rightdivision{\textbf{eq}}{k,n}$ and $\leftdivision{\textbf{eq}}{k,n}$ have weakly unique solutions in $C$.
\end{prop}
\begin{proof}
We show the result by decreasing induction on $k\leq n$. The initialization corresponds to $k=n$. In this case,
the data of a morphism $\Lambda \leftdivision{\textbf{Eq}}{n,n}\to C$ corresponds to two $n$-arrows $a$ and $b$ sharing the same source and such that $a$ is marked.
Let $\nu:a^{-1}\#_n a \to \Ib$.
If we define $x:=a^{-1}\#_n b$ and $y:\psi\#_n b:a\#_n x\to b$, the couple $(x,y)$ is a solution of $\leftdivision{\textbf{eq}}{n,n}$. If $b$ is marked, so is $x$.
We now show the weak uniqueness of the solution. Let $(\bar{x},\bar{y})$ be another solution. We then have a marked arrow:
\[z: \bar{x} \xrightarrow{\nu^{-1}} a^{-1}\#_n a\#_n \bar{x}\xrightarrow{\bar{y}} a^{-1}\#_n b.\]
The assertion for $\rightdivision{\textbf{eq}}{n,n}$ is similar.

Suppose now the result is true for $k+1$. We start by showing that solutions of $\leftdivision{\textbf{eq}}{k,n}$ and $\rightdivision{\textbf{eq}}{k,n}$ are weakly unique in $C$.
The data of a morphism $\Lambda \leftdivision{\textbf{Eq}}{k,n}\to C$ corresponds to an $n$-arrow
$x:s\to t$, a $k$-invertible arrow $a$ such that $\pi_k^+a=\pi_k^-x$, and an arrow $b:a\#_{k-1} s\to a\#_k t$. Let $(x,y:a\#_{k-1}x\to b)$ be a solution of this equation. Let $\nu:a^{-1}\#_{k-1} a \to \Ib_{\pi_k^+a}$ be a marked $(k+1)$-arrow. \change{We recall that the interchange rule implies that 
$$\begin{array}{rcl}
(\nu\#_{k-1}s)\#_{k}x &=& (\nu\#_{k-1} \Ib_{s})\#_{k}(\Ib_{\Ib_{\pi_k^+a}}\#_{k}x) \\
&=& (\nu\#_{k} \Ib_{\Ib_{\pi_k^+a}}) \#_{k-1} (  \Ib_{s}\#_{k} x) \\
&=& \nu \#_{k-1} x \\
&=& ( \Ib_{a^{-1}\#_{k-1} a} \#_{k} \nu  ) \#_{k-1} (x  \#_{k}  \Ib_{t}) \\
&=& (  \Ib_{a^{-1}\#_{k-1} a}  \#_{k-1}x)  \#_{k} (\Ib_{t} \#_{k-1}\nu) \\
&=& (a^{-1}\#_{k-1}a\#_{k-1}x)\#_{k}(\nu\#_{k-1}t)
\end{array}
$$}

The arrow $x$ is then also a solution of  $\leftdivision{\textbf{eq}}{k+1,n}$:
\[\begin{tikzcd}
    {(\nu\#_{k-1}s)\#_{k}x = (a^{-1}\#_{k-1}a\#_{k-1}x)\#_{k}(\nu\#_{k-1}t)} && { (a^{-1}\#_{k-1}b)\#_{k}(\nu\#_{k-1}t)}
    \arrow["{(a^{-1}\#_{k-1}y)\#_{k}(\nu\#_{k-1}t)}", from=1-1, to=1-3]
\end{tikzcd}\]
and so is weakly unique. The uniqueness of solutions of $\rightdivision{\textbf{eq}}{k,n}$ is proved similarly.

We show now that  $\leftdivision{\textbf{eq}}{k,n}$ and $\rightdivision{\textbf{eq}}{k,n}$ have solutions in $C$. Let $(x,y)$ be a solution of the equation
$$y:(\nu\#_0s)\#_{k}x\to (a^{-1}\#_{k-1}b)\#_{k}(\nu\#_{k-1}t)$$
Moreover, we can find such $x$ marked whenever $b$ is.
We then have
\[(\nu\#_0s)\#_{k}x = (a^{-1}\#_{k-1}a\#_{k-1}x)\#_{k}(\nu\#_{k-1}t).\]
By weak uniqueness of solutions of $\rightdivision{\textbf{eq}}{k+1,n}$, we then have a marked arrow
\[z:  a^{-1}\#_{k-1}a\#_{k-1}x \to a^{-1}\#_{k-1}b.\]
But $a\#_{k-1}x$ and $b$ are solutions of an equation $\leftdivision{\textbf{eq}}{k,n}$, and so there exists a marked arrow
\[\tilde{y}:a\#_{k-1}x\to b.\]
If $b$ is marked, the arrow $x$ that we produce is also marked.
The existence of a solution of $\rightdivision{\textbf{eq}}{k,n}$ is proved similarly.
\end{proof}

\begin{lemma}
\label{lemma:left_division_and_right_division_implies_lifting_against_all_equation}
If the equations $\rightdivision{\textbf{eq}}{k,n}$ and $\leftdivision{\textbf{eq}}{k,n}$ have solutions in $C$ for any integers $0<k\leq n$, then all equations have solutions in $C$.
\end{lemma}
\begin{proof}
Let $\Lambda P\to P$ be a left equation. There is a decomposition of the source of $y$ of the shape
\[ \pi_n^- y =  l_n \#_{n-1} (l_{n-1}\#_{n-2}...\#_{1}(l_1\#_0 x\#_0r_1)\#_{1}...\#_{n-2} r_{n-1})\#_{n-1} r_n \]
where for each $i$,  $l_i$ and $r_i$ are marked $i$-arrows in $P$. We can then use the existence of solutions to $\rightdivision{\textbf{eq}}{k,n}$ and $\leftdivision{\textbf{eq}}{k,n}$ to get two sequences of arrows $(x_k)_{0< k < 2n}$ and $(y_k)_{0< k< 2n}$ such that:
\begin{enumerate}
\item $y_{2n-1}: x_{2n-1}\#_{n-1} r_n \to \pi_n^+ y $;
\item $y_{2k-1}: x_{2k-1}\#_{k-1} r_k \to x_{2k}$;
\item $y_{2k-2}: l_k\#_{k-1} x_{2k-2}\to x_{2k-1}$.
\end{enumerate}
Moreover, arrows $x_k$ are marked whenever $\pi_n^+ y$ is. The couple $(x_0,\bar{y})$ is then a solution to $P$ where $\bar{y}$ is the composite:
\[\begin{array}{rcl}
\bar{y}:= &&(l_n \#_{n-1} (l_{n-1}\#_{n-2}...\#_{1}(( y_0\#_0 r_1)\#_{n}y_1)\#_{1}...\#_{n-2} r_{n-1})\#_{n-1} r_n)\\
&\#_{n}&(l_n \#_{n-1} (l_{n-1}\#_{n-2}...\#_{2}(( y_2\#_1 r_2)\#_{n}y_3)\#_{2}...\#_{n-2} r_{n-1})\#_{n-1} r_n)\\
&\#_{n}&...\\
&\#_{n}& (y_{2n-2}\#_{n-1} r_n)\#_{n}y_{2n-1}\\
\end{array}\]
If $\Lambda P\to P$ is a right equation, we define $\Lambda P\to P^{op}$ to be the left equation obtained by inverting the direction of the arrow of maximum dimension. A solution of $\Lambda P\to P$ is given by $(x, y^{-1})$ where $(x, y)$ is a solution of $\Lambda P\to P^{op}$. Moreover, one can find an arrow $x$ marked whenever the source of $y^{-1}$ is.
\end{proof}

\begin{lemma}
\label{lemma:condition_to_have_the_rlp_against_all_equation_and_saturations}
Let $C$ be an $m$-marked $\infty$-category such that all equations have solutions in $C$ and whenever $a$ and $c:a\to b$ are marked, so is $b$. Then $C$ has the right lifting property against all equations and saturations.
\end{lemma}
\change{
\begin{proof}
By definition, $C$ has the right lifting property against all equations.
Let $\Omega Q \to Q$ be a saturation, and let $n$, $x$, and $y$ be the integer and the two generators of \cref{defi:saturation_in_a_marked_category}. We denote by $P$ the $m$-marked polygraph obtained from $Q$ by unmarking $x$ and all the $n$-arrows appearing in the $n$-target of $y$. We also denote by $\Lambda P$ the $m$-marked sub-polygraph of $P$ that contains all generators except $x$ and $y$. The morphism $\Lambda P \to P$ is then an equation.

Suppose now that we have a morphism $f:\Omega Q \to C$. This corresponds to a solution $(x, y)$ of the equation $\Lambda P \to P$. We then know that there exists another solution $(\bar{x}, \bar{y})$ of the equation where $\bar{x}$ is marked. Furthermore, as $P \coprod_{\Lambda P} P \to \text{Uni}_{\Lambda P}^{\text{coh}}(P)$ is an equation, it has solutions in $C$, and there exists a marked arrow $z':\bar{x} \to x$. By assumption, this implies that $x$ is marked. This shows that we can lift the morphism $f$ to $Q$.
\end{proof}
}
\begin{lemma}
\label{lemma:fibration_admits_solution_to_right_division}
Fibrant objects have the right lifting property against the equations $\rightdivision{\textbf{eq}}{n,n}$ and saturations $\rightdivision{\textbf{sat}}{n,n}$.
\end{lemma}
\begin{proof}
Consider a lifting problem of $\rightdivision{\textbf{eq}}{n,n}$ against $C$. This means that we have in $C$ an $n$-arrow $b$ and a marked $n$-arrow $a$ that share the same source.

Since $C$ is fibrant, it has, by definition, the right lifting property against $\cylinder{\textbf{eq}}{n}$ as in \cref{rem:generating_anodyne_cofiration_seen_as_an_equation}. Using the same notations as in \ref{rem:generating_anodyne_cofiration_seen_as_an_equation} for the generators of $\Lambda \cylinder{\textbf{Eq}}{n}$, we choose the image of $a_l$ in $C$ to be an identity for all $l<n$, and $a_n = a$. This gives us a span:
\[
\xymatrix{
\Lambda \cylinder{\textbf{Eq}}{n} \ar[r] \ar[d]_{\cylinder{\textbf{eq}}{n}} & C \\
\cylinder{\textbf{Eq}}{n} &
}
\]
which has a dotted diagonal filling $(x, y)$. But this pair verifies $y: x \#_{k-1} a \to b$, and is thus a solution to the lifting problem above.

The proof for the saturation $\rightdivision{\textbf{sat}}{n,n}$ is similar.
\end{proof}

\begin{lemma}
\label{prop:marked_arrows_admit_inverse_in_fibrant_objects}
In a fibrant $m$-marked $\infty$-category, all marked arrows are invertible. Moreover, their inverses are marked.
\end{lemma}
\begin{proof}
\cref{lemma:fibration_admits_solution_to_right_division} states that $C$ has the right lifting property against $\rightdivision{\textbf{eq}}{n,n}$ and $\rightdivision{\textbf{sat}}{n,n}$.

First, the right lifting property against $\rightdivision{\textbf{eq}}{n,n}$ shows that for any marked arrow $a$, there exists a pair $(a^{-1}, \nu)$ where $\nu$ is marked and
\[\nu: a^{-1} \#_{n} a \to \Ib.\]
The fact that $a^{-1}$ is marked follows from the right lifting property against $\rightdivision{\textbf{sat}}{n,n}$.

Using again the right lifting property against $\rightdivision{\textbf{eq}}{n,n}$, we deduce that there are two marked arrows $(a^{-1})^{-1}$ and $\beta$ such that:
\[\beta: (a^{-1})^{-1} \#_{n} a^{-1} \to \Ib.\]

Finally, in the same way, we obtain a marked arrow:
\[\beta^{-1}: \Ib \to (a^{-1})^{-1} \#_{n} a^{-1}.\]
We then define $\epsilon: a \#_n a^{-1} \to \Ib$ as the composite:
\[
\begin{tikzcd}
    {a \#_n a^{-1}} &&& \Ib \\
    {(a^{-1})^{-1} \#_n a^{-1} \#_n a \#_n a^{-1}} &&& {(a^{-1})^{-1} \#_n a^{-1}}
    \arrow["{\beta^{-1} \#_n a \#_n a^{-1}}"', from=1-1, to=2-1]
    \arrow["{(a^{-1})^{-1} \#_n \nu \#_n a^{-1}}"', from=2-1, to=2-4]
    \arrow["\beta"', from=2-4, to=1-4]
\end{tikzcd}
\]
As it is a composite of marked arrows, $\epsilon$ is also marked.
This then shows that $a^{-1}$ is an inverse of $a$.
\end{proof}

\begin{lemma}
\label{lem:fibrant_implies_prefibrant}
Fibrant objects are prefibrant.
\end{lemma}
\begin{proof}
\cref{prop:marked_arrows_admit_inverse_in_fibrant_objects} implies the first condition. For the second one, let $y:x \to b$ be a marked arrow where $b$ is marked. The right lifting property against $\rightdivision{\textbf{sat}}{n,n}$, choosing $a$ to be an identity, implies that $x$ is marked. Now suppose given a marked arrow $y:b \to x$ where $x$ is marked. We have a marked arrow $y^{-1}:x \to b$, and thus $b$ is also marked.
\end{proof}

\begin{prop}\label{char_fibrant_obj}
For an $m$-marked $\infty$-category $C$, the following assertions are equivalent:
\begin{enumerate}
\item $C$ is prefibrant in the sense of \cref{def:prefibrant}.
\item All equations have solutions in $C$, and whenever $a$ and $c:a\to b$ are marked, so is $b$.
\item $C$ has the right lifting property against all equations and saturations.
\item $C$ is fibrant for the left semi-model structure of \cref{th:canonical_left_semi_model}.
\end{enumerate}
\end{prop}
\begin{proof}
The implication $(1) \Rightarrow (2)$ is a consequence of \cref{prop:C_prefibrant_implies_solution_to_left_division/saturation_and_right_division/saturation} and \cref{lemma:left_division_and_right_division_implies_lifting_against_all_equation}. \cref{lemma:condition_to_have_the_rlp_against_all_equation_and_saturations} states $(2) \Rightarrow (3)$. By \cref{prop:generating_anodyne_cof_are_equations_and_saturations}, generating anodyne cofibrations are either equations or saturations, and thus $(3) \Rightarrow (4)$. Eventually, the implication $(4) \Rightarrow (1)$ is the content of \cref{lem:fibrant_implies_prefibrant}.
\end{proof}

\subsection{Isofibrations}
\label{subsec:isofibration}

In this section, we provide a simpler characterization of fibrations between fibrant objects as the ``isofibrations'' in the following sense:

\begin{definition}
A morphism between $m$-marked $\infty$-categories is said to be an \emph{isofibration} if it has the lifting property against the maps:
\[ i^{+}_n : \Db_n^\flat \to (\Db_{n+1}, \overline{\{e_{n+1}\}}) \]
where $e_{n+1}$ is the unique non-identity arrow of $\Db_{n+1}$.
\end{definition}

\begin{notation}
\label{nota:lifting_properties}
Suppose given an equation $\Lambda P \to P$ and a lifting problem of the form:
\[
\begin{tikzcd}
    \Lambda P \ar[d] \ar[r] & C \ar[d,"p"] \\
    P \ar[r] \ar[ur,dotted] & D
\end{tikzcd}
\]
Given $a$ a generator of $P$, we will denote its image in $D$ also by $a$. If $a \in \Lambda P$, we denote by $\overline{a}$ its image in $C$. So in general $p(\overline{a}) = a$. If the dotted diagonal lift exists, or if we are in the process of constructing such a lift, the image of $x, y \in P$ in $C$ is also denoted $\overline{x}$ and $\overline{y}$, and we hence also have $p(\overline{x}) = x$ and $p(\overline{y}) = y$.
\end{notation}

Explicitly, a morphism $\pi: X \to Y$ between fibrant $m$-marked $\infty$-categories is an isofibration if for every $n$-dimensional arrow $f: a \to b$ in $X$, such that in $Y$ there is a parallel arrow $g: \pi(a) \to \pi(b)$ with a marked arrow $h: g \to \pi(f)$, then $g$ and $h$ can be lifted to arrows $\overline{g}: a \to b$ and $\overline{h}: \overline{g} \to f$ in $X$, with $\overline{h}$ marked, such that $\pi(\overline{g}) = g$ and $\pi(\overline{h}) = h$.

Note that it follows from \cref{lem:hemisphere_are_anodyne} that fibrations are isofibrations. We insist on the fact that we will only consider the notion of isofibration between \emph{fibrant} $m$-marked $\infty$-categories. We do not expect the definition given above to be very interesting outside this context.

\begin{lemma}\label{lem:dual_isofibration_lift}
Any isofibration between fibrant $m$-marked $\infty$-categories also has the lifting property against
\[ i^{-}_n : \Db_n^\flat \to (\Db_{n+1}, \overline{\{e_{n+1}\}}). \]
\end{lemma}

\begin{proof}
Let $\pi: X \to Y$ be an isofibration between fibrant $m$-marked $\infty$-categories, $f: a \to b$ an $n$-arrow in $X$, with $g: \pi(a) \to \pi(b)$ and $h: \pi(f) \to g$ two arrows in $Y$, where $h$ is marked.

As $Y$ is fibrant, according to \cref{prop:marked_arrows_admit_inverse_in_fibrant_objects}, the arrow $h$ admits an inverse, i.e., there is a marked arrow $h^{-1}: g \to \pi(f)$ and another marked arrow $t: h^{-1} \#_n h \to \Ib_g$ witnessing the inverse relation. One can then apply the isofibration property to lift $g$ and $h^{-1}$ to two arrows $\overline{g}: a \to b$ and $\overline{h^{-1}}: \overline{g} \to f$.

As $X$ is also fibrant, one can then consider an inverse $\overline{h'}$ of $\overline{h^{-1}}$ in $X$, whose image by $\pi$ will be a second inverse of $h^{-1}$ in $Y$, and again because $Y$ is fibrant, one can hence construct a marked arrow $h \to \pi(\overline{h'})$. Applying the isofibration property one more time then gives us a lift of $h$ and concludes the proof.
\end{proof}

\begin{lemma}
\label{lem:isofibration_lifts_all_equation}
An isofibration between fibrant $m$-marked $\infty$-categories has the right lifting property against all equations and saturations.
\end{lemma}

\begin{proof}
We will show that such a morphism has the lifting property against all left equations; the exact same argument shows that it also has the lifting property against all right equations.

Consider an isofibration $\pi: X \to Y$ between two fibrant $m$-marked $\infty$-categories and a lifting problem of $\pi$ against $\Lambda P \to P$:
\[\begin{tikzcd}
    {\Lambda P} & X \\
    P & Y
    \arrow["{(x,y)}"', from=2-1, to=2-2]
    \arrow["\pi", from=1-2, to=2-2]
    \arrow[from=1-1, to=1-2]
    \arrow[from=1-1, to=2-1]
\end{tikzcd}\]
We want to show that $x$ and $y$ can be lifted to $X$.

One first remarks that as $X$ is fibrant, the equation $\Lambda P \to P$ has solutions in $X$ according to \cref{char_fibrant_obj}. This implies that one can find a lift $(x', y'): P \to X$ that makes the upper triangle commutative.

\change{
Now, in $Y$, we have two solutions of the equation $\Lambda P \to P$, given by $(\pi(x'), \pi(y'))$ and $(x, y)$. As $Y$ is fibrant, $P \coprod_{\Lambda P} P \to \Uni_{\Lambda P}^{coh}(P)$ has solutions in $Y$, and there exist marked arrows:
\[ z: x \to \pi(x') \]
\[ w: s \#_{n} \pi(y') \to y \]
where $s$ is by construction a composite of $z$ with arrows in the source of $\pi(y')$.

By the isofibration property, there exists an arrow 
$$\overline{z}: \overline{x} \to x'$$
over $z$. This arrow induces an arrow $\overline{s}$ over $s$. 
By the dual isofibration property from \cref{lem:dual_isofibration_lift}, there exists an arrow 
$$\overline{w}: s \#_{n} y' \to \overline{y}$$
over $w$.
The pair $(\overline{x}, \overline{y})$ then induces the desired lift $P \to X$.

Now, to show that isofibrations have the right lifting property against saturations, one simply remarks that lifts against saturations are unique when they exist (saturations are epimorphisms), so as fibrant objects have the right lifting property against these maps, any map between fibrant objects also has the lifting property against all saturations.}
\end{proof}

\begin{prop}
\label{char_fibration}
A morphism between fibrant $m$-marked $\infty$-categories is a fibration if and only if it is an isofibration.
\end{prop}

\begin{proof}
According to \cref{lem:hemisphere_are_anodyne}, the morphism $i^{+}_n$ is an anodyne cofibration, so all fibrations (between fibrant objects) are isofibrations.

For the converse, as a morphism between fibrant objects is a fibration if and only if it has the right lifting property against generating anodyne cofibrations, which are either equations or saturations, \cref{lem:isofibration_lifts_all_equation} implies that isofibrations between fibrant objects are fibrations.
\end{proof}

As a consequence, we have:
\begin{cor}
\label{cor:Equations and saturations are acyclic cofibrations}
Equations and saturations are acyclic cofibrations.
\end{cor}
\change{
\begin{proof}
The \cref{lem:isofibration_lifts_all_equation} and the \cref{lem:isofibration_lifts_all_equation} implies that equations and saturations have the  lifting property against fibration between fibrants. By definition, this implies that these maps are  acyclic cofibrations.

\end{proof}
}
\subsection{Equivalences}
\label{subsec:equivalence}
We now turn to the characterization of weak equivalences between fibrant objects. 
\begin{definition}\label{def:equivalence}
A morphism $p: X \to Y$ between fibrant $m$-marked $\infty$-categories is an \emph{equivalence of $m$-marked $\infty$-categories} if:
\begin{enumerate}
 \item For any arrow $x \in X$, if $p(x)$ is marked in $Y$, then $x$ is marked in $X$.
 \item For any object $c \in Y$, there exists an object $\tilde{c} \in X$ and a marked arrow $e: p(\tilde{c}) \to c$.
 \item For any pair of parallel arrows $(a, b)$ in $X$, and any arrow $c: p(a) \to p(b)$ in $Y$, there exists an arrow $\tilde{c}: a \to b$ in $X$ and a marked arrow $e: p(\tilde{c}) \to c$ in $X$.
\end{enumerate}
\end{definition}

So informally, a functor is an equivalence if it is conservative, essentially surjective, and ``essentially surjective on each Hom $\infty$-category''.

\begin{prop}\label{char_equiv}
A morphism $f: X \to Y$ between fibrant objects in $\micat$ is a weak equivalence in the left semi-model structure of \cref{th:canonical_left_semi_model} if and only if it is an equivalence in the sense of \cref{def:equivalence}.
\end{prop}

\begin{proof}
\change{
We will use the characterization of weak equivalences between fibrant objects given in \cref{prop:weak equivalence has the homotopy right lifting property}. We recall that in our left semi-model structure, the generating cofibrations are given by
\[ I^\partial = \{ i_n : \partial \Db_n \to \Db_n \mid n \geqslant 0 \} \qquad  I^{+m} = \{ \Db_n \to (\Db_n, \overline{\{e_n\}}) \mid n \geqslant 0 \} \]
To express the homotopy right lifting property, we need a relative cylinder object for each of these cofibrations.

For a map of the form $\Db_n \to (\Db_n, \overline{\{e_n\}})$, we have that the canonical map
\[ (\Db_n, \overline{\{e_n\}}) \coprod_{\Db_n} (\Db_n, \overline{\{e_n\}}) \to (\Db_n, \overline{\{e_n\}}) \]
is an isomorphism, so $(\Db_n, \overline{\{e_n\}})$ is already a cylinder object. In particular, the weak left lifting property against these maps is exactly the same as the ordinary left lifting property and it corresponds exactly to the first condition of \cref{def:equivalence}.

For the map $i_n: \partial \Db_n \to \Db_n$, one obtains a relative cylinder object by considering the factorization:
\[ \Db_n \coprod_{\partial \Db_n} \Db_n \cto (\Db_{n+1}, \overline{\{e_{n+1}\}}) \to \Db_n \]
The first map freely adds a (marked) $(n+1)$-arrow between the two non-trivial arrows of the domain, so it is a cofibration. And one of the two maps $\Db_n \to (\Db_{n+1}, \overline{\{e_{n+1}\}})$ was shown to be an anodyne cofibration in \cref{lem:hemisphere_are_anodyne}, hence proving that this is a relative cylinder object for this cofibration. Using this cylinder to express the weak lifting property against $i_n$, one obtains exactly the second condition (for $n = 0$) and the third condition (for $n > 0$) of \cref{def:equivalence}. Indeed, suppose given a weak lifting diagram: }
\[
 \begin{tikzcd}
   \partial \Db_n \ar[ddd, cto] \ar[rr] \ar[dr] & & X \ar[ddd,"p"] \\
   & \Db_n \ar[d] \ar[ur, dotted] & \\
   & (\Db_{n+1}, \overline{\{e_n\}}) \ar[dr, dotted] & \\
  \Db_n \ar[rr] \ar[ur] & & Y \\  
 \end{tikzcd}
\]
\change{
The solid part of the diagram corresponds to a pair of parallel $(n-1)$-arrows $(a, b)$ in $X$, together with an $n$-arrow $c: p(a) \to p(b)$ in $Y$. The top dotted morphism gives us an arrow $\tilde{c}: a \to b$, while the bottom dotted morphism corresponds to a marked $(n+1)$-arrow $e: p(\tilde{c}) \to c$. So this lifting condition corresponds exactly to the third point of \cref{def:equivalence} (with the second point corresponding to the case $n = 0$).
}
\end{proof}

\subsection{The Saturated Inductive Localization.}
\label{subsec:saturation}\cref{char_fibrant_obj} produces a characterization of fibrant objects of the left semi-model structure of \cref{th:canonical_left_semi_model}: a marked $\infty$-category is fibrant if the marked arrows have inverses and if an arrow isomorphic to a marked arrow is marked.

A careful reader might have noticed, however, that this is not sufficient to show that the marked arrows are exactly the arrows that have inverses in the sense of \cref{def:inverse}.

\begin{example} Let $C$ be a category, seen as an $\infty$-category with no non-identity arrows of dimension strictly superior to $1$. We endow $C$ with the marking $C^\flat$, where only the identity arrows are marked.

With this marking, $C$ is fibrant; indeed, it satisfies all the conditions of \cref{char_fibrant_obj}. But if the category $C$ has non-identity invertible arrows, these would be arrows that have inverses in the sense of \cref{def:inverse} without being marked.
\end{example}

In this section, we ``fix'' this problem by introducing a Bousfield localization in which the fibrant objects have these properties.

\begin{definition} A marked $\infty$-category $C$ is said to satisfy the \emph{$2$-out-of-$6$ property} if given three composable $n$-arrows $f$, $g$, and $h$ such that $f \#_{n-1} g $ and $ g \#_{n-1} h $ are marked, then $f$, $g$, and $h$ are marked. \end{definition}

\begin{remark}\label{rk:homotopy_category} If $C$ is a fibrant $m$-marked $\infty$-category, then the relation $f \sim g$ defined by $\exists c : f \to g$ a marked $(n+1)$-arrow, is an equivalence relation on $n$-arrows. Indeed, it is reflexive and transitive as identities are marked and composites of marked arrows are marked, and it is symmetric as marked arrows have inverses.

This equivalence relation is moreover compatible with all composition operations, so that one can define a ``homotopy $n$-category'' $h_n C$, which is an $n$-category whose $k$-(arrows for $k<n$ are those of $C$ and its $n$-arrows are equivalence classes for this relation. We will use in particular that given two parallel $(n-2)$-arrows $u, v$ in $C$ we have a category $h_nC(u,v)$ whose objects are $(n-1)$-arrows $u \to v$ and whose morphisms are equivalence classes of $n$-arrows between them. \end{remark}

\begin{lemma}\label{lem:2_out_of_6} For an $m$-marked $\infty$-category $C$, the following conditions are equivalent:

 \begin{enumerate}
 \item\label{lem:2_out_of_6:1} An arrow in $C$ is marked if and only if it has an inverse in the sense of \cref{def:inverse}.
 \item\label{lem:2_out_of_6:2} $C$ is fibrant in the \change{inductive  left semi-model structure $\micat_\ind$ of \cref{th:canonical_left_semi_model}} and satisfies the $2$-out-of-$6$ property.
 \end{enumerate}
 \end{lemma}

\begin{proof}
\change{
We first consider $C$  an $m$-marked $\infty$-category which satisfies \ref{lem:2_out_of_6:1}, and we check it fulfills the conditions of \cref{def:prefibrant}. By \cref{char_fibrant_obj}, this will imply that $C$ is fibrant. The first condition of \cref{def:prefibrant} is immediate; we check the second condition. Let $b$ and $c : a \to b$ be two marked arrows. By assumption, $b$ is invertible, and there exists then an arrow $b^{-1}$ and two marked arrows $\epsilon : b^{-1} \#_n b \to \Ib$ and $\nu : b \#_n b^{-1} \to \Ib$. We then have marked arrows:
 \[ b^{-1} \#_n a \overset{b^{-1} \#_n c}{\to} b^{-1} \#_n b \overset{\epsilon}{\to} \Ib  \]
 \[  a \#_n b^{-1} \overset{ c \#_n b^{-1}}{\to}  b \#_n b^{-1} \overset{\epsilon}{\to} \Ib  \]}
This shows that $b^{-1}$ is also an inverse for $a$, and hence if all arrows with an inverse are marked, $a$ is marked as well. Note that if it is $a$ which is marked in the first place, then one can consider an inverse $c^{-1} : b \to a $ and apply the same argument.

Next, we show that $C$ satisfies $2$-out-of-$6$. For this, we can rely on \cref{rk:homotopy_category}. An $n$-arrow has an inverse in the sense of \cref{def:inverse} if and only if it is an isomorphism in the category $h_nC(u,v)$ where $u$ and $v$ are its $(n-2)$-dimensional source and target. Our assumption is then that an $n$-arrow is marked if and only if its equivalence class is invertible in the category $h_nC(u,v)$. The fact that marked arrows satisfy $2$-out-of-$6$ then follows from the fact that isomorphisms in a category satisfy the $2$-out-of-$6$ condition.

Conversely, assuming that $C$ satisfies condition \ref{lem:2_out_of_6:2}, we have that marked arrows have inverses because $C$ is fibrant and \cref{char_fibrant_obj}. If an arrow $a$ has an inverse $a^{-1}$, then both $a \#_{n-1} a^{-1}$ and $a^{-1} \#_{n-1} a$ are marked because they are equivalent to identities, and it follows from the $2$-out-of-$6$ condition that $a$ (and $a^{-1}$) is marked.
\end{proof}

\begin{theorem}\label{theo:saturated_inductive_model_structure} 
The inductive semi-model structure $\micat_\ind$ of \cref{th:canonical_left_semi_model} admits a Bousfield localization (as a left semi-model structure) in which the fibrant objects are the marked $\infty$-categories that satisfy the equivalent conditions of \cref{lem:2_out_of_6}.

We call this left semi-model structure the \emph{saturated inductive left semi-model structure} and denote it by $\micat_\satind$.
\end{theorem}

As a Bousfield localization, this left semi-model structure has the same cofibrations and the same fibrations between fibrant objects as the left semi-model structure from \cref{th:canonical_left_semi_model}.

\begin{proof}
The key point here is that the $2$-out-of-$6$ condition for a marked $\infty$-category corresponds to the lifting property against certain cofibrations.

For each $n$, we consider the polygraphs $X_n$ generated by three composable $n$-arrows
 \[ \Db_n \coprod_{\Db_{n-1}} \Db_n \coprod_{\Db_{n-1}} \Db_n \]
where each pushout uses the target maps on the left and the source map on the right. We call $f$, $g$, and $h$ the three $n$-dimensional generators of $X_n$. We consider the map $s_n$:
 \[s_n : \left(X_n, \overline{\{f \#_{n-1} g, g \#_{n-1} h  \}} \right) \to \left(X_n, \overline{\{ f , g , h \} } \right)  \]
which is the identity of $X_n$ (with two different markings). $s_n$ is a cofibration, and a marked $m$-category has the right lifting property against all the $s_n$ if and only if it satisfies the $2$-out-of-$6$ property.

\change{
Using \cref{theo:left Bousfield localization}, we define $\micat_\satind$ as the left Bousfield localization of $\micat_\ind$ at the set $\{s_n\}_{n\in \Nb}$.}

\change{
\Cref{lemma:technical result on localizations} characterizes the fibrant objects of $\micat_\satind$ as the fibrant objects of $\micat_\ind$ that have the right lifting property against the $s_n$ and their higher homotopy codiagonal maps $\nabla^k(s_n)$. However, as $s_n$ is a cofibration $A \to B$ that only adds some marking, its codiagonal $B \coprod_A B \to B$ is an isomorphism, and so the $\nabla^k s_n$ can all be taken to be isomorphisms; we only need to check the right lifting property with respect to the maps $s_n$ themselves. Thus, fibrant objects of $\micat_\satind$ correspond to the marked $\infty$-categories that satisfy the equivalent conditions of \cref{lem:2_out_of_6}. This concludes the proof.} \end{proof}

\section{Comparison with Other Model Structures}
\subsection{Truncation Functors}
\label{subsec:truncation}
\begin{definition}
Let $m<p\leq \infty$.
There is a functor:
\[
\begin{array}{rccl}
\pi_{m}:&\micat[p] &\to& \micat[m] \\
&(X,M)&\mapsto & (X,\overline{M}).
\end{array}
\]
that marks every arrow of dimension $m+1$,
an obvious inclusion functor:
\[
\begin{array}{rccl}
\iota_{p}:&\micat[m]&\to& \micat[p]\\
&(X,M)&\mapsto & (X,M)
\end{array}
\]
and eventually, a functor:
\[
\begin{array}{rccl}
\tau_{m}:&\micat[p]&\to& \micat[m]\\
&(X,M)&\mapsto & (X\cap M_{>m},M)
\end{array}
\]
where $X\cap M_{>m}$ is the sub $\infty$-category of $X$ whose arrows of dimension strictly superior to $m$ are the ones in $M$. As $M$ is assumed to be closed under composition and contains the identities, $X\cap M_{>m}$ is indeed an $\infty$-category. 

These functors fit into the following adjunctions:
\[\pi_{m}\dashv\iota_{p}\dashv\tau_{m}.\]

\end{definition}

\begin{notation}
Because $\iota_p$ is the inclusion of a full subcategory, we will often identify $X$ and $\iota_p X$ in our notation. In the same way, for a morphism $f\in \Hom(X,\tau_m(Y))$, the corresponding morphism in $\Hom(\iota_p X,Y)$ will also be denoted $f$.
\end{notation}

\begin{prop}
\label{prop:pi iota and tau are quillen}
For $m<p$, the adjoint pairs $(\pi_{m}\dashv\iota_{p})$ and $(\iota_{p}\dashv\tau_{m})$ are Quillen pairs (definition \cref{defi:basic def on left semi-model structure}) \change{ both between $\icat^{+m}_\satind$ and   $\icat^{+p}_\satind$ and between $\icat^{+m}_\ind$ and   $\icat^{+p}_\ind$}. \end{prop}

\change{
\begin{proof}
The left adjoint functors $\pi_{m}$ and $\iota_{p}$ obviously preserve cofibrations. Their respective right adjoint functors $\iota_{p}$ and $\tau_{m}$ obviously preserve the isofibrations of \cref{subsec:isofibration}, and fibrant objects for either the inductive (characterized by \cref{def:prefibrant} and \cref{char_fibrant_obj}) or saturated inductive model structures (whose characterization is given in \cref{lem:2_out_of_6}). This implies that the right adjoint functors preserve fibrations between fibrant objects.  The left adjoint then also preserves
acyclic cofibrations as well, and this concludes the proof.
\end{proof}

\begin{prop}
\label{prop:iota is homotopically ff}
For any $m<p\leq \infty$, a morphism $f$ in $\micat[m]_{\satind}$ is a cofibration (resp. acyclic cofibration, resp. fibration, resp. acyclic fibration, resp. weak equivalence) if and only if $\iota_p(f)$ is in $\micat[p]_{\satind}$.
\end{prop}

\begin{proof}
This directly follows from \cref{prop:pi iota and tau are quillen} and from the fact that $\iota_p$ is the inclusion of a full subcategory.
\end{proof}

}

As mentioned in the introduction, we can consider the two towers of left semi-model structures:
\[
\icat^{+0}_{\satind} \overset{\pi_0}{\leftarrow} \icat^{+1}_{\satind} \overset{\pi_1}{\leftarrow} \icat^{+2}_{\satind} \overset{\pi_2}{\leftarrow} \dots \overset{\pi_{n-1}}{\leftarrow} \icat^{+n}_{\satind} \overset{\pi_n}{\leftarrow} \dots  
\]
\[
\icat^{+0}_{\satind} \overset{\tau_0}{\leftarrow} \icat^{+1}_{\satind} \overset{\tau_1}{\leftarrow} \icat^{+2}_{\satind} \overset{\tau_2}{\leftarrow} \dots \overset{\tau_{n-1}}{\leftarrow} \icat^{+n}_{\satind} \overset{\tau_n}{\leftarrow} \dots  
\]
and take the projective limit of either tower to get a definition of `strict $(\infty,\infty)$-categories''.

Our goal in this section is to show that the left semi-model structure $\icat^{+\infty}_{\satind}$ is equivalent to the limit of the second tower (with $\tau$ functors). Here, by projective limit, we mean a homotopy theoretic limit of these towers, that is, a homotopy limit of the corresponding tower of $(\infty,1)$-categories. Such projective limits of model structures have been studied in \cite{bergner2012homotopy} and \cite{harpaz2020lax}, and we will use the construction from these papers.

\begin{remark} \label{rk:limit_structure}
It should be noted that the results from \cite{bergner2012homotopy} and \cite{harpaz2020lax} are only proved for Quillen model structures, so they do not immediately apply to the left semi-model structures that we are using here. The proof from these two papers easily adapts to the setting of left semi-model structures with very few modifications, so it should be safe to assume these results can be applied here as well. Though to avoid relying on this, we will give an independent proof that the left semi-model structure we use as a model of these projective limits exists and state our main theorem as an equivalence with this left semi-model structure. The only aspect that still relies on applying the results of \cite{bergner2012homotopy} or \cite{harpaz2020lax} to left semi-model structures is in order to interpret our results as saying something about homotopy limits of towers. 
\end{remark}

\begin{definition}
We define the category $\lax_{n\in\mathbb{N}} \micat[n]$, the \change{\emph{putative lax limit of $\micat[m]$}}, whose objects are sequences
$X_\bullet=\{(X_n,f_n)\}_{n\in \mathbb{N}}$ where $X_n\in  \micat[n]$ and $f_n: X_n\to \tau_{n}X_{n+1}$. By adjunction, objects are in bijection with sequences
\[X_0\xrightarrow{{f_{0}}} X_1 \xrightarrow{{f_{1}}} \dots \xrightarrow{{f_{n-1}}} X_n\xrightarrow{{f_{n}}} \dots \]
where each $X_n\in  \micat[n]$.
\end{definition}

\begin{prop}
There exists a left semi-model structure on $\lax_{n\in\mathbb{N}} \micat[n]$, called \emph{the putative lax-limit left semi-model  structure} and denoted by $\lax_{n\in\mathbb{N}} \micat[n]_\satind$, where fibrations and weak equivalences are pointwise fibrations and weak equivalences of the saturated inductive left semi-model structure, and cofibrations are morphisms $h:X_\bullet\to Y_\bullet$ such that $h_0:X_0\to Y_0$ is a cofibration in $\micat[0]$, and for all $n$, the dotted morphism in the following diagrams is a cofibration in $\micat[i+1]$:
\[\begin{tikzcd}
    {X_n} & {X_{n+1}} \\
    {Y_n} & {Y_n\coprod_{X_n} X_{n+1}} \\
    && {Y_{n+1}}
    \arrow[from=1-1, to=2-1]
    \arrow[from=1-1, to=1-2]
    \arrow[from=1-2, to=2-2]
    \arrow[from=2-1, to=2-2]
    \arrow[curve={height=12pt}, from=2-1, to=3-3]
    \arrow[curve={height=-12pt}, from=1-2, to=3-3]
    \arrow[dotted, from=2-2, to=3-3]
\end{tikzcd}\]
\end{prop}
\change{
\begin{proof}
First, let us notice that $\lax_{n\in\mathbb{N}} \micat[n]$ can be identified with the full subcategory of functors $X:\mathbb{N} \to \micat[\infty]$ such that $X_n \in \micat[n]$.

There is a left semi-model structure on such functors, where fibrations and weak equivalences are pointwise: the Reedy (or projective) model structure as presented at the end of \cref{sec:semi-model-categories}. 
The cofibrations of this model structure are as described in the proposition, and we claim that this model structure ``restricts'' to $\lax_{n\in\mathbb{N}} \micat[n]$.

By this last assertion, we mean that given two sequences $X_\bullet, Y_\bullet \in \lax_{n\in\mathbb{N}} \micat[n]$ and a map $X \to Y$, the factorizations of $f$ as (cofibration, acyclic fibration) or (acyclic cofibration, fibration) in the Reedy left semi-model structure can be done within $\lax_{n\in\mathbb{N}} \micat[n]$, which shows that one can deduce all the properties in the definition of semi-model structures from the fact that they are satisfied by the Reedy model structure.

We will prove the claim for the (acyclic cofibration, fibration) factorization system, the proof for the other one being identical.
We can construct by induction on $n$ a functorial factorization $X_n\to E_n\to Y_n$ of $p_n$ such that $X_0\to E_0\to Y_0$, and $X_n \coprod_{X_{n-1}} E_{n-1} \to E_n \to Y_n$ for $n>0$, is an acyclic cofibration followed by a fibration of $\micat[n]_{\satind}$. As the functor $\iota_\infty:\micat{n}_\satind\to \micat_\satind$ is both a left and right Quillen functor, it preserves acyclic cofibrations and fibrations, and the resulting factorization $X\to E\to Y$ is an acyclic cofibration followed by a fibration of the Reedy left semi-model structure.

We can then deduce that the Reedy left semi-model structure ``restricts'' to $\lax_{n\in\mathbb{N}} \micat[n]$, which concludes the proof.
\end{proof}}
\begin{definition}
\label{defi:definition_of_the_adjunction_between_limits}
We have an adjunction
\[\begin{tikzcd}
    {\lax_{n\in\mathbb{N}} \micat[n]} && {\micat[\infty]}
    \arrow[""{name=0, anchor=center, inner sep=0}, "c", curve={height=-18pt}, from=1-1, to=1-3]
    \arrow[""{name=1, anchor=center, inner sep=0}, "\tau", curve={height=-18pt}, from=1-3, to=1-1]
    \arrow["\dashv"{anchor=center, rotate=-90}, draw=none, from=0, to=1]
\end{tikzcd}\]
where the left adjoint sends a sequence  $X_\bullet$ to its colimit:
\[c(X_\bullet) := \colim_{n\in \mathbb{N}} X_n,\]
and the right adjoint sends an $\infty$-marked $\infty$-category $X$ on the sequence
\[\tau_0(X)\to \dots \to \tau_n(X)\to \dots \]
\end{definition}

\begin{prop}
\label{prop:adjunction_between_the_lax_limit_model_structure_and_the_inductive_one}
This adjunction induces a Quillen adjunction between $\lax_{n\in\mathbb{N}} \micat[n]_\satind$ and $\micat[\infty]_{\satind}$ where the left adjoint preserves weak equivalences and fibrant objects.
\end{prop}

\begin{proof}
\change{
The functor $ c $ preserves cofibrations and acyclic cofibrations because of \cref{lem:Reedy_colim}, and hence is a left Quillen functor.

Secondly, because the left semi-model structure on $ \micat[\infty] $ is $ \omega $-combinatorial, its weak equivalences are closed under $ \omega $-filtered colimits (this is shown for Quillen model structures as Proposition 7.3 of \cite{dugger2001combinatorial}, and for left semi-model structures as Proposition 7.7 of \cite{henry2020minimal}). This implies that the functor $ c $ also preserves weak equivalences: if $ f:X_\bullet \to Y_\bullet $ is an equivalence in $ \lax_{n \in \mathbb{N}} \micat[n] $, then the map
\[
c(f): \colim_{n \in \mathbb{N}} X_n \to \colim_{n \in \mathbb{N}} Y_n
\]
is a filtered colimit of weak equivalences, and so is a weak equivalence. This implies that  $c $ also preserves acyclic cofibrations, which concludes the proof.
}
\end{proof}

\begin{prop}
\label{prop:the_limit_model_structure}
There is a left Bousfield localization of $\lax_{n\in\mathbb{N}} \micat[n]_\satind$, called the \emph{putative limit structure} and denoted by $\plim_{n\in\mathbb{N}} \micat[n]_\satind$, where $X_\bullet$ is fibrant if and only if it is fibrant in the putative lax-limit left semi-model  model structure and if for all integers $n$, $f_n: X_n\to \tau_{n}X_{n+1}$ is a weak equivalence. Moreover, weak equivalences between fibrant objects are pointwise equivalences.
\end{prop}

\begin{remark}
  According to our (unproven) claim (see \cref{rk:limit_structure}) that the results of \cite{bergner2012homotopy} or \cite{harpaz2020lax} can be applied to left semi-model structures, the $\infty$-category obtained as the localization of this Bousfield localization would be equivalent to the limit of the $\infty$-categories obtained as the localization of the $\micat[n]$ (with the $\tau_n$ functors as transitions). 
\end{remark}

We need to introduce certain constructions before proving the proposition:

\begin{construction}
  Let $k$ be any positive integer. We define
  \[ \lax_{i\in\{k,k+1\}} (\micat[i],\tau_i) \]
  to be the category whose objects are triples $(X,X',f:X\to \tau_k(X'))$ where $X$ and $X'$ are respectively $k$-marked and $(k+1)$-marked $\infty$-categories. By adjunction, these objects are in bijection with sequences:
\[X\xrightarrow{{f}} X'\]
where $X$ and $X'$ are respectively $k$-marked and $(k+1)$-marked $\infty$-categories.

There is an adjunction
\[\begin{tikzcd}
    {\lax_{i\in\{k,k+1\}} (\micat[i],\tau_i)} && {\lax_{i\in\mathbb{N}} (\micat[i],\tau_i)}
    \arrow[""{name=0, anchor=center, inner sep=0}, "\alpha_k", curve={height=-12pt}, shorten <=9pt, shorten >=9pt, from=1-1, to=1-3]
    \arrow[""{name=1, anchor=center, inner sep=0}, "\beta_k", curve={height=-12pt}, shorten <=17pt, shorten >=17pt, from=1-3, to=1-1]
    \arrow["\dashv"{anchor=center, rotate=-90}, draw=none, from=0, to=1]
  \end{tikzcd}\]

where the left adjoint $\alpha_k$ sends $X \to Y$ to the sequence
\[\emptyset \to \dots \to \emptyset \to X \xrightarrow{f} Y \to Y \to \dots \to Y \to \dots \]
while the right adjoint $\beta_k$ sends $X_\bullet$ to 
\[X_k \xrightarrow{{f}} X_{k+1}.\]
\end{construction}

\begin{lemma}
\label{label:Cofibration_in_Ik}
Let $i:A \cto B$ be a cofibration between cofibrant objects in $\micat[k]$ and $ I_A B$ a relative cylinder object for $i$ (as in \cref{prop:weak equivalence has the homotopy right lifting property}). Let $\phi$ be the morphism in  $\lax_{i\in \{k,k+1\}} (\micat[i],\tau_i)$ given by the square:
  \[\begin{tikzcd}
      A \ar[r] \ar[d] & B \ar[d] \\
      B \ar[r] & I_A B
\end{tikzcd}\]
There exists a morphism $\psi$ in  $\lax_{i\in \{k,k+1\}} (\micat[i],\tau_i)$ corresponding to a square
\begin{equation}
\label{eq:label}
\begin{tikzcd}
      B \coprod_A B \ar[r] \ar[d] & I_A B \coprod_B I_A B \ar[d] \\
      I_A B \coprod_B I_A B \ar[r] & W
\end{tikzcd}
\end{equation}
where $W$ is a relative cylinder object for $B \coprod_A B \to I_A B$, and such that $\alpha_k(\psi)$ is a relative cylinder  for $\alpha_k(\phi)$.
\end{lemma}

\begin{proof}
\change{
One can first observe that the horizontal map $B \coprod_A B \cto I_A B \coprod_B I_A B$ is already a relative cylinder object for $A \cto B$. 
By definition of the putative lax-limit left semi-model  structure, we then have to construct a square of shape \eqref{eq:label}, with $W$ a relative cylinder object for $B \coprod_A B \to I_A B$, and such that the canonical map

$$ \big( I_A B \coprod_B I_A B\big)\coprod_{(B\coprod_AB)}\big(I_A B \coprod_B I_A B\big)\to W$$ is a weak equivalence.}

We will proceed in three steps. We first factorize the leftmost map:
    \[\begin{tikzcd}
        B \coprod_A B \ar[r] \ar[d,cto] & I_A B \coprod_B I_A B \ar[dd] \\
       I_A B \coprod_B I_A B \ar[d,"\sim"]  & \\
      B  \ar[r] & I_A B
    \end{tikzcd}\]
We then forms the pushout $P$:
      \[\begin{tikzcd}
          B \coprod_A B \ar[r] \ar[d,cto] \ar[dr,phantom,"\ulcorner"very near end] & I_A B \coprod_B I_A B \ar[d,cto] \\
       I_A B \coprod_B I_A B \ar[d]  \ar[r] & P \ar[d]  \\
      B  \ar[r] & I_A B
    \end{tikzcd}\]
 Eventually, we factor the map $P \to I_A B$ can into a cofibration followed by a weak equivalence.
        \[\begin{tikzcd}
          B \coprod_A B \ar[r] \ar[d,cto] \ar[dr,phantom,"\ulcorner"very near end] & I_A B \coprod_B I_A B \ar[d,cto] \\
          I_A B \coprod_B I_A B \ar[dd]  \ar[r] & P \ar[d,cto]  \\
           & W \ar[d,"\sim"] \\
      B  \ar[r] & I_A B
    \end{tikzcd}\]
Which gives a relative cylinder object, and hence a homotopy codiagonal for $\phi$ of the form:
        \[\begin{tikzcd}
          B \coprod_A B \ar[r] \ar[d,cto]  & I_A B \coprod_B I_A B \ar[d,cto] \\
          I_A B \coprod_B I_A B  \ar[r] & W
    \end{tikzcd}\]
But one can see that the object $W$ we constructed above is itself a relative cylinder object for the map $B \coprod_A B \to I_A B \coprod_B I_A B$, which concludes the proof.
\end{proof}

\change{
\begin{proof}[Proof of \cref{prop:the_limit_model_structure}] 
Let $I_k$ be the set of cofibrations of ${\lax_{i\in\mathbb{N}} (\micat[i],\tau_i)}$ of the form
$$\{\alpha_k(A\to B) \to \alpha_k(B \to I_A B)\}$$
where $i: A \cto B$ is a generating cofibration of $\micat[k]$ and $I_A B$ is a relative cylinder object for $i$.

We then define the putative limit left semi-model structure as the left Bousfield localization of the lax-putative limit left semi-model structure by all sets $I_k$ (for all values of $k$). The existence of this localization is asserted by Theorem 7.3 of \cite{henry2020combinatorial}. By \cref{lemma:technical result on localizations}, fibrant objects correspond to morphisms having the right lifting property against iterated homotopy codiagonals of maps in $I_k$. Since weak equivalences between fibrant objects of the localized left semi-model structure correspond to weak equivalences in the unlocalized left semi-model structure, they also correspond to pointwise weak equivalences.

To show that the adjunction given in \cref{defi:definition_of_the_adjunction_between_limits} induces an adjunction between the putative limit left semi-model structure and the inductive left semi-model structure, one has to demonstrate that for any integer $k$, and $\phi \in I_k$, $c(\phi)$ is a weak equivalence of the inductive left semi-model structure. Let $i: A \cto B$ be the generating cofibration of $\micat[k]$ such that $\phi$ is 
$$\alpha_k(A \to B) \to \alpha_k(B \to I_A B).$$
The morphism $c(\phi)$ then corresponds to $B \to I_A B$, which is a weak equivalence by the definition of a relative cylinder object.

To conclude the characterization of fibrant objects of this left semi-model structure, we will show that for any fibrant object $(X_i,f_i)$ of the unlocalized left semi-model structure, the following conditions are equivalent:
\begin{enumerate}
\item $(X_i,f_i)$ has the right lifting property against all maps in $I_k$.
\item For any $k$, $f_k: X_k \to \tau_k X_{k+1}$ is a weak equivalence.
\item $(X_i,f_i)$ has the right lifting property against all maps of the form $\{\alpha_k(A \to B)\to \alpha_k(B\to I_AB)\}$ where $A \to B$ is an arbitrary cofibration in $\micat[k]$.
\item $(X_i,f_i)$ has the right lifting property against iterated homotopy codiagonals of maps in $I_k$.
\end{enumerate}
The implications $(1) \Rightarrow (2)$ and $(2) \Rightarrow (3)$ are a reformulation of \cref{prop:weak equivalence has the homotopy right lifting property}. The implication $(3) \Rightarrow (4)$ is \cref{label:Cofibration_in_Ik}. Finally, the implication $(4) \Rightarrow (1)$ is straightforward.
\end{proof}
}

\begin{theorem}
  The Quillen adjunction between the putative limit left semi-model structure of \cref{prop:the_limit_model_structure} and the inductive left semi-model structure is a Quillen equivalence.

\[ \plim_{n\in\mathbb{N}} \micat[n]_\satind \simeq \micat[\infty]_\satind \]
\end{theorem}

\begin{proof}
As the left adjoint preserves weak equivalence and fibrant objects of the unsaturated left semi-model structure by \cref{prop:adjunction_between_the_lax_limit_model_structure_and_the_inductive_one}, one has to show that for every fibrant $\infty$-marked $\infty$-category $X$, and for every cofibrant and fibrant sequence $X_\bullet$ of the putative limit left semi-model structure, we have two weak equivalences:
\[c \tau X \to X~~~~\text{ and }~~~~X_\bullet \to \tau c X_\bullet.\]
The first one is immediate because
\[X \cong \colim_{n\in\mathbb{N}} \tau_n X.\]

Let $X_\bullet$ be a cofibrant and fibrant object of the putative limit left semi-model structure. Because $X_\bullet$ and $\tau c X_\bullet$ are fibrant, the second comparison morphism is a weak equivalence if and only if for all $k$, $X_k \to \colim_{n\in\mathbb{N}} \tau_k(X_n)$ is a weak equivalence. In order to show this, consider the diagram:
\[\begin{tikzcd}
    {X_0} & {...} & {X_k} & {X_k} & {...} & {X_k} & {...} \\
    {X_0} & {...} & {X_k} & {\tau_k(X_{k+1})} & {...} & {\tau_k(X_{n})} & {...}
    \arrow["\sim", from=1-1, to=2-1]
    \arrow["\sim", from=1-3, to=2-3]
    \arrow["\sim", from=1-4, to=2-4]
    \arrow["\sim", from=1-6, to=2-6]
    \arrow["\sim"', from=2-3, to=2-4]
    \arrow["\sim"', from=2-4, to=2-5]
    \arrow[from=2-1, to=2-2]
    \arrow[from=2-2, to=2-3]
    \arrow["\sim"', from=2-5, to=2-6]
    \arrow["\sim"', from=2-6, to=2-7]
    \arrow["\sim", from=1-3, to=1-4]
    \arrow["\sim", from=1-4, to=1-5]
    \arrow["\sim", from=1-5, to=1-6]
    \arrow["\sim", from=1-6, to=1-7]
    \arrow[from=1-1, to=1-2]
    \arrow[from=1-2, to=1-3]
\end{tikzcd}\]
where, by two out of three, all the vertical morphisms are weak equivalences. The previous diagram corresponds to a weak equivalence in the unlocalized left semi-model structure on $\lax_{n\in\mathbb{N}} \micat[n]$ between $X_{\min(\bullet,k)}$ and $(\tau_k(X_{\bullet})$.
Because the left adjoint $c$ preserves weak equivalences of the unlocalized left semi-model structure by proposition \cref{prop:adjunction_between_the_lax_limit_model_structure_and_the_inductive_one}, this induces a weak equivalence:
$$X_k \cong \colim_{n\in \mathbb{N}} X_{\min(n,k)} \to \colim_{n\in\mathbb{N}} \tau_k(X_n)$$

\end{proof}

\subsection{Coinductive Localization and Comparison with $\icat_\can$}
\label{subsec:Folk}

\change{
  Following \cite[Definition 4.2]{lafont2010folk}, we can also give a coinductive notion of invertible arrows in an $\infty$-category. In short, an $n$-arrow $a \colon \pi_{n-1}^- a \to \pi_{n-1}^+ a$ is said to be coinductively invertible if there is an $n$-arrow $\tilde{a}:\pi_{n-1}^+ a \to \pi_{n-1}^- a$ and two coinductively invertible $(n+1)$-arrows
  \[ c \colon \tilde{a}\#_{n-1}a \to \Ib_{\pi^-_{n-1} a} \]
  \[ c'\colon a\#_{n-1}\tilde{a} \to \Ib_{\pi^+_{n-1} a} \]
  The notion is called ``weakly invertible'' in \cite{lafont2010folk}. Note that this is a coinductive definition, that is an arrow is coinductively invertible if there are two such arrows $c$ and $c'$, which themselves have such ``weak inverses'' $\tilde{c}$ and $\tilde{c'}$ with four witness $n+2$ arrows, which are themselves coinductively invertible, i.e., have weak inverses and there are eight $(n+3)$-arrows witnessing this, and so on... We can make this definition more formal as follows:
}

\begin{definition}
\label{InvertibilitySet}
Let $D$ be an $\infty$-category. An \emph{invertibility set}
in $D$ is a set $E = \coprod_{n>0} E_n$ with $E_n \subset D_n$ 
such that, 
for all $n > 0$ and $a \in E_n$, there exists $\tilde{a} \in E_n$
of the form
$
\tilde{a} \colon \pi^+_{n-1}a \to \pi^-_{n-1}a
$
and $c, c' \in E_{n+1}$ of the form
\[
c \colon \tilde{a}\#_{n-1}a \to \Ib_{\pi^-_{n-1} a} \quad\text{ and }\quad
c' \colon a\#_{n-1}\tilde{a} \to \Ib_{\pi^+_{n-1} a}.
\]
\end{definition}

\begin{definition}
  \label{def:coinductively_invertible}
Let $D$ be an $\infty$-category and $n > 0$. Given $a \in D_n$, the $n$-arrow $a$ is \emph{coinductively invertible} 
if there exists an invertibility set $E$ such that $a \in E$.
\end{definition}

\begin{prop}\label{CoinductiveDefinitionEquivalences}
Let $D$ be an $\infty$-category and $n > 0$. An $n$-arrow $a$ is coinductively invertible if and only if there exists an $n$-arrow $\tilde{a}$ of the form
$
\tilde{a} : \pi^+_{n-1}a \to \pi^-_{n-1}a
$
and two coinductively invertible $(n+1)$-arrows $c, c'$ of the form
\[
c \colon \tilde{a}\#_{n-1}a \to \Ib_{\pi^-_{n-1} a} \quad\text{ and }\quad
c' \colon a\#_{n-1}\tilde{a} \to \Ib_{\pi^+_{n-1} a}.
\]
\end{prop}
\begin{proof}
This is \cite[Lemme 1.1.8]{loubaton2021conditions}.
\end{proof}

\begin{lemma}
\label{lemma:somes_property_of_coinductive_arrows}
Let $X$ be an $\infty$-category, and $M$ the set of coinductively invertible arrows. The set $M$ satisfies the two following properties:
\begin{enumerate}
\item $M = \overline{M}$.
\item For all $c : a \to b$ in $M$, $a \in M \Leftrightarrow b \in M$.
\end{enumerate}
\end{lemma}
\begin{proof}
The first point is the third and the fourth point of example $1.1.9$ of \cite{loubaton2021conditions}, and the second one is a consequence of proposition $1.1.10$ of \textit{loc. cit}.
\end{proof}

\begin{prop}
\label{prop:marked arrows are coinductively invertible}
If $X$ is a fibrant $m$-marked $\infty$-category, all marked arrows in $X$ are coinductively invertible in the underlying $\infty$-category.
\end{prop}
\begin{proof}
The \cref{prop:marked_arrows_admit_inverse_in_fibrant_objects} directly implies that the set of all marked arrows is an invertibility set. By definition, all marked arrows are then coinductively invertible.
\end{proof}

\begin{prop}
\label{prop:if_invertible_arrows_are_marked_then_X_fibrant}
Let $X$ be an $\infty$-category and $M$ the set of coinductively invertible arrows. The marked $\infty$-category $(X,M)$ is then fibrant in the saturated inductive semi-model structure.
\end{prop}
\begin{proof}

\change{\Cref{CoinductiveDefinitionEquivalences} shows that $(X,M)$ satisfies point \ref{lem:2_out_of_6:1} of \cref{lem:2_out_of_6}, which is a characterization of the fibrant objects in the saturated inductive semi-model structure (see \cref{theo:saturated_inductive_model_structure}).} 
\end{proof}

Next we remark that coinductively invertible arrows can be characterized using a lifting property:

\begin{definition}
Let $G_1$ be the $\infty$-category obtained from the factorization of $\Db_1 \to \Db_{0}$ as a cofibration $k_1: \Db_1 \to G_1$ followed by an acyclic fibration $t_1: G_1 \to \Db_{1}$. We then define $G_n := \Sigma^{n-1} G_1$ and $k_n := \Sigma^{n-1} k_1: \Db_n \to G_n$, $t_n := \Sigma^{n-1} t_1: G_n \to \Db_{n-1}$. Let us recall that the definition of the functor $\Sigma^{n-1}$ is given in \cref{defi:suspension}. As the suspension preserves acyclic fibrations and cofibrations, the pair $(k_n, t_n)$ is a factorization of $\Db_n \to \Db_{n-1}$ into a cofibration followed by an acyclic fibration. 
\end{definition}

\begin{prop}
\label{prop:Gn_represents_ci}
Let $X$ be an $\infty$-category, and $f$ an $n$-arrow of $X$.
There exists a lifting in the following diagram:
\[\begin{tikzcd}[ampersand replacement=\&]
	{\Db_n} \& X \\
	{G_n}
	\arrow["f", from=1-1, to=1-2]
	\arrow["{k_n}"', from=1-1, to=2-1]
	\arrow[dashed, from=2-1, to=1-2]
\end{tikzcd}\]
if and only if $f$ is coinductively invertible.
\end{prop}
\begin{proof}
This is a reformulation of lemma $4.36$ of \cite{lafont2010folk}.
\end{proof}

We recall now the model structure on $\icat$ constructed in \cite{lafont2010folk}.
\change{
\begin{theorem}
\label{theo:canonical_model_structure}
There exists a model structure on $\icat$, called the \textit{canonical model structure} and denoted by $\icat_\can$ such that
\begin{enumerate}
\item cofibrant $\infty$-categories are polygraphs.
\item Acyclic fibrations are the morphisms having the left lifting property with respect to the set of morphisms $\{\partial \Db_n \to \Db_n , n \in \Nb\}$.
\item Fibrations are the morphisms having the left lifting property with respect to the set of morphisms $\{\Db_n \xrightarrow{i_n^+} \Db_{n+1} \xrightarrow{k_{n+1}} G_{n+1} , n \in \Nb\}$.
\change{\item Cofibrations and acyclic cofibrations are morphisms having the right lifting property against, respectively, acyclic fibrations and fibrations.}.
\end{enumerate}
\end{theorem}
\begin{proof}
This is Theorem 4.39 and 5.3 of \cite{lafont2010folk}. The first point is the main result of \cite{metayer2008cofibrant}.
\end{proof}
}

\begin{definition}
\label{defi:coinductive_model_structure}
The \emph{coinductive left semi-model structure} on $\micat[\infty]$, denoted by $\micat[\infty]_{\coind}$, is the left Bousfield localization of the left semi-model structure on $\micat[\infty]_{\satind}$ by the set of morphisms:
\[\{(G_n,\overline{\emptyset}) \to \Db_{n-1}^\flat, n \in \mathbb{N}^*\}\]
\end{definition}

\begin{remark}
\label{rem:fibrant_in_the_localization}
Remark that if we define $\tilde{G_n} := \pi_{n-1}(G_n, \overline{\emptyset})$, the sequence     
\[(G_n, \overline{\emptyset}) \xrightarrow{p_n} \tilde{G_n} \xrightarrow{\tilde{k_n}} \Db_{n-1}^\flat\] is a factorization as a cofibration followed by an acyclic fibration in the inductive left semi-model structure. Using the terminology of \cite{henry2020combinatorial}, we will say that the cofibration $p_n$ \textit{represents} the morphism $(G_n, \overline{\emptyset}) \to \Db_{n-1}^\flat$. As we can see in the construction of the left Bousfield localization provided in the proof of Theorem 7.3 of $\textit{op cit}$, a marked $\infty$-category $X$ is fibrant in the coinductive left semi-model structure if and only if $X$ is fibrant in the inductive left semi-model structure and has the right lifting property against morphisms $k_n$ and iterated homotopy codiagonals of $k_n$ for all $n > 0$.
\end{remark}

\begin{prop}
\label{prop:if_X_fibrant_marked_arrow_correspond_to_invertible_one}
Let $X$ be a fibrant $\infty$-marked $\infty$-category in the inductive left semi-model structure. Then $X$ is fibrant in the coinductive left semi-model structure if and only if marked arrows are exactly the coinductively invertible arrows of the underlying $\infty$-category.
\end{prop}
\begin{proof}
Suppose first that $X$ is fibrant in the coinductive left semi-model structure and let $f$ be a coinductively invertible arrow of the underlying $\infty$-category. By \cref{prop:Gn_represents_ci}, this corresponds to a morphism $f:(G_n, \overline{\emptyset}) \to X$. As remarked in \cref{rem:fibrant_in_the_localization}, $X$ has the right lifting property against $k_n$, which implies that $f$ can be lifted to $\pi_{n-1}(G_n)$. That shows that $f$ is marked. Moreover, \cref{prop:marked_arrows_admit_inverse_in_fibrant_objects} states that all marked arrows are coinductively invertible. This shows that marked arrows exactly correspond to coinductively invertible ones.

For the other direction, suppose that $X$ is a marked $\infty$-category, fibrant in the inductive left semi-model structure, whose marked arrows are the coinductively invertible ones. We want to show that $X$ is fibrant in the coinductive left semi-model structure. According to \cref{prop:if_invertible_arrows_are_marked_then_X_fibrant}, $X$ is fibrant in the nonlocalized left semi-model structure. We then have to show that for all integers $n > 0$, $X$ has the left lifting property against $k_n$ and iterated homotopy codiagonals of $k_n$. Remark now that, as $\tilde{G_n} \coprod_{(G_n, \overline{\emptyset})} \tilde{G_n} = \tilde{G_n}$, all the iterated homotopy codiagonals are identities. To conclude, it is enough to show that $X$ has the left lifting property against morphisms $k_n$ for $n > 0$, which is obvious by assumption and by the \cref{prop:Gn_represents_ci}.
\end{proof}

\change{
\begin{lemma}
\label{lemma:intermediate}
Let $X$ be an $\infty$-category, and let $M$ be the set of coinductive invertible arrows.
The canonical morphism $X^\flat \to (X, M)$ is an anodyne cofibration of the coinductive left semi-model structure.
\end{lemma}
\begin{proof}
We denote by $(X, M')$ the marked $\infty$-category obtained as the pushout of the following span:
\[\begin{tikzcd}[ampersand replacement=\&]
	{\coprod_{\Hom(G_n, X)} \tilde{G_n}} \& {\coprod_{\Hom(G_n, X)} G_n^{\flat}} \& {X^\flat}
	\arrow["{\coprod p_n}"', from=1-2, to=1-1]
	\arrow[from=1-2, to=1-3]
\end{tikzcd}\]
By stability by coproducts and pushouts, the canonical morphism $X^\flat \to (X, M')$ is an anodyne cofibration of the coinductive left semi-model structure.

Moreover, Lemma 4.9 of \cite{lafont2010folk} applied to the acyclic fibration $G_n\to \Db_{n-1}$ implies that any arrow of $G_n$ of dimension higher or equal to $n$ is coinductively invertible. In particular, every marked arrow of $\tilde{G_n}$ is coinductively invertible. We then have $M' \subset M$, and \ref{prop:Gn_represents_ci} implies that $M \subset M'$. Furthermore, \cref{prop:if_X_fibrant_marked_arrow_correspond_to_invertible_one} implies that $(X, M)$ is a fibrant object of the coinductive left semi-model structure.
\end{proof}
}

\begin{theorem}
\label{theo;equivalenc eentre can et coind}
\change{
The adjunction 
\[\begin{tikzcd}[ampersand replacement=\&]
	{(\uvar)^{\flat}:\icat} \& {\micat[\infty]:U}
	\arrow[shift left=3, from=1-1, to=1-2]
	\arrow[shift left=3, from=1-2, to=1-1]
\end{tikzcd}\]
induces a Quillen equivalence between $\icat_{\can}$ and $\micat[\infty]_{\coind}$.
}
\end{theorem}
\change{
\begin{proof}
We first show that this adjunction is a Quillen adjunction.

Remark that the left adjoint obviously preserves generating cofibrations. Furthermore, for any integer $n$, the morphism 
$$\Db_n^\flat\xrightarrow{i_n^-} \Db_{n+1}^\flat\xrightarrow{k_{n+1}} G_{n+1}^\flat$$ admits a retract given by the weak equivalence $G_{n+1}^\flat\to \Db_n^\flat$, and so it is a acyclic cofibration of $\micat_{\can}$. The left adjoint then preserves cofibration and acyclic cofibration, which implies that the  adjunction is a Quillen adjunction. 

We now  show that this adjunction is a Quillen equivalence. Let $X$ be a cofibrant $\infty$-category and let $M$ be the set of coinductive invertible arrows of $X$. The lemma \cref{prop:if_X_fibrant_marked_arrow_correspond_to_invertible_one} and \cref{lemma:intermediate} imply that $(X,M)$ is the fibrant replacement of $X^{\flat}$. The derived unit then corresponds to the isomorphism $U (X^{\flat})_{fib}\cong U(X,M)\cong X$. 

Remark now that the right adjoint preserves colimits and cofibrations. It is then sufficient to compute the derived counit on cofibrant and fibrant objects of $\micat[\infty]_{\coind}$. Given such an object $(X,M)$, we then have $((U(X,M))_{cof})^{\flat}\cong X^{\flat}$. As \cref{prop:if_X_fibrant_marked_arrow_correspond_to_invertible_one} states that $M$ is the set of coinductive invertible arrows of $X$, \cref{lemma:intermediate} implies that the derived counit $X^\flat\to (X,M)$ is a weak equivalence.
\end{proof}
\begin{theorem}
\label{theo;equivalenc eentre can et coind2}
The full subcategory of fibrant objects of $\micat[\infty]_{\coind}$ is isomorphic  $\icat$.
 Moreover, a morphism between fibrant objects of $\micat[\infty]_{\coind}$ is a weak equivalence (resp. fibration, resp. acyclic fibration) if and only if the underlying morphism in $\icat_{\can}$ is a weak equivalence (resp. fibration, resp. acyclic fibration).
\end{theorem}

\begin{proof}
The first claim directly follows from \cref{prop:if_X_fibrant_marked_arrow_correspond_to_invertible_one} and from the fact that any functor between $\infty$-categories preserves coinductively invertible arrows.

For the second claim, suppose we are given a morphism $p:(X,M)\to (Y,N)$ between fibrant objects of $\micat[\infty]_{\coind}$. If $U(p)$ is a weak equivalence, so is $p$ by \cref{theo;equivalenc eentre can et coind}

Suppose now that $U(p)$ is an acyclic fibration in $\icat_{\can}$. The morphism $p$ then as the right lifting property against the set $I^{\partial}$ (defined in \cref{def:generating_cofibrations}).
To demonstrate that $p$ is an acyclic fibration, it remains to show that an arrow is marked in $X$ if and only if its image in $Y$ is. Since $M$ and $N$ correspond respectively to the set of coinductively invertible arrows of $X$ and $Y$, this follows from Lemma 4.9 of \cite{lafont2010folk}.

Finally, suppose that $U(p)$ is a fibration in $\icat_{\can}$. As $(X,M)$ and $(Y,N)$ are, by definition, fibrant in $\micat[\infty]_{\coind}$, we need to show that $p$ is an isofibration. Applying Lemma 4.9 of \cite{lafont2010folk} to $G_n\to \Db_{n-1}$, we find that the marked arrows of $\tilde{G_n}$ correspond to coinductively invertible arrows of $G_n$. This marked $\infty$-category is, in particular, fibrant in $\micat[\infty]_{\coind}$. Since $\Db_{n-1}^\flat$ is also fibrant, and since $U$ induces an equivalence between the subcategories of fibrant objects, $p$ has the right lifting property against $\Db_{n-1}^\flat\xrightarrow{i_n^+} (\Db_{n-1},\overline{{e_n}})\to \tilde{G_n}$. Finally, since $(Y,N)$ has by definition the right lifting property against $(\Db_{n-1},\overline{{e_n}})\to \tilde{G_n}$, $p$ has the right lifting property against $\Db_{n-1}^\flat\xrightarrow{i_n^+} (\Db_{n-1},\overline{{e_n}})$ and is thus an isofibration.
\end{proof}
}

Note that if $m < \infty$, then every $m$-marked $\infty$-category which is fibrant for the saturated inductive left semi-model structure is also fibrant for the coinductive left semi-model structure. Hence, when restricting the previous theorem to $m$-marked objects for $m < \infty$, we no longer need to move to the coinductive left semi-model structure and we directly obtain the following:

\begin{cor}
\label{cor:link_beetwen_m_saturated_inductive_and_folk}
If $m<\infty$, the full subcategory of fibrant objects of $\micat[m]_{\satind}$ is isomorphic to the subcategory of $\icat$ composed of $\infty$-categories whose arrows of dimension strictly superior to $m$ are coinductively invertible. Moreover, a morphism between fibrant $m$-marked $\infty$-categories is a weak equivalence (resp. fibration, resp. acyclic fibration) in $\micat[m]_{\satind}$ if and only if the underlying morphism in $\icat$ is a weak equivalence (resp. fibration, resp. acyclic fibration) in $\icat_{\can}$.
\end{cor}

\subsection{The Canonical Model Structure vs the Limit of the $\pi$-Tower}
\label{sec:lim_of_pi_tower}

In this section, we will compare the canonical model structure with the limits of the tower of $\pi$ functors as considered in \cref{subsec:truncation}. 

Given a strict $\infty$-category $C$, it is possible to define an $(\infty,m)$ localization $\pi_m X$, and this defines an object of the limit of the tower of $\pi$ functors. But this construction does not produce an equivalence between this limit and the canonical model structure, contrary to what seemed to have been believed previously. \change{Here we are using ``limit'' as ``the (homotopy) limit of the corresponding tower of associated $(\infty,1)$-categories,'' without referring to any specific model.}

We will show this by building a {morphism $C_\infty \to D_\infty$} that is not an equivalence of the {coinductive} model structure, but becomes invertible in the limit of the $\pi$-tower. Though we believe this is not the case, this still leaves open the possibility that the limit of the $\pi$-tower is equivalent to a further localization of the coinductive left semi-model structure,  \change{where this morphism (and probably others) would become invertible}. If this were the case, then the limit of the $\pi$-tower would be equivalent to a localization $\icat_\can$.

More precisely, we will show:
\begin{prop}
\label{prop_conterexemple}
There exists a morphism $f:C_\infty \to D_\infty$ between cofibrant $\infty$-marked $\infty$-categories such that 
\begin{enumerate}
\item  $f$ is not a weak equivalence in the coinductive left semi-model structure on $\infty$-marked $\infty$-categories defined in \cref{def:coinductively_invertible},
\item for all integers $n$, $\pi_nf$ is a weak equivalence in the saturated inductive left semi-model structure on $n$-marked $\infty$-categories defined in \cref{theo:saturated_inductive_model_structure}.
\end{enumerate}
\end{prop}
As an immediate consequence, we get:
\begin{cor}
\change{
The $(\infty,1)$-functor from the $(\infty,1)$-category associated to $\icat_{\can}$ to 
the limit of the diagram of $(\infty,1)$-categories associated to $(\micat[n]_{\satind}, \pi_n)$ induced by the diagram}
\[\begin{tikzcd}[ampersand replacement=\&]
	\& {\micat[\infty]_{\coind}} \\
	{...} \& {{\micat[n]_{\satind}}} \& {...}
	\arrow[from=1-2, to=2-1]
	\arrow["{\pi_n}", from=1-2, to=2-2]
	\arrow[from=1-2, to=2-3]
	\arrow["{\pi_{n-1}}", from=2-2, to=2-1]
	\arrow["{\pi_n}", from=2-3, to=2-2]
\end{tikzcd}\]
is not an equivalence.
\end{cor}

\begin{construction}
Let $E_1$ denote the following $2$-polygraph: 
\[\begin{tikzcd}
	b && b \\
	& a && a
	\arrow[from=1-1, to=2-2]
	\arrow["f"{description}, from=2-2, to=1-3]
	\arrow[from=1-3, to=2-4]
	\arrow[""{name=0, anchor=center, inner sep=0}, "{\Ib_b}", from=1-1, to=1-3]
	\arrow[""{name=1, anchor=center, inner sep=0}, "{\Ib_b}"', from=2-2, to=2-4]
	\arrow[shorten <=5pt, shorten >=3pt, Rightarrow, from=0, to=2-2]
	\arrow[shorten <=5pt, shorten >=3pt, Rightarrow, from=1, to=1-3]
\end{tikzcd}\]
and $E_n:=\Sigma^{n-1}E_1$. Let us recall that the definition of the functor $\Sigma^{n-1}$ is given in \cref{defi:suspension}. When writing $\Db_n\to E_n$, we will always consider the morphism representing the $n$-arrow $\Sigma^{n-1}f$. We define by induction a sequence of polygraphs $(P_n)_{n\in \mathbb{N}}$. We set $P_0:=\Db_1$ and $P_n$ as the pushout: 
\[\begin{tikzcd}
	{\coprod_{(P_n)_{n+1}}\Db_{n+1}} & {P_n} \\
	{\coprod_{(P_n)_{n+1}}E_{n+1}} & {P_{n+1}}
	\arrow[from=2-1, to=2-2]
	\arrow[from=1-1, to=1-2]
	\arrow[from=1-1, to=2-1]
	\arrow[from=1-2, to=2-2]
	\arrow["\lrcorner"{anchor=center, pos=0.125, rotate=180}, draw=none, from=2-2, to=1-1]
      \end{tikzcd}\]
\change{where $(P_n)_{n+1}$ is the set of $(n+1)$-arrows of $P_n$.}

Informally, taking a pushout along $\Db_{n+1} \to E_n$ means freely adding a left and a right inverse to an arrow $f$ (except there is no marking yet) and so $P_{n+1}$ is constructed by freely adding left and right inverses to all $(n+1)$-arrows of $P_n$.

When writing $\Db_1\to P_n$, we will always consider the morphism representing the $1$-arrow $P_0\to P_n$.
Finally, for $n \in \mathbb{N}\cup\{\infty\}$ we define $C_n$ and $D_n$ as the following pushouts:
\[\begin{tikzcd}
	{\coprod_{k< n}\Db_1} & {\Db_1} & {\Db_0} \\
	{\coprod_{k<n}P_k} & {C_n} & {D_n}
	\arrow[from=1-1, to=1-2]
	\arrow[from=2-1, to=2-2]
	\arrow[from=1-2, to=2-2]
	\arrow[from=1-1, to=2-1]
	\arrow[from=1-2, to=1-3]
	\arrow[from=2-2, to=2-3]
	\arrow[from=1-3, to=2-3]
	\arrow["\lrcorner"{anchor=center, pos=0.125, rotate=180}, draw=none, from=2-2, to=1-1]
	\arrow["\lrcorner"{anchor=center, pos=0.125, rotate=180}, draw=none, from=2-3, to=1-2]
\end{tikzcd}\]
\end{construction}

The morphism $C_\infty\to D_\infty$ will be the map $f$ of \cref{prop_conterexemple}. The informal idea is that in $C_\infty$ the $1$-arrow corresponding to the vertical map $\Db_1 \to C_\infty$ has ``coinductive inverse up to height $n$'' for all $n$, but is not coinductively invertible. So when $C_\infty$ is seen as an object of the canonical (or coinductive) model structure this $1$-arrow is not invertible, but as soon as we localize to make all the $n$-arrows invertible for some integer $n$, then this $1$-arrow will become invertible. In contrast in $D_\infty$ this arrow becomes an identity, so it is invertible from the start. In the rest of the section, we will justify this rigorously.

\bigskip

We begin by showing the first point of \cref{prop_conterexemple}, namely that $C_\infty\to D_\infty$ is not a weak equivalence in the coinductive left semi-model structure.

\begin{lemma}
\label{lemma:technical lemma}
Let $P$ be a polygraph and $f$ a coinductively invertible $k$-arrow in $P$. For every $k$-generator $g$ appearing in the decomposition of $f$, there exists a sequence of generating arrows $(g_n)_{n\in\mathbb{N}}$ such that 
\begin{enumerate}
\item for $n>0$, $g_n$ is a $(n+k)$-generator and $g_0 = g$,
\item for $n>0$, $g_n$ appears in the decomposition of the source of $g_{n+1}$.
\end{enumerate}
\end{lemma}
\begin{proof}
We show this result by coinduction on $k$. Suppose the result is true for all $(k+1)$-arrows, and let $f:a\to b$ be a coinductively invertible $k$-arrow, and $g$ a $k$-generator appearing in the decomposition of $f$. There exists a $k$-arrow $f':b\to a$ and a coinductively invertible $(n+1)$-arrow $\alpha:f\#_{k-1}f'\to \Ib_a$. As $g$ is a $k$-generator appearing in the decomposition of $f\#_{k-1}f'$ (which is the source of $\alpha$), we can find a $(k+1)$-generator $\beta$ appearing in the decomposition of $\alpha$ and such that $g$ is in the decomposition of the source of $\beta$. As $\alpha$ is coinductively invertible, one can continue this process coinductively starting from $\beta$ to build a sequence of generators $(\beta_n)_{n\in\mathbb{N}}$ satisfying the desired property. We then set $g_0:=g$, and $g_n:=\beta_{n-1}$. This sequence also satisfies the desired property.
\end{proof}
\begin{cor}
  The $\infty$-categories $C_\infty$ and $D_\infty$ have no coinductively invertible arrows except identities.
\end{cor}

\begin{proof}
  We will show this assertion for $C_\infty$; the proof for $D_\infty$ is essentially the same. We proceed by contradiction: let $f$ be a non-identity coinductively invertible $k$-arrow of $C_\infty$. As $f$ is not an identity, there must be at least one $k$-generator $g$ appearing in its decomposition. Since $C_\infty$ is a polygraph, one can apply \cref{lemma:technical lemma} to obtain a sequence $(g_m)_{m \in \mathbb{N}}$ of generators of $C_\infty$. Eventually shifting the sequence, one can freely assume that $g_0$ is of dimension strictly greater than $1$. The generators of $C_\infty$ are obtained by gluing the generators of $P_n$ for all $n$ at the unique generator of $\Db_1$, so this $g_0$ must be in one of the $P_n$. It then follows by induction that all the $g_m$ are in the same $P_n$, but this leads to a contradiction as the dimension of the generators of $P_n$ is bounded above.
\end{proof}

\begin{cor}
The marked $\infty$-categories $C_\infty^\flat$ and $D_\infty^\flat$ are fibrant in the coinductive left semi-model structure.
\end{cor}

\begin{proof}
It is immediate that $C_\infty^\flat$ and $D_\infty^\flat$ fulfills the conditions of  \cref{def:prefibrant} and hence they are fibrant in the inductive left semi-model structure by \cref{char_fibrant_obj}. Hence, by \cref{prop:if_X_fibrant_marked_arrow_correspond_to_invertible_one}, we only need to check that all their coinductively invertible arrows are marked. By the previous corollary, only their identity arrows are coinductively invertible, which concludes the proof.
\end{proof}

\begin{lemma}
\label{lem:conterexemple_1}
The morphism $C_\infty \to D_\infty$ is not a weak equivalence in $\micat[\infty]_{\coind}$.
\end{lemma}

\begin{proof}
As both $C_\infty$ and $D_\infty$ are fibrant in the coinductive left semi-model structure, which is a Bousfield localization of the inductive left semi-model structure, this map is a coinductive equivalence if and only if it is an inductive equivalence. Hence, one can test whether it is an equivalence using \cref{def:equivalence} and \cref{char_equiv}, but this map fails to satisfy condition (1) of \cref{def:equivalence}, as the $1$-arrow of $C_\infty$ corresponding to the vertical map $\Db_1 \to C_\infty$ is not marked and maps to an identity arrow (hence marked) in $D_\infty$.
\end{proof}

Let us now show the second point, namely that for any integer $n$, $\pi_n C_\infty \to \pi_n D_\infty$ is a weak equivalence of $\micat[n]_{\satind}$.

\change{
\begin{lemma}
\label{lemma:conterexample 1.5}
For any $n>0$, the map $(\Db_n,\overline{\{e_n\}})\to \pi_{n-1} E_n$ is an acyclic cofibration of $\micat[\infty]_{\satind}$.
\end{lemma}

\begin{proof}
This map is the composition of pushouts along the equations $\big(\leftdivision{\textbf{eq}}{n,n}\big)^{d_{n+1}}$,  $\big(\leftdivision{\textbf{eq}}{n,n}\big)^{d_{n+1}’}$ and the saturations $\big(\leftdivision{\textbf{sat}}{n,n}\big)^{d_{n+1}}$,  $\big(\leftdivision{\textbf{sat}}{n,n}\big)^{d_{n+1}’}$, where $(\uvar)^{d_n}$ is the duality that inverts the direction of $(n+1)$-arrows, and $(\uvar)^{d’_{n+1}}$ is the duality that inverts the direction of both $n$-arrows and $(n+1)$-arrows. By \cref{cor:Equations and saturations are acyclic cofibrations}, this concludes the proof.
\end{proof}
}

\begin{lemma}
\label{lemma:conterexemple_2}
For any $n>0$, the map $\pi_{n+1} E_{n+1} \to \pi_n E_{n+1}$ is an acyclic cofibration in $\micat[\infty]_{\satind}$.
\end{lemma}
\begin{proof}
One should first note that this map is an isomorphism of the underlying $\infty$-categories and corresponds to marking all the $n$-arrows. In particular, it is a cofibration. Moreover, $\pi_{n+1} E_{n+1}$ is cofibrant as its underlying $\infty$-category is a polygraph. Using the characterization of fibrant objects in the saturated inductive left semi-model structure (see \cref{lem:2_out_of_6} and \cref{theo:saturated_inductive_model_structure}), one easily sees that fibrant objects have the unique left lifting property against $\pi_{n+1} E_{n+1} \to \pi_n E_{n+1}$.

\change{The class of morphisms having the unique left lifting property against this map then contains every morphism $C \to 1$ where $C$ is fibrant. As this class is closed under left cancellation, it includes any map between fibrant objects, and so in particular, any fibration between fibrant objects.
It follows that $\pi_{n+1} E_{n+1} \to \pi_n E_{n+1}$ is an acyclic cofibration.}
\end{proof}

\begin{lemma}
\label{lemma:piP to D0 is a we}
For all $n$, $\pi_n P_n \to \Db_0$ is a weak equivalence in $\micat[\infty]_{\satind}$.
\end{lemma}
\begin{proof}
We will proceed by induction. The case $n=0$ is obvious. Suppose proven that $\pi_{n} P_{n} \to \Db_0$ is a weak equivalence. 
We define $\tilde{P}_{n+1}$ as the pushouts:
\[\begin{tikzcd}
	{\coprod_{(P_n)_{n+1}} \Db_{n+1}} & {P_n} \\
	{\coprod_{(P_n)_{n+1}} \pi_{n} E_{n+1}} & {\tilde{P}_{n+1}}
	\arrow[from=1-1, to=1-2]
	\arrow[from=2-1, to=2-2]
	\arrow[from=1-1, to=2-1]
	\arrow[from=1-2, to=2-2]
	\arrow["\lrcorner"{anchor=center, pos=0.125, rotate=180}, draw=none, from=2-2, to=1-1]
\end{tikzcd}\]
By \cite[Corollary 2.4.4]{henry2020weak}, and Lemmas \ref{lemma:conterexemple_2} and \ref{lemma:conterexample 1.5}, all morphisms labeled by $\sim$ in the following diagrams are acyclic cofibrations, and hence weak equivalences:
\[\begin{tikzcd}[ampersand replacement=\&]
	{\coprod_{(P_n)_{n+1}} \Db_{n+1}} \& {P_n} \& {\coprod_{(P_n)_{n+1}} \Db_{n+1}} \& {P_{n+1}} \\
	{\coprod_{(P_n)_{n+1}} \pi_{n+1} E_{n+1}} \& {\pi_{n+1} P_{n+1}} \& {\coprod_{(P_n)_{n+1}} (\Db_{n+1}, \overline{\{e_n\}})} \& {\pi_{n} P_{n}} \\
	{\coprod_{(P_n)_{n+1}} \pi_{n} E_{n+1}} \& {\tilde{P}_{n+1}} \& {\coprod_{(P_n)_{n+1}} \pi_{n} E_{n+1}} \& {\tilde{P}_{n+1}}
	\arrow[from=2-1, to=2-2]
	\arrow[from=1-2, to=2-2]
	\arrow[from=1-1, to=2-1]
	\arrow["\sim", from=2-4, to=3-4]
	\arrow[from=1-1, to=1-2]
	\arrow["\sim"', from=2-1, to=3-1]
	\arrow["\sim", from=2-2, to=3-2]
	\arrow[from=3-1, to=3-2]
	\arrow[from=1-3, to=2-3]
	\arrow["\sim"', from=2-3, to=3-3]
	\arrow[from=1-3, to=1-4]
	\arrow[from=2-3, to=2-4]
	\arrow[from=1-4, to=2-4]
	\arrow["\lrcorner"{anchor=center, pos=0.125, rotate=180}, draw=none, from=3-2, to=2-1]
	\arrow["\lrcorner"{anchor=center, pos=0.125, rotate=180}, draw=none, from=2-2, to=1-1]
	\arrow["\lrcorner"{anchor=center, pos=0.125, rotate=180}, draw=none, from=2-4, to=1-3]
	\arrow["\lrcorner"{anchor=center, pos=0.125, rotate=180}, draw=none, from=3-4, to=2-3]
	\arrow[from=3-3, to=3-4]
\end{tikzcd}\]
By two out of three, and using the assumption that $\pi_{n} P_{n} \to \Db_0$  is a weak equivalence, the map $\tilde{P}_{n+1} \to \Db_0$ is a weak equivalence, and by stability by composition, so is the map  $\pi_{n+1} P_{n+1} \to \Db_0$.
\end{proof}

\begin{lemma}
\label{lemma:pi_n_C_equal_pi_n_D}
For all $n$, the induced morphism $\pi_n C_\infty \to \pi_n D_\infty$ is a weak equivalence in $\micat[n]_{\satind}$
\end{lemma}
\begin{proof}
By \cref{prop:iota is homotopically ff}, it is sufficient to show that $\pi_n C_\infty \to \pi_n D_\infty$ is a weak equivalence in $\micat[\infty]_{\satind}$.

Using \cref{lemma:piP to D0 is a we} and since weak equivalences between cofibrant objects are stable by pushout, we have a diagram where all morphisms labeled by $\sim$ are weak equivalences:
\[\begin{tikzcd}
	{\coprod_{k\in\mathbb{N}}\pi_n\Db_1} & {\coprod_{k\in\mathbb{N}}\pi_n{P_k}} & {(\coprod_{k< n}P_k) \coprod (\coprod_{k\geq n}\Db_0)} \\
	{\pi_n\Db_1} & {\pi_nC_{\infty}} & {D_n} \\
	{\Db_0} & {\pi_n D_{\infty}} & {D_n}
	\arrow[from=1-1, to=2-1]
	\arrow[from=2-1, to=3-1]
	\arrow[from=3-1, to=3-2]
	\arrow[from=2-1, to=2-2]
	\arrow[from=1-1, to=1-2]
	\arrow[from=1-2, to=2-2]
	\arrow[from=1-3, to=2-3]
	\arrow[""{name=0, anchor=center, inner sep=0}, "\sim", from=2-2, to=2-3]
	\arrow[""{name=1, anchor=center, inner sep=0}, "\sim", from=1-2, to=1-3]
	\arrow["\sim", from=3-2, to=3-3]
	\arrow[Rightarrow, no head, from=2-3, to=3-3]
	\arrow["\lrcorner"{anchor=center, pos=0.125, rotate=180}, draw=none, from=2-2, to=1-1]
	\arrow[from=2-2, to=3-2]
	\arrow["\lrcorner"{anchor=center, pos=0.125, rotate=180}, draw=none, from=3-2, to=2-1]
	\arrow["\lrcorner"{anchor=center, pos=0.125, rotate=180}, draw=none, from=2-3, to=1-2]
	\arrow["\lrcorner"{anchor=center, pos=0.125, rotate=180}, draw=none, from=3-3, to=1-3]
\end{tikzcd}\]
By two out of three, this shows the result.
\end{proof}

\begin{proof}[Proof of \cref{prop_conterexemple}]
We choose $f$ to be the morphism $C^\flat_\infty\to D^\flat_\infty$. The first point follows from \cref{lem:conterexemple_1} and the second from \cref{lemma:pi_n_C_equal_pi_n_D}.
\end{proof}

\subsection{Complicial Sets and Stratified Street Nerve}

\label{subsec:complicial}

In this section, we show that the Street nerve can be made into a right Quillen functor from the saturated inductive left semi-model structure on $\micat[\infty]$ to the Ozornova-Rovelli-Verity model structure for complicial sets. We refer to \cite{verity2008weak} and \cite{riehl2018complicial} for a detailed introduction to complicial sets; we will simply recall the important definitions below.

\begin{definition}
A \emph{stratified simplicial set} is a simplicial set $X$, together with a set $M \subset \coprod_{k > 0} X_n$ of simplices of positive dimension called \emph{thin} simplices that includes all degenerate simplices.

A morphism of stratified simplicial sets is a morphism between the underlying simplicial sets that sends thin simplices to thin simplices. The category of stratified simplicial sets is denoted $\strat$.
\end{definition}

The \textit{join} is an important operation for simplicial sets, which is defined on representables by the formula 
\[
\Delta[n] \star \Delta[m] := \Delta[n+m+1].
\]
We can extend it to any pair of simplicial sets by setting
\[
X \star Y := \colim_{\Delta^+_{/X}\times \Delta^+_{/Y}} \Delta[n] \star \Delta[m]
\]
\change{
where $\Delta^+$ is the augmented simplex category whose objects are possibly empty finite ordered sets and where we set the convention $$\Delta[n] \star \Delta[-1] := \Delta[n] =: \Delta[n-1] \star \Delta[n].$$
The set of $n$-simplices of $X \star Y$ is then in bijection with the set}
$$\{x\star y, (x,y)\in \coprod_{k < n} X_k \times Y_{n-k-1}\}\cup \{x\star \emptyset, x\in X_n\}\cup \{\emptyset\star y, y\in Y_n\}$$
See, for example, \cite[Definition 1.2.8.1]{lurie2009higher} and below. We now define it for stratified simplicial sets as follows:

\begin{definition}
If $(X,M)$ and $(Y,N)$ are two stratified simplicial sets, we define $M \star N$ as the set of simplices of $X \star Y$ of the form $x \star y$ where either $x$ or $y$ is thin, with the convention that $\emptyset$ is not thin. We then define
\[
(X,M) \star (Y,N) := (X \star Y, M \star N),
\]
\end{definition}

\begin{definition}
We define several marked simplicial sets whose underlying simplicial set is $\Delta[n]$:
\begin{enumerate}
\item $\Delta[n]$, where \change{degenerate simplices are thin}.
\item $\Delta[n]_t$, where the top $n$-simplex is thin.
\item $\Delta^k[n]$, where all simplices that include $\{k-1,k,k+1\} \cap [n]$  and \change{degenerate simplices are thin}.
\item $(\Delta^k[n])'$, where all simplices that include $\{k-1,k,k+1\} \cap [n]$, together with the $(k-1)$-face and the $(k+1)$-face, and all the degenerate simplices are thin.
\item $(\Delta^k[n])''$, where all simplices that include $\{k-1,k,k+1\} \cap [n]$, together with the $(k-1)$-face, the $k$-face, and the $(k+1)$-face are thin and all the degenerate simplices are thin.
\item $\Delta[3]^{eq}$, where simplices of dimension strictly higher than $2$, together with $[0,2]$ and $[1,3]$, and all degenerate simplices are thin.
\item $\Delta[n]^\sharp$, where all simplices are thin.
\end{enumerate}
Eventually, we will consider $\Lambda^k[n]$ endowed with the following marking: a simplex is marked in $\Lambda^k[n]$ if and only if it is marked in $\Delta^k[n]$.
\end{definition}

\begin{definition}[{\cite[Definition 1.19]{ozornova2020model}}]
An \emph{elementary anodyne extension} is one of the following:
\begin{enumerate}
\item The \emph{complicial horn inclusions:}
\[
\Lambda^k[n] \to \Delta^k[n], \quad n \geq 1, \quad n \geq k \geq 0.
\]
\item The \emph{complicial thinness extensions:}
\[
(\Delta^k[n])' \to (\Delta^k[n])'', \quad n \geq 2, \quad n \geq k \geq 0.
\]
\item The \emph{saturation extensions:}
\[
\Delta[n] \star \Delta[3]^{eq} \to \Delta[n] \star \Delta[3]^{\sharp}, \quad n \geq -1.
\]
\item The \emph{$m$-triviality extensions:}
\[
\Delta[n] \to \Delta[n]_t, \quad n > m.
\]
\end{enumerate}
\end{definition}

\begin{remark}
In the case where $m = \infty$, there is no $m$-triviality extension.
\end{remark}

\begin{definition}
An \emph{$m$-complicial set} is a marked simplicial set having the right lifting property against all elementary anodyne extensions.
\end{definition}

\begin{remark}
In the definition of complicial sets, we have included the "saturation extension" as part of our elementary anodyne extensions. These are not always included and play a role similar to the saturated localization of the inductive left semi-model structure considered in \cref{subsec:saturation}. See also \cite{riehl2018complicial} for a more general discussion of saturation for complicial sets.
\end{remark}

As demonstrated in \cite{loubaton2022n}, $m$-complicial sets are a model for $(\infty, m)$-categories. For example, $0$-complicial sets and $1$-complicial sets are essentially the same as Kan complexes and quasicategories, respectively.

\begin{theorem}[Verity \cite{verity2008weak}, Riehl \cite{riehl2018complicial}, Ozornova-Rovelli \cite{ozornova2020model}]
There is a model structure on $\strat$, \change{which we will call the \emph{Verity} model structure,} where the cofibrations are all monomorphisms, and the acyclic cofibrations are generated by elementary anodyne extensions. Fibrant objects of this model structure are the $m$-complicial sets. We denote $\mstrat_V$ the category $\strat$ endowed with this model structure.
\end{theorem}

We will use the join to define the adjunction between stratified simplicial sets and marked $\infty$-categories.

\begin{definition}
 Let $(C,M)$ and $(D,N)$ be two marked $\infty$-categories. The \emph{join of $(C,M)$ and $(D,N)$}, denoted $(C,M) \star (D,N)$, is the colimit of the following diagram:
\[\begin{tikzcd}[ampersand replacement=\&]
	{C \ltens \{0\} \ltens D \coprod C \ltens \{1\} \ltens D} \& {C \ltens \Db_1 \ltens D} \\
	{C \coprod B} \& {C \star B}
	\arrow[from=1-1, to=2-1]
	\arrow[from=1-2, to=2-2]
	\arrow[""{name=0, anchor=center, inner sep=0}, from=1-1, to=1-2]
	\arrow[from=2-1, to=2-2]
	\arrow["\lrcorner"{anchor=center, pos=0.125, rotate=180}, draw=none, from=2-2, to=0]
\end{tikzcd}\]
As noted in Proposition 3.3.11 of \cite{arapreparation} at the level of $\infty$-categories, this is the usual join of $\infty$-categories, as defined in Paragraph 6.30 of \cite{ara2016join}.
\change{By the definition of the operation $\ltens$, we then have $(C,M) \star (D,N) \cong (C \star D, \overline{M \star N})$, where 
$$M \star N := \{x \star y \mid x \in M, y \in N\} \cup \{x \star \emptyset \mid x \in M\} \cup \{\emptyset \star y \mid y \in N\}.$$}
\end{definition}

\begin{prop}
\label{prop:acyclic_cofibration_stable_by_joint}
Let $X \to Y$ be a cofibration and $K \to L$ an acyclic cofibration of $\micat[\infty]_{\satind}$. The morphisms
\[K \star Y \coprod_{X \star K} L \star X \to L \star Y \quad \text{and} \quad Y \star K \coprod_{K \star X} X \star L \to Y \star L\]
are acyclic cofibrations of $\micat[\infty]_{\satind}$.
\end{prop}
\begin{proof}
\change{
By construction, we have a cocartesian square
\[\begin{tikzcd}[ampersand replacement=\&]
	{K \ltens \Db_1 \ltens Y \cup L \ltens \partial \Db_1 \ltens Y \cup L \ltens \Db_1 \ltens X} \& {K \star Y \coprod_{X \star K} L \star X} \\
	{L \ltens \Db_1 \ltens Y} \& {L \star Y}
	\arrow[from=1-1, to=2-1]
	\arrow[""{name=0, anchor=center, inner sep=0}, from=1-1, to=1-2]
	\arrow[from=2-1, to=2-2]
	\arrow[from=1-2, to=2-2]
	\arrow["\lrcorner"{anchor=center, pos=0.125, rotate=180}, draw=none, from=2-2, to=0]
\end{tikzcd}\]
By \cref{lem:acyclic_cof_left_pushout_product}, the left-hand vertical morphism is an acyclic cofibration, and so is the right one. We proceed analogously for the second morphism.}
\end{proof}

\begin{definition}
The terminal category $1$ has a monoid structure for this join operation. The multiplication $1 \star 1 \to 1$ is the unique morphism to the terminal $\infty$-category.

By the universal property of the category $\Delta$, this induces a cosimplicial object $|\uvar| : \Delta \to \micat[\infty]$ where
\[|\Delta[n]| := 1 \star 1 \star \ldots \star 1.\]
The $\omega$-category $|\Delta[n]|$ is traditionally called the $n^{th}$ oriental.
We denote $|\uvar| : \textbf{Sset} \to \micat[\infty]$ the extension by colimits of this cosimplicial object.

For all $n$, $|\Delta[n]|$ is an $n$-polygraph that admits only one $n$-generator.
If $M$ is a marking for $K$, we denote $|M|$ the set of arrows obtained as composition:
\[\Db_n \to \Delta[n] \xrightarrow{|v|} K\]
where the left morphism corresponds to the top arrow of the $n^{th}$ orientals, and the right morphism is in $M$.
We can now extend the strictification functor to stratified simplicial sets:
\[\begin{array}{rclr}
|\uvar| :& \strat &\to& \micat \\
&(K,M) &\mapsto & (|K|, \overline{|M|})
\end{array}\]
This functor is cocontinuous and induces an adjunction:
\[\begin{tikzcd}
    \strat && \micat
    \arrow[""{name=0, anchor=center, inner sep=0}, "{|\_|}", curve={height=-12pt}, from=1-1, to=1-3]
    \arrow[""{name=1, anchor=center, inner sep=0}, "N", curve={height=-12pt}, from=1-3, to=1-1]
    \arrow["\dashv"{anchor=center, rotate=-90}, draw=none, from=0, to=1]
\end{tikzcd}\]
The right adjoint is called the \emph{stratified Street nerve}. 
\end{definition}

\begin{remark}
In the case $m=\infty$, this adjunction models the forgetful functor from strict $\infty$-categories to weak $(\infty,\infty)$-categories (given by the stratified Street nerve $N$). The left adjoint corresponds to the ``strictification functor'' that sends a weak $(\infty,\infty)$-category to a strict $\infty$-category in a universal way.
\end{remark}

\begin{prop}
\label{prop:nerve of fibrant object}
The stratified Street nerve sends fibrant objects of $\micat_{\satind}$ to fibrant objects of $\mstrat_V$.
\end{prop}
\begin{proof}
Suppose first that $m < \infty$ and let $(X,M)$ be a fibrant $m$-marked $\infty$-category for the saturated inductive left semi-model structure. According to \cref{cor:link_beetwen_m_saturated_inductive_and_folk}, $M$ consists of coinductively invertible arrows of $X$, and $N((X,M))$ is equal to the stratified simplicial set associated with the Street nerve of $X$ defined in \cite[Définition 5.2.1]{loubaton2021conditions}. Theorem 5.2.12 of \textit{op. cit.} then implies that the stratified Street nerve sends fibrant objects of the saturated inductive left semi-model structure on $\micat[m]$ to $m$-complicial sets.

Now, let $C$ be a fibrant $\infty$-marked $\infty$-category for the saturated inductive left semi-model structure. As the stratified {Street} nerve preserves directed colimits, there is an isomorphism
\[N(C) \cong \colim_{n \in \mathbb{N}} N(\tau_n C)\]
For all $n$, $\tau_n C$ is fibrant for the saturated inductive left semi-model structure for $n$-marked $\infty$-categories, and $N(\tau_n C)$ is then a fibrant object of the model structure for $n$-complicial sets. As the model structure for $\infty$-complicial sets is $\omega$-combinatorial, fibrant objects are stable under directed colimits, and $N(C)$ is fibrant.
\end{proof}

\change{
\begin{lemma}
\label{lemma:strictification_computes_with_join}
Let $(K,M)$ be a stratified simplicial set and $L$ a simplicial set. We denote $N$ the set of degenerate simplices of $L$. There exists an isomorphism
$$|(K,M) \star (L,N)| \cong |(K,M)| \star |(L,N)|$$
natural in $K$ and $L$.
\end{lemma}
\begin{proof}
Proposition 7.13 of \cite{ara2016join} provides a natural isomorphism $|K \star L| \cong |K| \star |L|$. Moreover, \cref{lem:compatibility_of_coproduct_and_saturation} implies that $\overline{\overline{|M|} \star \overline{|N|}} = \overline{|M| \star |N|}$. Since we have $|(K,M) \star (L,N)| \cong (|K \star L|, \overline{|M \star N|})$ and $(|K| \star |L|, \overline{\overline{|M|} \star \overline{|N|}})$, this concludes the proof.
\end{proof}
}

\begin{lemma}
\label{lemma:realization_send_horn_on_ac}
The strictification functor sends complicial horn inclusions to acyclic cofibrations of the saturated inductive left semi-model structure for $m$-marked $\infty$-categories.
\end{lemma}
\begin{proof}
The morphism $|\Lambda^1[2]| \to |\Delta[2]^1|$ corresponds to the following inclusion of marked $\infty$-categories:
\[\begin{tikzcd}
    & \bullet &&& \bullet \\
    \bullet && \bullet & \bullet && \bullet
    \arrow[from=2-1, to=1-2]
    \arrow[""{name=0, anchor=center, inner sep=0}, from=1-2, to=2-3]
    \arrow[""{name=1, anchor=center, inner sep=0}, from=2-4, to=1-5]
    \arrow[from=1-5, to=2-6]
    \arrow[""{name=2, anchor=center, inner sep=0}, from=2-4, to=2-6]
    \arrow["\sim"', shorten <=3pt, Rightarrow, from=2, to=1-5]
    \arrow[shorten <=13pt, shorten >=13pt, hook, from=0, to=1]
\end{tikzcd}\]
which is obviously an equation. The two morphisms $|\Lambda^0[2]| \to |\Delta^0[2]|$ and $|\Lambda^2[2]| \to |\Delta^2[2]|$ are respectively equal to $\leftdivision{\textbf{eq}}{1,1}$ and $\rightdivision{\textbf{eq}}{1,1}$. Furthermore, we can see that for all $0 < k < n$, we have:
\[\Delta^k[n] = \Delta[k-2] \star \Delta^1[2] \star \Delta[n-k-2] \]
and
$\Lambda^k[n]$ is the sub-object:
\[
\begin{array}{rl}
&\partial \Delta[k-2] \star \Delta^1[2] \star \Delta[n-k-2]\\
\cup& \Delta[k-2] \star \Lambda^1[2] \star \Delta[n-k-2]\\
\cup& \Delta[k-2] \star \Delta^1[2] \star \partial \Delta[n-k-2].
\end{array}
\]
By \cref{lemma:strictification_computes_with_join}, the strictification functor commutes with the join.
\cref{prop:acyclic_cofibration_stable_by_joint} then implies that $|\Lambda^k[n]| \to |\Delta^k[n]|$ is an acyclic cofibration. We proceed analogously for the cases $k=0$ and $k=n$.
\end{proof}

\begin{theorem}
The strictification functor and the stratified Street nerve form a Quillen adjunction between the model structure for $m$-complicial sets and the saturated inductive left semi-model structure on $\micat[m]$.
\end{theorem}
\begin{proof}
Because of \cref{lemma:realization_send_horn_on_ac}, it remains to show that complicial thinness extensions, saturation extensions, and $m$-triviality extensions are sent to acyclic cofibrations.
Let $i$ be such a morphism. According to \cref{prop:nerve of fibrant object}, any fibrant object of the saturated inductive left semi-model structure has the right lifting property against $|i|$. As $|i|$ is an identity on the underlying $\infty$-category, lifts against it are unique if they exist. This implies that any morphism between fibrant objects has the right lifting property against $|i|$, and this morphism is then an acyclic cofibration. This concludes the proof.
\end{proof}

We can use this to generalize the results from \cite{loubaton2021conditions}: The stratified Street nerve:
\[ \Ncal : \icat \to \mstrat \]
introduced in \cite{loubaton2021conditions}, is exactly the stratified Street nerve $N$ of the present paper combined with the fully faithful inclusion $\icat \subset \micat$ constructed in \cref{subsec:Folk}, which makes all coinductively invertible arrows marked. Hence:

\begin{prop}\label{prop:Marked_nerve_pres_fib}
 Let $f: X \to Y$ be a fibration (resp. an acyclic fibration, resp. a weak equivalence) of the canonical model structure 
 $\icat_\can$, then its stratified Street nerve $\Ncal(f) : \Ncal(X) \to \Ncal(Y)$ is a fibration (resp. an acyclic fibration, resp. a weak equivalence) in the Verity model structure $\change{\mstrat_V}$.
\end{prop}

The main result of \cite{loubaton2021conditions} corresponds to the special case of preservation of fibrant objects.

Note that, in particular, the proposition shows that the stratified Street nerve from \cite{loubaton2021conditions}, while not being a right Quillen functor, is still a morphism of Brown categories of fibrant objects (\cite{brown1973abstract}), and so it does define a limit-preserving functor on the corresponding associated $(\infty,1)$-categories. 
\begin{proof}
 As the stratified Street nerve $N : \micat_\satind \to \mstrat_V$ is a right Quillen functor, it preserves fibrations and acyclic fibrations, as well as weak equivalences between fibrant objects. Moreover, we have shown in \cref{theo;equivalenc eentre can et coind2} that the functor sending a strict $\infty$-category to the marked one where the marked arrows are the coinductively invertible ones, preserves fibrations, acyclic fibrations, and weak equivalences.
\end{proof}

\change{Finally, we want to clarify that this ``forgetful'' functor from strict $\infty$-categories to weak $\infty$-categories is not an ``inclusion'' in the sense that it is not fully faithful in any reasonable homotopy-theoretic sense. That is, this functor is not just an inclusion of strict $(\infty,m)$-categories into weak $\infty$-categories, but exhibits strict $\infty$-categories as objects with more structure than weak $\infty$-categories, even though the exact nature of this additional structure is not quite clear, beyond some special cases. Below we reproduce an example, mostly due to Dimitri Ara in \cite{ara2019quillen}, and just slightly adjusted to our setting, showing that this functor is not fully faithful on the homotopy category:}

\begin{example} For $M$ a commutative monoid, we write $B^n M$ for the strict $\infty$-category with only one cell in dimension $k <n$, the monoid $M$ of cell in dimension $n$ and only identity arrows in dimension $k>n$. All composition operations are given by the operation of $M$. We then claim that

\[ [ B^2 \mathbb{N}, B^4 \mathbb{Z} ]_{\micat[m]_{\satind}} \simeq \{0\} \qquad [ N(B^2 \mathbb{N}), N(B^4 \mathbb{Z}) ]_{\strat_V} \simeq \mathbb{Z} \]

Where $[ \_, \_ ]$ denote the set of homotopy class of maps (i.e. morphism in the homotopy category). We have not specified the marking, but these results do not depend on markings as long as the nerve $N$ is applied to a fibrant replacement (as it is a right Quillen functor).

First, we need to notice that $ B^2 \mathbb{N}$ is the $\infty$-category freely generated by one object $*$ and a $2$-cell whose source and target are the identity of $*$. In particular it is a polygraph, and hence is a cofibrant object.

On the other hand $B^4 \mathbb{Z}$ is a strict $\infty$-groupoid (every cell is strictly invertible) so, whatever marking we start from on $B^4 \mathbb{Z}$, a fibrant replacement will just be $B^4(\mathbb{Z})^\sharp$ as marking cells that are invertible is an acyclic cofibration and once every cell is marked it is a fibrant object by (\cref{lem:2_out_of_6} and \cref{theo:saturated_inductive_model_structure}). So, the set of maps $[ B^2 \mathbb{N}, B^4 \mathbb{Z} ]_{\micat[m]_{\satind}}$ in the homotopy category can be computed as homotopy class of maps from  $B^2 \mathbb{N}$ to $B^4 \mathbb{Z}^\sharp$, but as there is no non-trivial $2$-cell in $B^4 \mathbb{Z}$ the only such maps in the constant maps equal to $0$. Hence
\[ [ B^2 \mathbb{N}, B^4 \mathbb{Z} ]_{\micat[m]_{\satind}} \simeq \{0\}. \]

We now move to the computation of the hom set in complicial sets. In order to apply the Street nerve in a homotopy relevant way, we need to take fibrant replacement of both object. For $B^4 \mathbb{Z}$ we discuss this above and it corresponds to take $B^4 \mathbb{Z}^\sharp$. For $B^2 \mathbb{N}$, as it has no non-identity invertible cells, $B^2 \mathbb{N}^\flat$ is already fibrant. In particular $N(B^4 \mathbb{Z})$ is a complicial set whose cells are all thin (marked). Hence the marking we put in $N(B^2 \mathbb{N})$ actually do not matter in the computation and

\[ [(B^2\mathbb{N})^\flat ,  (B^4 \mathbb{Z})^\sharp]_{\mstrat_V} = [( B^2\mathbb{N})^\sharp ,  (B^4\mathbb{Z})^\sharp]_{\mstrat_V}  \]

Hence we need to compute a set of homotopy class of maps between two complicial sets where every cell is marked - so this boils down to computing a set of homotopy class of maps in the Kan-Quillen model structure on simplicial set, using the unmarked Street nerve. We can now rely on two results from \cite{ara2019quillen} to show arrive at our result:

Theorem 4.7 of \cite{ara2019quillen} shows that for any group $G$, $N(B^n G)$ is an Eilenberg MacLane space $K(\pi,n)$. Theorem 4.9 (and especially example 4.10) shows that $N(B^2 \mathbb{N})$ is homotopically equivalent to $N(B^2 \mathbb{Z})$ and hence is also an Eilenberg MacLane $K(2,\mathbb{Z})$, so using the well known equivalence between simplicial sets and spaces, we can write 
\[ [ N(B^2 \mathbb{N}), N(B^4 \mathbb{Z}) ]_{\strat_V} \simeq [K(2,\mathbb{Z}), K(4, \mathbb{Z}) ]_{\text{Space}} \]
and for this final hon set we can use methods from topology: $K(2,\mathbb{Z})$ can be realized as $\mathbf{CP}^\infty$ and hence
\[ [K(2,\mathbb{Z}), K(4, \mathbb{Z}) ]_{\text{Space}}  = H^4( \mathbf{CP}^\infty ) = \mathbb{Z} \]
where these last claim can be found in many algebraic topology textbook, for example \cite{hatcher2002algebraic}.

\end{example}

\appendix

\section{Left Semi-model categories} \label{sec:semi-model-categories}
Semi-model categories were first introduced by Spitzweck in \cite{spitzweck2001operads}, following a remark by Hovey in \cite{hovey1998monoidal} that given a combinatorial symmetric monoidal model category $\Vcal$, the category of monoids in $\Vcal$ carries such a structure without assuming that $\Vcal$ satisfies the "monoid axiom." This observation is sufficient for studying the homotopy theory of monoids in $\Vcal$. A more general (but not equivalent) notion of semi-model structure was later introduced by Fresse in Section 12 of \cite{fresse2009modules}.

Contrary to what the name might suggest, a left semi-model category is not "half of a model category." It is a minor weakening of the definition of a Quillen model category that allows for nearly all standard homotopical constructions but is somewhat easier to define. This minor weakening often eliminates technical or unnatural assumptions in certain theorems, such as the monoid axiom mentioned above or the requirement of properness when constructing localizations (see \cref{theo:left Bousfield localization} below).

In brief, a left semi-model category is similar to a model category, but certain axioms, such as the lifting property and the existence of factorizations, are only required to hold for morphisms with cofibrant domains. Since any map can be replaced by an equivalent one with a cofibrant domain, and only maps between cofibrant and fibrant objects contribute directly to the homotopy theory, this restriction does not significantly alter the theory. The primary drawback of using left semi-model structures is practical: most of the literature focuses on Quillen model structures, so results must be re-proven for semi-model structures. A substantial body of work (see below) has been completed on this topic, and no serious difficulties have arisen so far.

In this paper, all Quillen model structures and left semi-model structures we encounter are "combinatorial" (in the sense of \cref{defi:basic def on left semi-model structure} below). In particular, they have fully formed weak factorization systems, rather than the weakened version assumed in \cite{spitzweck2001operads}, \cite{fresse2009modules}, or \cite{henry2020weak}. Assuming the existence of full factorization systems simplifies the definition, which we will adopt here. In \cite{henry2020combinatorial}, these are referred to as "factorization left semi-model categories," which is not the most general definition found in the literature.

\begin{definition}\label{def:premodel}
A \emph{premodel category} is a complete and cocomplete category $\Ccal$ equipped with two weak factorization systems: (\emph{anodyne cofibrations}, \emph{fibrations}) and (\emph{cofibrations}, \emph{anodyne fibrations}), where the anodyne cofibrations are also cofibrations, or equivalently, the anodyne fibrations are fibrations.
\end{definition}

\begin{definition}
An object $C$ is \textit{fibrant} if the map $C \to 1$ is a fibration. An object is \textit{cofibrant} if the map $\emptyset \to C$ is a cofibration.
\end{definition}

\begin{definition}\label{def:SemiModelCat}
A \emph{(Spitzweck factorization) left semi-model category} is a premodel category with a class $\Wcal$ of morphisms, called weak equivalences, satisfying the following conditions:
\begin{enumerate}
    \item The class $\Wcal$ contains all isomorphisms and satisfies the 2-out-of-3 property.
    \item A fibration is anodyne if and only if it is in $\Wcal$.
    \item A cofibration with a cofibrant domain is anodyne if and only if it is in $\Wcal$.
\end{enumerate}
\end{definition}

Note that if we remove the restriction "with cofibrant domain" in the third axiom, we recover the definition of a Quillen model structure. In the remainder of the paper, we will simply refer to these structures as left semi-model categories.

\begin{remark}
We should clarify the terminology here compared to what we used, for instance, in \cref{def:generating_anodyne}. Often, as in the present paper, we begin with a premodel category with two weak factorization systems (anodyne cofibrations, fibrations) and (cofibrations, anodyne fibrations) that does not itself form a left semi-model category. However, we use a "saturation" construction described in Section 4 of \cite{henry2020combinatorial}, which adjusts the weak factorization systems without altering the underlying category, the cofibrations with cofibrant domains, or the fibrations with fibrant domains. The resulting premodel category is a left semi-model category. These new factorization systems are typically called "trivial" or "acyclic" instead of "anodyne." See Sections 3 and 4 of \cite{henry2020combinatorial} for more details on this process.

In this paper, this distinction means that, contrary to what \cref{def:SemiModelCat} might suggest, \cref{th:canonical_left_semi_model} does not imply that a cofibration (with a cofibrant domain) that is an equivalence is an anodyne cofibration as defined in \cref{def:generating_anodyne}. Instead, the premodel structure that \cref{th:canonical_left_semi_model} asserts to be a left semi-model category involves weak factorization systems for (acyclic cofibrations, fibrations) and (cofibrations, acyclic fibrations). Therefore, a cofibration with a cofibrant domain is an equivalence if and only if it is an acyclic cofibration.
\end{remark}

In particular, the full subcategory of $\Ccal$ consisting of cofibrant objects forms a model category, except that it may not be closed under limits and colimits—hence the need to consider the non-cofibrant objects of $\Ccal$ as well.

The basic theory of left semi-model categories operates similarly to Quillen model categories: the homotopy category can be defined by formally inverting the maps in $\Wcal$ or by defining a homotopy relation between bifibrant objects. See \cite{spitzweck2001operads} or \cite{henry2020weak}\footnote{Semi-model categories are particular cases of weak model structures as defined in \cite{henry2020weak}, so the results from {this} work can be applied.}. The $\infty$-categorical localization is also considered in the appendices of \cite{lomonaco2022large} (under the assumption that the factorization systems are functorial, which will always be the case in this paper), and it functions similarly to the corresponding localization in Quillen model categories.

\begin{definition}
\label{defi:basic def on left semi-model structure}
\begin{itemize}
  \item[]
  \item A premodel category is said to be \emph{combinatorial} if its underlying category is locally presentable and both factorization systems are cofibrantly generated. It is said to be $\omega$-combinatorial if furthermore  the  underlying category is locally $\omega$-presentable and the codomains of the generating cofibrations and acyclic cofibrations are $\omega$-small.
  \item A \emph{Quillen adjunction} between premodel categories is an adjunction $L : \Ccal \leftrightarrows \Dcal : R$ such that $L$ sends cofibrations and anodyne cofibrations to cofibrations and anodyne cofibrations, or equivalently, such that $R$ sends fibrations and anodyne fibrations to fibrations and anodyne fibrations.
  \item A \emph{monoidal premodel category} is a premodel category $\Ccal$, endowed with a monoidal closed structure, such that the monoidal unit is cofibrant, and for each pair of cofibrations $i : A \to B$ and $j : C \to D$, the map
    \[
    i \widehat{\otimes} j : B \otimes C \coprod_{A \otimes C} A \otimes D  \to B \otimes D
    \]
    is also a cofibration. Moreover, if $i$ or $j$ is anodyne, then $i \widehat{\otimes} j$ is also anodyne.
\end{itemize}

A left semi-model category is said to be \emph{combinatorial} or \emph{monoidal} if its underlying category is, and an adjunction between left semi-model categories is said to be a \emph{Quillen adjunction} if it is a Quillen adjunction of the underlying premodel categories.
\end{definition}There are more general notions of monoidal structures or Quillen adjunctions for left semi-model structures that only involve the cofibrations between cofibrant objects, such as the "weak Quillen functors" discussed in \cite{henry2020weak}. However, we do not need these generalizations in the present paper.

Similarly to what happens with Quillen model categories, Quillen adjunctions between left semi-model categories induce adjunctions between their homotopy categories and even between their $\infty$-categorical localizations (see, for example, \cite{henry2020weak}).

\begin{definition}
  A Quillen adjunction $F:C \leftrightarrows D:R$ is a \emph{Quillen equivalence} if the induced adjunction between their homotopy categories is an equivalence of categories.
\end{definition}

Various equivalent characterizations of Quillen equivalences can be found in Proposition 2.4.5 of \cite{henry2020weak}.

The following result will be used to characterize weak equivalences, or at least the weak equivalences between fibrant objects, in various left semi-model structures:

\begin{prop}
\label{prop:weak equivalence has the homotopy right lifting property}
  Let $\Ccal$ be a left semi-model category, and let $f:X \to Y$ be a morphism between two fibrant objects. Then $f$ is a weak equivalence if and only if $f$ has the so-called "homotopy right lifting property" against all cofibrations between cofibrant objects. That is, for each cofibration $i:A \cto B$ with cofibrant domain in $\Ccal$ and any commutative square:
\[  \begin{tikzcd}
    A \ar[d,"i"] \ar[r] & X \ar[d,"f"] \\
    B \ar[r] & Y 
  \end{tikzcd} \]
  there exist dotted morphisms making the following diagram commute:
\[ \begin{tikzcd}
    A \ar[ddd,"i"] \ar[rr] \ar[dr] & & X \ar[ddd,"f"] \\
    & B  \ar[ur,dotted] \ar[d] &\\
    & I_A B \ar[dr,dotted] & \\
    B \ar[ur] \ar[rr] & & Y 
  \end{tikzcd}
\]
  where $I_A B$ is a relative cylinder object for $i$, that is, a middle object of some (cofibration, anodyne fibration) factorization of the codiagonal map of $i$:
\[ B \coprod_A B \cto I_A B \overset{\sim}{\fto} B \]
  Moreover, if $I$ is a generating set of cofibrations, then it is sufficient to check this for $i \in I$.
\end{prop}

This is well known for Quillen model categories and proved in the more general setting of weak model categories in Appendix A of \cite{henry2020weak} (see Remark A.2.7).

\bigskip

We will occasionally need to take left Bousfield localizations of left semi-model categories. This is actually easier than Bousfield localization of Quillen model categories as it no longer requires any properness assumptions. It was shown in \cite{batanin2024left} that left Bousfield localization of combinatorial left semi-model categories at a set of maps yields another left semi-model category. This result was later reproved and generalized in \cite{henry2020combinatorial} to include both left and right Bousfield localizations of combinatorial and accessible left semi-model categories, but we will only need the version from \cite{batanin2024left} here:

\begin{theorem}
\label{theo:left Bousfield localization}
  Let $\Ccal$ be a combinatorial left semi-model category, and let $S$ be a set of morphisms between cofibrant objects in $\Ccal$. Then there is another left semi-model category $\Ccal_S$, called the \emph{left Bousfield localization of $\Ccal$ at $S$}, with the same underlying category as $\Ccal$, such that:

  \begin{itemize}
  \item $\Ccal_S$ has the same cofibrations as $\Ccal$, and the identity functor $\Ccal \to \Ccal_S$ is a left Quillen functor.
  \item A left Quillen functor $\Ccal \to \Dcal$ to any other left semi-model structure is a left Quillen functor $\Ccal_S \to \Dcal$ if and only if it sends the morphisms in $S$ to weak equivalences.
  \end{itemize}
\end{theorem}

The fibrant objects of $\Ccal_S$ are the objects that are fibrant in $\Ccal$ and are "S-local". However, in order to define $S$-local objects, one needs to define mapping spaces. To avoid this, we provide the following characterization:

\begin{lemma}
\label{lemma:technical result on localizations}
  Let $\Ccal_S$ be a left Bousfield localization of $\Ccal$. Assume that all morphisms in $S$ are cofibrations between cofibrant objects (or have been replaced by equivalent cofibrations). For each cofibration $i:A \to B \in S$, let $\nabla i$ be a cofibration between cofibrant objects homotopy equivalent to the map $B \coprod_A B \to B$, for example a factorization
  \[ B \coprod_A B \overset{\nabla i}{\cto} I_A B \overset{\sim}{\fto} B\]
  and let $\nabla^k i$ be a series of cofibrations obtained by iterating this process, that is, $\nabla^k i = \nabla(\nabla^{k-1} i)$. Then an object is fibrant in $\Ccal_S$ if and only if it is fibrant in $\Ccal$ and has the right lifting property against $\nabla^k i$ for all $k$ and all $i \in S$.
\end{lemma}

\bigskip

Finally, we can form Reedy model structures in this context as well. This is very similar to the treatment of classical Reedy model structures (see, for example, Chapter 5.2 in \cite{hovey1999models}).

Given a Reedy category $R$ and $\mathcal{C}$ a premodel category, the category of functors $\mathcal{C}^R$ has a premodel structure whose (anodyne) fibrations and (anodyne) cofibrations are the Reedy (anodyne) fibrations and cofibrations. That is, a natural transformation $f_r:X(r) \to Y(r)$ in $\mathcal{C}^R$ is an (anodyne) cofibration if and only if for each $r \in R$ the natural map
\[ X(r) \coprod_{L_r X} L_r Y \to Y(r) \]
where
\[ L_r X = \Colim_{r' \to r \in R^+ \atop r' \neq r } X(r')\]
is an (anodyne) cofibration. Dually, this natural transformation is an (anodyne) fibration if the natural map
\[ X(r) \to Y(r) \times_{M_r Y} M_r X \]
where
\[ M_r X = \Lim_{r \to r' \in R^- \atop r' \neq r } X(r') \]
is an (anodyne) fibration. We have:

\begin{theorem}
  If $\mathcal{C}$ is a premodel category and $R$ is a Reedy category, then $\mathcal{C}^R$ is a premodel category with the class of maps described above.

  Furthermore, if $\mathcal{C}$ is a left semi-model category, then $\mathcal{C}^R$ is a left semi-model category with the weak equivalences being the levelwise weak equivalences.
\end{theorem}

A more detailed treatment of Reedy model structures in the context of weak model categories, with more detailed proofs, can be found in Appendix C.2 of \cite{Bardomiano2024homotopy}. Though most of it is devoted to dealing with weakened assumptions regarding the existence of limits and colimits, which are not relevant in the present context.

\begin{proof}
  The proof carries over essentially unchanged from the case of Quillen model structures. The proof that these form weak factorization systems on $\mathcal{C}^R$ as in Theorem 5.2.5 of \cite{hovey1999models} relies on the fact that we have weak factorization systems on $\mathcal{C}$ and hence carries over to the case of premodel categories. The other key argument can be found in the proof of Theorem 5.1.3 of \cite{hovey1999models} and shows that, because of the $2$-out-of-$3$ property for weak equivalences, a Reedy fibration is a Reedy anodyne fibration if and only if it is a weak equivalence, and by the exact same argument, a Reedy cofibration with cofibrant domain is a Reedy anodyne cofibration if and only if it is a levelwise equivalence.
\end{proof}

A lemma that plays a significant role in this proof and that we will use at some points is:

\begin{lemma}\label{lem:Reedy_colim}
  If $R$ is a direct category (that is, a Reedy category with $R = R^+$) and $A \to B$ is a Reedy (anodyne) cofibration in $\mathcal{C}^R$, then the comparison map
\[ \Colim_{r \in R} A(r) \to \Colim_{r \in R} B(r) \]
  is an (anodyne) cofibration.
\end{lemma}

\begin{proof}
  This is essentially Corollary 5.1.5 of \cite{hovey1999models}. The simplest way to prove it is to observe that the colimit functor is the left adjoint to the "constant" functor, and the constant functor clearly sends the fibrations and anodyne fibrations of $\mathcal{C}$ to Reedy fibrations and anodyne Reedy fibrations in $\mathcal{C}^R$, as Reedy fibrations for a direct category are just levelwise fibrations.
\end{proof}

\bibliography{Biblio}{}
\bibliographystyle{plain}
\end{document}